\documentclass[10pt]{amsart}
\usepackage[english]{babel}
\usepackage[utf8]{inputenc}
\usepackage[a4paper,total={160mm,247mm},left=25mm,top=25mm]{geometry}

\usepackage{amssymb,amsthm, amsfonts, amsmath}
\usepackage{mathtools, mathrsfs, stmaryrd, graphicx, upgreek,dsfont}
\usepackage{multirow} 
\usepackage{relsize}

\usepackage[bbgreekl]{mathbbol}
\DeclareSymbolFontAlphabet{\mathbb}{AMSb} 
\DeclareSymbolFontAlphabet{\mathbbl}{bbold}

\usepackage{enumerate} 
\usepackage[shortlabels]{enumitem}

\usepackage[all,cmtip]{xy} 

\usepackage[dvipsnames]{xcolor}
\usepackage[pagebackref]{hyperref} 
\hypersetup{
	colorlinks=true,
	menucolor=red,
	linkcolor=blue,
	citecolor=BrickRed,
	filecolor=red,      
	urlcolor=MidnightBlue,
	pdftitle={(Algebraic) p-adic Artin Formalism of Twisted Triple Product},
	pdfpagemode=UseNone,
	pdfstartview={XYZ null null 1.00}
}

\usepackage{url}
\usepackage[capitalise]{cleveref} 
\usepackage[autostyle]{csquotes}
\usepackage{dirtytalk}

\makeatletter			
\newcommand*\bigcdot{\mathpalette\bigcdot@{1.2}}
\newcommand*\bigcdot@[2]{\mathbin{\vcenter{\hbox{\scalebox{#2}{$\m@th#1\bullet$}}}}}
\makeatother

\usepackage{tikz-cd}
\tikzset{ symbol/.style={draw=none, every to/.append style={edge node={node [sloped, allow upside down, auto=false]{$#1$}}}}}

\numberwithin{equation}{section}

\usepackage{perpage}
\MakePerPage{footnote}

\makeatletter
\newcommand{\mylabel}[2]{#2\def\@currentlabel{#2}\label{#1}}
\makeatother

\newtheorem{introtheorem}{Theorem} 

\newtheorem{conj}{Conjecture}

\newtheorem{thm}{Theorem}[section]
\newtheorem{cor}[thm]{Corollary}
\newtheorem{prop}[thm]{Proposition}
\newtheorem{lemma}[thm]{Lemma}

\theoremstyle{definition}
\newtheorem{defn}[thm]{Definition}

\newtheorem{exmp}[thm]{Example}

\theoremstyle{remark} 
\newtheorem{rmk}[thm]{Remark}
\newtheorem*{acknowledgements}{Acknowledgments}

\newcommand{\bb}{\mathbb}
\newcommand{\bbl}{\mathbbl}
\newcommand{\bs}{\boldsymbol}
\newcommand{\mbf}{\mathbf}
\newcommand{\msf}{\mathsf}
\newcommand{\scr}{\mathscr}
\newcommand{\mrm}{\mathrm}
\newcommand{\cal}{\mathcal}

\newcommand{\diam}[1]{\left\langle#1\right\rangle}
\newcommand{\ba}{\backslash}

\newcommand{\abs}[1]{\left\lvert#1\right\rvert}
\newcommand{\bbox}[1]{\left\llbracket#1\right\rrbracket}

\newcommand{\ad}{\operatorname{ad}}
\newcommand{\As}{\operatorname{As}}
\DeclareMathOperator{\Char}{char}
\newcommand{\Exp}{\operatorname{Exp}}
\newcommand{\EXP}{\operatorname{EXP}}
\newcommand{\Ext}{\operatorname{Ext}}
\newcommand{\Fitt}{\operatorname{Fitt}}
\newcommand{\Fil}{\mathrm{Fil}}
\newcommand{\fr}{\mathrm{Fr}}
\newcommand{\Gal}{\operatorname{Gal}}
\newcommand{\GL}[1]{\operatorname{GL}_{#1}}	
\newcommand{\Hom}{\operatorname{Hom}}
\newcommand{\id}{\operatorname{id}}
\newcommand{\Ind}{\operatorname{Ind}}
\DeclareMathOperator{\im}{im}
\newcommand{\Log}{\operatorname{Log}}
\DeclareMathOperator{\LOG}{LOG}
\DeclareMathOperator{\ord}{ord}
\newcommand{\pan}{\mathrm{Pan}}
\DeclareMathOperator{\rank}{rank}
\DeclareMathOperator{\res}{res}
\DeclareMathOperator{\Tor}{Tor}
\DeclareMathOperator{\tr}{{tr}^*}

\renewcommand{\b}{\mathrm{bal}}
\newcommand{\f}{\mathbf{f}}
\newcommand{\scrF}{\mathscr{F}}
\newcommand{\g}{\mathbf{g}}
\newcommand{\G}{G_{F,\Sigma}}

\newcommand{\Oo}{\mathcal{O}}
\newcommand{\p}{\mathfrak{p}}
\newcommand{\bP}{\mathbbl{P}}
\newcommand{\fp}{\mathfrak{P}}

\newcommand{\Q}{\mathbb{Q}}

\newcommand{\R}{\mathcal{R}}
\newcommand{\T}{\mathbf{T}}
\newcommand{\uu}{\breve{u}}
\newcommand{\W}{\mathcal{W}}
\newcommand{\Z}{\mathbb{Z}}
\newcommand{\Zp}{\mathbb{Z}_p}

\newcommand{\dcom}{\widetilde{R\Gamma}_{\mathrm{f}}}
\newcommand{\cohom}{\widetilde{H}_{\mathrm{f}}}

\newcommand{\hotimes}{\widehat{\otimes}}

\begin{document} 
	
	\title[$ p $-adic Artin formalism of twisted triple product Galois representations]{(Algebraic) $ p $-adic Artin formalism of twisted triple product Galois representations over real quadratic fields} 
	
	\author{Bhargab Das} 
	\address{(Bhargab Das) Harish-Chandra Research Institute, A CI of Homi Bhabha National Institute, Chhatnag Road, Jhunsi, Prayagraj - 211019, India}
	\email{bhargabdas@hri.res.in}
	
	\author{Aprameyo Pal} 
	\address{(Aprameyo Pal) Harish-Chandra Research Institute, A CI of Homi Bhabha National Institute, Chhatnag Road, Jhunsi, Prayagraj - 211019, India}
	\email{aprameyopal@hri.res.in}

	\date{\today }
	\keywords{Selmer complexes, $p$-adic $L$-functions, characteristic ideals}
	\subjclass[2020]{Primary 11R23, Secondary 11F41, 11G40, 14G10}
	
	\begin{abstract}
		In this article, we investigate factorization problems for twisted triple product Galois representations over real quadratic fields, arising from families of Hilbert cusp forms. Specifically, we address the factorization in two distinct settings determined by the order of vanishing of associated $L$-functions at their central critical values-namely, the rank (1,1) and rank (0,2) cases. Our results generalize the algebraic factorization framework developed by Büyükboduk et al. in higher rank scenario to the setting of real quadratic fields.
		Notably, our work yields the first known factorization result in the higher rank (0,2) case, marking a significant advancement in the study of triple product motives over totally real fields.

	\end{abstract}
	
	\maketitle

	\tableofcontents

	\section{Introduction}
	\label{section:intro}
	The foremost purpose of this article is to study the factorization problem for certain twisted triple product (algebraic) $ p $-adic $ L $-functions. We explore this problem in an irregular setup in the sense of Perrin-Riou's $ p $-adic $ L $-functions and \enquote*{regular} submodule. Our work is an real quadratic extension of the results obtained in \cite{BCS23} in two cases namely rank (1,1) (cf. \cite[\textsection 4.5.2]{Darmon16GenralisedKatoClasses}) and rank (0,2) (cf. \cite[\textsection 4.5.3]{Darmon16GenralisedKatoClasses}). Additionally we highlight significant differences/distinctions between our problem and those in \cite{Greenberg1982,Palva18} in \textsection \ref{sec:comparison}, which makes our cases particularly challenging. 
	
	Let us fix a prime $ p \ge 5 $ once and for all. We also choose $p$-adic embedding $ \overline{\Q} \hookrightarrow \overline{\Q}_p \hookrightarrow \bb{C}_p$, and a complex embedding $ \overline{\Q} \hookrightarrow \mathbb{C}$. We also fix an isomorphism $\iota_p \colon \bb{C} \rightarrow \bb{C}_p$, such that the following diagram commutes with our choices for a fixed real quadratic $ F/\Q $:
	\begin{equation*}
		\begin{tikzcd}[row sep=tiny]
			&&& \bb{C} \arrow[dd, "\iota_p"] \\
			\Q \arrow[r, hookrightarrow] &F \arrow[r, hookrightarrow]  &\overline{\Q} \arrow[ur, hookrightarrow] \arrow[dr, hookrightarrow] & \\
			&&& \bb{C}_p 
		\end{tikzcd} 
	\end{equation*}	
	We also assume that $ p $ splits in $ F $ and set $ \Gal(F/\Q) = \{1,\theta\} $. Let $ \f $ and $ \g $ be cuspidal primitive Hida families of Hilbert modular forms (see \cref{section : the_setup} for more details) and $ \g^\theta $ is the $ \theta $-twist of $ \g $. Suppose $ \g^\theta \cong \g \otimes \omega $ for a quadratic character $ \omega $. The first task we undertook is to formulate precise conjectures about factorization in the following cases: 
	\begin{itemize}
		\item[\textbf{(1,1)}] The $ (\g,\b)  $-unbalanced $ p $-adic triple product $ L $-functions $ \scr{L}_p^{(\g,\b)}(\f \otimes \As(\g) ) $ in \cref{conjecture_1} which predicts that 
		\begin{equation*}
			\scr{L}_p^{(\g,\b)}(\f \otimes \As(\g) )(P,Q)  = \scr{L}_p^{(\g,\b)}(\f \otimes \ad^0{\g}(\lambda) )(P,Q)\cdot i^*_\f\left(\Log_{\f \otimes \lambda}^1(\mathbf{BK}_{\f\otimes \lambda}^{\rm LZ}) \right)
		\end{equation*} 
		for a quadratic Galois character $ \lambda $ related to $ \omega $. Here the factor $ \Log_{\f \otimes \lambda}^1(\mathbf{BK}_{\f\otimes \lambda}^{\rm LZ}) $ is related to the Euler system constructed in \cite{loeffler2020iwasawatheoryquadratichilbert} which is closely associated to the derivative of $ p $-adic $ L $-function of $ \f \otimes \lambda $ (see \textsection \ref{6.1.2}).
		\item[\textbf{(0,2)}] The $ \g  $-unbalanced $ p $-adic triple product $ L $-functions $ \scr{L}_p^{\g}(\f \otimes \As(\g) ) $ in \cref{conjecture:3} predicting that
		\[	\scr{L}_p^{\g}(\f \otimes \As(\g) )  = \scr{L}_p^{\b}(\f \otimes \ad^0{\g}(\lambda) )\cdot i^*_\f\left(\Log^2_{\f \otimes \lambda}(\mathbf{BK}^{\rm Plec}_{\f\otimes \lambda,1} \wedge \mathbf{BK}^{\rm Plec}_{\f\otimes \lambda,2} )\right), \]
		for a quadratic Galois character $ \lambda $ related to $ \omega $. Here we consider the conjectural existence of rank 2 Euler system $ \mathbf{BK}^{\rm Plec}_{\f\otimes \lambda,1} \wedge \mathbf{BK}^{\rm Plec}_{\f\otimes \lambda,2} $ by \ref{plectic} which has evidence in some recent works of Henri Darmon, Michele Fornea, Lennart Gehrmann and many more (cf. \cite{fornea23plectic, Darmon_Fornea_2025}, also see \textsection \ref{7.1.2}).   
	\end{itemize}
	In both cases we consider the situation when the interpolation range is empty and sign conditions such that analogues problem for $ \f $-dominant and balanced triple product $ p $-adic $ L $-functions reduces to $ 0=0 $.
	
	The general theory of Iwasawa-Greenberg main conjecture would suggest an algebraic counterpart of these two cases of factorization conjectures. The author gave some evidence towards these versions of the main conjectures in \cite{myself2}. These (algebraic factorizations) are formulated in terms of \enquote{modules of leading terms}. Inspired from \cite{BCS23} we defined and study these objects in \textsection \ref{sec : module_leading} in a much general context and produced some useful results for \enquote{rank 2} scenario, using the general theory of exterior power by dual modules. The results we got are as follows:
	\begin{itemize}
		\item[\textbf{(1,1)}] The module of $ (\g,\b)  $-unbalanced (algebraic) $ p $-adic triple product $ L $-functions $ \scr{L}_p^{\rm alg}(\f \otimes \As(\g), \Delta_{(\g, \b)} ) $ can be written as
		\begin{equation*}
			\scr{L}_p^{\rm alg}(\f \otimes \As(\g), \Delta_{(\g, \b)} )  = 	\scr{L}_p^{\rm alg}(\f \otimes \ad^0{\g}(\lambda), \tr \Delta_{(\g, \b)} )\cdot i^*_\f\left(\Log_{\f \otimes \lambda}^1(\mathbf{BK}_{\f\otimes \lambda}^{\rm LZ}) \right),
		\end{equation*} 
		for a quadratic Galois character $ \lambda $ related to $ \omega $. We urge the reader to check \cref{theoremA} and \cref{main_theorem_1} for this result written in the modules of leading terms language. 
		\item[\textbf{(0,2)}] The module of $ \g  $-unbalanced (algebraic) $ p $-adic triple product $ L $-functions $ \scr{L}_p^{\rm alg}(\f \otimes \As(\g), \Delta_{\g} ) $ factors as
		\[	\scr{L}_p^{\rm alg}(\f \otimes \As(\g), \Delta_{\g} )  = \scr{L}_p^{\rm alg}(\f \otimes \ad^0{\g}(\lambda), \tr \Delta_{\b} )\cdot i^*_\f\left(\Log^2_{\f \otimes \lambda}(\mathbf{BK}^{\rm Plec}_{\f\otimes \lambda,1} \wedge \mathbf{BK}^{\rm Plec}_{\f\otimes \lambda,2} )\right), \]
		for a quadratic Galois character $ \lambda $ related to $ \omega $. Again we request the reader to see \cref{theoremB} and \cref{main_theorem_2} for the modules of leading terms version.    
	\end{itemize}
	We also remark that these are unlike the scenario considered in \cite{Palva18} (see \textsection \ref{sec:comparison} for more discussion) and reflect of the \emph{irregularity} of the factorization problems discussed in this paper. We also state the difference between our work and that in the algebraic part of \cite{BCS23} in terms of complications arising from Greenberg-local conditions (see \textsection \ref{sec:comparison} and \textsection \ref{sec:greenberg} for more discussion on Greenberg-local conditions).   
	\subsection{Comparison with earlier work}
	\label{sec:comparison}
	We explain in this subsection about a comparison of the setting of our present work with that of \cite{Greenberg1982, Palva18, BCS23} and highlight the key technical differences and difficulties.
	\subsubsection{}
	We have considered the representations over $ G_F $ instead of $ G_\Q $ for a real quadratic field $ F/\Q $. In order to do so we have to consider local conditions in two different places associated to $ \p_1 $ and $ \p_2 $ separately along with the bad places. This is a stark contrast between our work and results in \cite{Greenberg1982, Palva18, BCS23} as in these previous works the authors only considered representations of $ G_\Q $. \\
	In fact we hope that we can further pursue our factorization results in more \emph{analytic} way extending the results of \cite{Gross80,Dasgupta16,2024artinformalismtripleproduct} towards real quadratic field along with shedding some lights about BDP philosophy for Hilbert modular forms.
	\subsubsection{}
	The authors in \cite{Greenberg1982, Palva18} considered only Selmer complexes with torsion cohomology group (cf. \cite[Apendix A]{BCS23} and \cite[\textsection2.4.3]{BCS23}). But the complexes in exact triangles \eqref{comtriangle1} and \eqref{comtriangle} in our consideration have non-torsion cohomology groups like the situation in \cite{BCS23}. Additionally we have shown that we may have rank two cohomology groups via the Panchiskin defect (see \textsection \ref{sec_pan}) which is not the case of \cite{BCS23}. Thus we have extracted more complicated formulas regarding higher rank situations in \textsection \ref{sec : module_leading} accounts for	the appearance of $p$-adic regulators and derivatives of $p$-adic $L$-functions in our setup. This is the first time such a result has been established for the higher rank (0,2) case, highlighting the novelty of our approach.
	\subsubsection{Twisted triple product over $ \Q $}
	We conclude this section discussing the triple product when $ \f $ is a family of cuspforms on $ \GL{2,\Q} $ and $ \g $ is as usual. Then the $ 8 $-dimensional representation $ T_\f \hotimes_{\Zp} \As(\g) $ (cf. \cite[Definition 6.4]{ForneaJin}) has similar decomposition like \eqref{mainses}. In this situation the factorization problem of (algebraic) $p$-adic $L$-functions (which are conjecturally related with the $ p $-adic $ L $-functions constructed in \cite{BCF}) has almost similar flavor like the one in \cite{BCS23} with a twist by quadratic character. We will discuss this scenario more precisely in an imminent survey article about the $ p $-adic Artin formalism phenomena.     
	
	\subsection{Layout}
	We sketch an outline of the structure of our article. We briefly introduce the basic notions about Hilbert cuspforms in \textsection \ref{sec2.1}, \textsection \ref{2.2}, and \textsection \ref{2.3}. Thereafter we define the Asai representation in \textsection \ref{sec2.4} and use this to explain our set-up in \textsection \ref{subsection:Decomposition of Big Galois Modules}. Then we record our main results about \emph{rank (1,1) case} (resp. \emph{rank (0,2) case}) in \textsection \ref{rank1_scenarion} (resp. \textsection \ref{rank2_scenario}). 
	
	In \cref{section : Selmer_complexes}, we recall the machinery of Selmer complexes constructed in \cite{Nek06}. Using this we define the module of (algebraic) $p$-adic $L$-functions of Galois representations of our concern. We make a choice of suitable Greenberg-local conditions to study the factorization problems. We also computed necessary Tamagawa factors in \textsection \ref{sec:tamagawa} which is useful to show the perfectness of Selmer complexes.
	
	In \cref{sec : module_leading}, we introduce a canonical submodule of an extended Selmer module, which is called the module of leading terms and discuss its relation with Kolyvagin systems in \cref{koly_ralation}.  
	
	Based on this construction, we prove Theorem \ref{theoremA} in \cref{section : factorization_1} and Theorem \ref{theoremB} in \cref{section : factorization_2}.
	
	\begin{acknowledgements}
		The first author acknowledges the support of the HRI institute fellowship for research scholars and the second author acknowledges the SERB SRG grant (SRG/2022/000647). The authors would also like to thank Dr. K\^{a}zim B\"{u}y\"{u}kboduk for listening to the doubts with patience and sharing his insights in TIFR Mumbai. 
	\end{acknowledgements}
	\section{The Setup}
	\label{section : the_setup}
	As mentioned in \cref{section:intro}	Let us fix a prime $ p \ge 5 $ once and for all. We also choose $p$-adic embedding $ \overline{\Q} \hookrightarrow \overline{\Q}_p \hookrightarrow \bb{C}_p$, and a complex embedding $ \overline{\Q} \hookrightarrow \mathbb{C}$. We also fix an isomorphism $\iota_p \colon \bb{C} \rightarrow \bb{C}_p$, such that the following diagram commutes with our choices for a fixed real quadratic $ F/\Q $:
	\begin{equation*}
		\begin{tikzcd}[row sep=tiny]
			&&& \bb{C} \arrow[dd, "\iota_p"] \\
			\Q \arrow[r, hookrightarrow] &F \arrow[r, hookrightarrow]  &\overline{\Q} \arrow[ur, hookrightarrow] \arrow[dr, hookrightarrow] & \\
			&&& \bb{C}_p 
		\end{tikzcd} 
	\end{equation*}	
	In this section at first, we will recall some preliminary results about families of Hilbert cuspforms.
	
	\subsection{Hilbert Cuspforms}
	\label{sec2.1}
	Suppose $ F $ be a real quadratic number field with Galois group $ \Gal(F/\Q) = \{1, \theta\} $ and ring of integers $ \Oo_F $. Let us denote the set of embedding $ F \hookrightarrow \overline{\Q} $ as $ I_F $ and put $ t_F = \sum_{\tau \in I_F} [\tau] \in \Z[I_F]$. We also assume that $ p $ splits in $ \Oo_F $. We denote $ \frak{d}_{F/\Q} $ as the discriminant of $ F $ and $ \frak{D}_{F/\Q} = N_{F/\Q}(\frak{d}_{F/\Q}) $. Let $ \theta_F $ be the adelic character associated with $ \theta $.
	\begin{defn}
		We define a power map for every $ s = \sum_{\tau \in I_F} s_\tau[\tau] $ as
		$$ (\cdot)^s \colon F \otimes_\Q \overline{\Q} \rightarrow \overline{\Q}, \quad x \otimes c \mapsto \prod_{\tau}(c \tau(x))^{s_\tau} . $$
	\end{defn}  
	Let $ Z_F(1) = F^\times \backslash \bb{A}_F^\times/\widehat{\Oo}^{p,\times}_F F^\times_{\infty,+} $, then there is a group homomorphism induced via the cyclotomic character $ \chi_{\mrm{cycl},F} $,
	$$ \check{\chi}_{\mrm{cycl},F} \colon Z_F(1) \rightarrow \Z_p^\times, \quad y \mapsto y_p^{-t_F}\mid y^\infty \mid^{-1}_{\bb{A}_F}. $$ 
	Moreover, the isomorphism $ \Z_p^\times \cong (1 + p\Zp) \times (\Z/p\Z)^\times$ gives the factorization 
	$$ \check{\chi}_{\mrm{cycl},F} = \eta_F \cdot \omega_F. $$

	Let $ i \in \mathbb{C} $ be the square root of unity and consider the Poincare half-plane $ \mathfrak{H} $ with respect to $ i $. Let us define the algebraic group $ \operatorname{Res}_{F/\Q}\GL{2,F} $ as $ \cal G $.  There is a transitive group action of $\cal G(R)^+ \cong \prod_{I_F} \GL{2} (\mathbb{R})^+ $ on the complex manifold $ \mathfrak{H}^{I_F} $ and $ C_\infty^+ \coloneqq \operatorname{Stab}_{\cal G(R)^+}(\mathbf{i}) $ where $ \mathbf{i} = (i,...,i) \in \mathfrak{H}^{I_F}$.  \\
	For a non zero integral ideal $ \mathfrak{N} \subset \Oo_F $ lets define the following open compact subgroups of $ \cal G(\widehat{\Z}) $
	\begin{itemize}
		\item $ U_0(\mathfrak{N}) \coloneqq \biggl\{ 
		\begin{pmatrix}
			a &b \\
			c &d
		\end{pmatrix}
		\in \cal G(\widehat{\Z}) \mid c \in \mathfrak{N}\widehat{\Oo}_F
		\biggl\} $,
		\item $ U_1(\mathfrak{N}) \coloneqq \biggl\{ 
		\begin{pmatrix}
			a &b \\
			c &d
		\end{pmatrix}
		\in U_0(\mathfrak{N}) \mid d-1 \in \mathfrak{N}\widehat{\Oo}_F
		\biggl\} $,
		\item $ U^1(\mathfrak{N}) \coloneqq \biggl\{ 
		\begin{pmatrix}
			a &b \\
			c &d
		\end{pmatrix}
		\in U_0(\mathfrak{N}) \mid a-1 \in \mathfrak{N}\widehat{\Oo}_F
		\biggl\} $,
		\item $V(\mathfrak{N})=U_1(\mathfrak{N})\cap U^1(\mathfrak{N})$.
	\end{itemize}
	Let $ K \le \cal G(\mathbb{A}^\infty) $ be an open compact subgroup,
	\begin{defn}
		A holomorphic Hilbert cuspform of weight $ (k,w)\in \Z[I_F]^2,\, (k-2w = m \cdot t_F,\, m \in \Z_{>0}) $, level $K$ is a continuous function $ f \colon \cal G(\mathbb{A}) \rightarrow \mathbb{C} $ such that the following properties hold,
		\begin{itemize}
			\item $f(\alpha x u) = f(x) j_k(u_\infty,\mathbf{i})^{-1}  $ where $ \alpha \in \cal G(\Q) $, $ u \in K \cdot C_\infty^+$ and the automorphy factor is $$ j_k\Bigg(\begin{pmatrix}
				a &b \\
				c &d
			\end{pmatrix},z\Bigg) = (ad-bc)^w(cz +d)^k $$ for $ \begin{pmatrix}
				a &b \\
				c &d
			\end{pmatrix} \in \cal G(\mathbb{R}) $, $ z \in \mathfrak{H}^{I_F} $.
			
			\item for every finite adelic point $x \in \cal G( \bb{A}^\infty )$ the well-defined function $ f_x \colon \mathfrak{H}^{I_F} \rightarrow \mathbb{C} $ given by $ f_x(z) = f(xu_\infty)j_k(u_\infty, \mathbf{i}) $ is holomorphic, where for each $ z \in \mathfrak{H}^{I_F} $ we choose $ u_\infty \in \cal G(\mathbb{R})^+ $ such that $ u_\infty \mathbf{i} = z $.
			\item for all adelic points $x \in \cal G ( \mathbb{A} )$ and for all additive measures on $F \backslash \mathbb{A}_F$ we have
			$$ \int_{F \backslash \mathbb{A}_F}\Bigg(\begin{pmatrix}
				1 &a \\
				0 &1
			\end{pmatrix}x\Bigg)da= 0 $$
			\item for all finite adelic point $x \in \cal G( A^\infty )$ the function $\mid \Im(z)^{k/2}f_x(z)\mid$ is uniformly bounded on $\mathfrak{H}^{I_F}$.
		\end{itemize}
		We denote the space of such Hilbert cuspforms as $ S_{k,w}(K;\mathbb{C}) $ and when $ K = U_1(\mathfrak{N}) $ it will be $ S_{k,w}(\mathfrak{N};\mathbb{C}) $.
	\end{defn}
	Moreover, \cite[Theorem 1.1]{Hida91} states that each cuspform in $ S_{k,w}(V(\mathfrak{N});\mathbb{C}) $ has an adelic $ q $-expansion. Let $\bb{A}_{F,+}^{\times}\coloneqq (\bb{A}_{F}^{\infty})^\times F^\times_{\infty,+} $, and we denote the following
	\begin{equation*}
		\msf{a}(\cdot,f)\colon \bb{A}_{F,+}^{\times} \to \bb{C}, \quad \msf{a}_p(\cdot,f)\colon \bb{A}_{F,+}^{\times} \to \overline{\bb{Q}}_p
	\end{equation*}
	as the Fourier coefficient functions for every $ f \in S_{k,w}(V(\frak{N});\mathbb{C}) $. Then for any $ \Oo_F $-subalgebra $ A \subset \bb{C} $, we denote
	$$S_{k,w}(K;A) \coloneqq \{ f \in S_{k,w}(K;\mathbb{C}) \mid \msf{a}(y,f) \in A \text{ for all } y \in \bb{A}_{F,+}^{\times} \}.$$ 
	\begin{rmk}
		$ \msf{a}(y,f) = 0 = \msf{a}_p(y,f) $ if $ y \notin \widehat{\Oo}_F^\times F^\times_{\infty, +} $. 	
	\end{rmk}
	\subsubsection{Hecke Theory}
	Let $V(\frak{N}) \le K \le U_0(\frak{N}) $ be an open compact subgroup of $ \cal G(\bb{A}^\infty) $. For any $ g \in \cal G(\bb{A})$, one can consider the following double coset operator $ [KgK] $, acting on a Hilbert cuspform $ f $ of level $ K $ 
	\begin{equation*}
		([KgK]f)(x) = \coprod_i f(x\gamma_i)
	\end{equation*} 
	where $ \gamma_i $'s are coming from the decomposition $ [KgK]=\coprod_i\gamma_iK $.
	\begin{defn}
		\label{defn : Hecke_operator}
		\hfill
		\begin{itemize}
			\item For every prime ideal $ \frak{q} $ of $ \Oo_F $ and a choice of uniformizer $ \varpi_{\frak{q}} $ of $ \Oo_{F,\frak{q}} $, the Hecke operators of $ S_{k,w}(K;\mathbb{C}) $ are defined as $$ T(\varpi_{\frak{q}}) = \{\varpi_{\frak{q}}^{w-t_F}\}\left[K \begin{pmatrix}
				\varpi &0\\
				0 &1
			\end{pmatrix}K \right] . $$  
			\item For every $ a \in \Oo_{F,\frak{N}}^{\times}= \prod_{\frak{q} \mid \frak{N}}\Oo_{F,\frak{q}}^{\times} $, there is a Hecke operator 
			$$ T(a,1)=\left[ K \begin{pmatrix}
				a &0\\
				0 &1
			\end{pmatrix}K \right] .$$
			\item For any $ z \in Z_{\cal G(\bb{A}^\infty)} $ in the center of $ G(\bb{A}^\infty) $, the diamond operator $ \left\langle z \right\rangle  $ acts via 
			$$ ( \left\langle z \right\rangle f)(x)=f(xz) .$$
		\end{itemize}
	\end{defn}
	The Hecke algebra  $ H_{k,w}(K;\Oo_F) $ is defined to be the $ \Oo_F $-subalgebra of $ \mrm{End}_{\bb{C}}\left( S_{k,w}(K;\mathbb{C})\right)  $ generated by the Hecke operators in the Definition \ref{defn : Hecke_operator}. For any $ \Oo_F $-algebra $ A $ contained in $ \bb{C} $, we have $$ H_{k,w}(K;A)\coloneqq H_{k,w}(K;\Oo_F) \otimes_{\Oo_F} A .$$ 
	For every idele $ y \in \widehat{\Oo}_F \cap \bb{A}_F^\times $, let \( y = a \prod_{\frak{q}} \varpi_{\frak{q}}^{e_\frak{q}}u \) such that $ u \in \det(V(\frak{N})) $ and $ a \in \Oo_{F,\frak{N}}^\times $. The Hecke operator 
	\begin{equation}
		\label{big_T_operator}
		T_0(y) = T(a,1)\prod_{\frak{q}} T(\varpi_{\frak{q}}^{e_\frak{q}})
	\end{equation}
	depends only on $ y $.
	\begin{rmk}\cite[Theorem 2.2]{Hida91}
		\label{cuspform_pairing}
		If $ A $ is integrally closed domain and finite flat over $ \Zp $. then $ S_{k,w}(K;A) $ is stable under $ H_{k,w}(K;A) $ and we have a nondegenerate bilinear pairing
		\begin{equation*}
			S_{k,w}(K;A) \times H_{k,w}(K;A) \rightarrow A, \quad (f,h) \mapsto \msf{a}(1,h(f)).
		\end{equation*}
	\end{rmk}
	\subsection{Hida Families of Hilbert Cuspforms}
	\label{2.2}
	Let $O$ be a valuation ring  in $\overline{\bb{Q}}_p$ finite flat over $\bb{Z}_p$ containing $ \Oo_F $. Let $\frak{N}$ be an $\mathcal{O}_F$-ideal prime to $p$. For a open compact subgroup $ K $ such that $U_1(\frak{N})\le K\le U_0(\frak{N})$, we set
	$K(p^r) = K \cap V(p^r)$ and $K(p^\infty) = \cap_{r\ge1} K(p^r)$. Lets define the projective limit of $p$-adic Hecke algebras as
	\[
	\mbf{H}_F(K;O):=\varprojlim_r\ H_{k,w}(K(p^r);O)
	\]
	which is independent of the weight $(k,w)$. Since $\mathbf{H}_F(K;O)$ is compact, we have a decomposition of algebras
	\[
	\mathbf{H}_F(K;O) = \mathbf{H}^{\mrm{n.o.}}_F(K;O) \oplus \mathbf{H}_F^{\mrm{ss}}(K;O)
	\] 
	such that $\msf{T}(\varpi_p) = \varprojlim_r T_0(\varpi_p) $ is a unit in $\mathbf{H}_F^{\mrm{n.o.}}(K;O)$ and topologically nilpotent in $\mathbf{H}_F^{\mrm{ss}}(K;O)$. We denote by 
	\begin{equation}\label{eno}
		e_{\mrm{n.o.}} = \underset{n \to \infty}{\lim} \msf{T}(\varpi_p)^{n!}
	\end{equation}
	the idempotent corresponding to the nearly ordinary part $\mathbf{H}^{\mrm{n.o.}}_F(K;O)$. 
	For any $r\ge1$ we set
	\begin{equation}
		Z^r_F(K) := \bb{A}_F^\times/F^\times (\bb{A}_{F}^\infty \cap \det K(p^r)) F_{\infty,+}^\times\quad\text{and}\quad \bb{G}_F^r(K):= Z^r_F(K)\times(\mathcal{O}_F/p^r\mathcal{O}_F)^\times.
	\end{equation}
	If we denote the projective limits by
	\begin{equation}
		Z_F(K):= \varprojlim_r\ Z^r_F(K),\qquad \bb{G}_F(K) :=\varprojlim_r\ \bb{G}^r_F(K),
	\end{equation}
	then $\bb{G}_F(K) = Z_F(K)\times\mathcal{O}_{F,p}^\times$ and we have a group homomorphism 
	\begin{equation}
		\bb{G}_F(K)\to \mathbf{H}_F(K;O)^\times,\qquad (z,a)\mapsto \langle z,a\rangle:=\diam{z} T(a^{-1},1)
	\end{equation}
	that gives  $\mbf{H}_F(K;O)$ a $O\llbracket\bb{G}_F(K)\rrbracket$-algebra structure.

	\noindent Let $\mrm{cl}_F^+(\frak{N}p)$ be the strict ray class group of modulus $\frak{N}p$. We denote $\overline{\mathcal{E}}^+_{\frak{N}p}$ as the closure of the totally positive units of $\mathcal{O}_F$ congruent to $1\pmod{\frak{N}p}$ inside $\mathcal{O}_{F,p}^\times$. Hence we have an exact sequence
	$$	1 \rightarrow \overline{\mathcal{E}}^+_{\frak{N}p}\backslash\big(1+p\mathcal{O}_{F,p}\big) \rightarrow Z_F(K) \rightarrow \mrm{cl}_F^+(\frak{N}p) \rightarrow 1,	$$
	which splits for large enough $p$ since the order of the group $\mrm{cl}_F^+(\frak{N}p)$ is prime to $p$. We denote the canonical decomposition by
	\begin{equation}\label{GaloisDecomposition}
		Z_F(K) \overset{\sim}{\to} \overline{\mathcal{E}}^+_{\frak{N}p}\backslash\big(1+p\mathcal{O}_{F,p}\big) \times \mrm{cl}_F^+(\frak{N}p),\qquad z\mapsto \big(\xi_z,\bar{z}\big).	
	\end{equation} 
	If we set
	$\mathfrak{J}_F = \overline{\mathcal{E}}^+_{\frak{N}p}\backslash\big(1+p\mathcal{O}_{F,p}\big) \times (1+p \mathcal{O}_{F,p})$, then the exact sequence
	$$ 1 \rightarrow \mathfrak{J}_F \rightarrow \bb{G}_F(K) \rightarrow \mrm{cl}_F^+(\frak{N}p)\times (\mathcal{O}_F/p)^\times \rightarrow 1 $$	
	splits canonically for large enough $p$. The group $\mathfrak{J}_F$ is a finitely generated $\bb{Z}_p$-module of $\bb{Z}_p$-rank $[F:\bb{Q}]+1+\delta$, where $\delta$ is Leopoldt's defect for $F$. Let $W$ be the torsion-free part of $\mathfrak{J}_F$ and denote by
	$\Lambda_F = O\llbracket W \rrbracket$ the associated completed group ring.

	\begin{thm}\cite[Theorem 2.4]{hida89nearly}
		The nearly ordinary Hecke algebra $\mathbf{H}_F^{\mrm{n.o.}}(K;O)$ is torsion-free and finite over $\Lambda_F$.
	\end{thm}

	\noindent The completed group ring $O\llbracket\bb{G}_F(\frak{N})\rrbracket$ naturally decomposes as the direct sum $\bigoplus_\Psi \Lambda_{F,\Psi}$ ranging over all the characters of the torsion subgroup
	$\bb{G}_F(\frak{N})_\mrm{tor} = \mathfrak{I}_{F,\mrm{tor}}\times \mrm{cl}_F^+(\frak{N}p)\times (\mathcal{O}_F/p)^\times$.
	It induces a decomposition of the nearly ordinary Hecke algebra $
	\mbf{H}_F^{\mrm{n.o.}}(K;O) = \bigoplus_{\Psi}\mathbf{H}_F^{\mrm{n.o.}}(K;O)_\Psi$.
	
	\begin{defn}\label{def I-adic cuspforms}
		Let $\Psi: \bb{G}_F(K)_\mrm{tor} \rightarrow O^\times$ be a character. For any $\Lambda_{F,\Psi}$-algebra  $\R$, the space of nearly ordinary $\R$-adic cuspforms of tame level $K$ and character $\Psi$ is 
		\[
		\overline{\mbf{S}}_F^\mrm{n.o.}(K,\Psi;\R) :=
		\mrm{Hom}_{\bs{\Lambda}_{F,\Psi}\mbox{-}\mrm{mod}} \big( \mbf{H}_F^\mrm{n.o.}(K(p^\infty);O)_\Psi, \R\big).
		\]
		When an $\R$-adic cuspform is also a $\Lambda_{F,\Psi}$-algebra homomorphism, we call it a Hida family.
	\end{defn}

	Let
	$\psi: \mrm{cl}_F^+(\frak{N}p^\alpha) \rightarrow O^\times$,
	$\psi': (\mathcal{O}_F/p^\alpha)^\times \rightarrow O^\times$  be a pair of characters 
	and  $(k,w)$ a weight satisfying $k-2w = mt_F$. The group homomorphism
	\[
	\bb{G}_F(K)\to O^\times,\qquad (z,a) \mapsto \psi(z)\psi'(a)\check{\chi}_{\mrm{cycl},F}(z)^ma^{t_F-w}
	\]
	determines a $O$-algebra homomorphism $P_{k,w,\psi,\psi'}:O\llbracket\bb{G}_F(K) \rrbracket\rightarrow O$.

	\begin{defn}
		For  a $\Lambda_{F,\Psi}$-algebra $\R$ the set of arithmetic points, denoted by $\W_{\mrm{cl}}(\R)$, is defined as the subset of $\W(\R) = \operatorname{Spf}(\R)(\bb{C}_p)$ consisting of $ O $-algebra homomorphisms those are equivalent with some $P_{k,w,\psi,\psi'}$ when restricted to $\Lambda_{F,\Psi}$. For any such $ P \in \W_{\mrm{cl}}(\R) $ we denote $ k $ as $ k(P) $, $ w $ as $ w(P) $ and $ m(P) \in \Z_{>0} $ such that $ k(P)-2w(P)=m(P).t_F $. 
	\end{defn}
	By the \cite[Theorem 2.4]{hida89nearly}, for any such $ P \in \W_{\mrm{cl}}(\R) $ and a Hida family $ \f \in \overline{\mbf{S}}_F^\mrm{n.o.}(K;\R) $, the composite $ \f(P) \coloneqq P \circ \f $ is a $ \overline{\Q}_p $-linear map $ \f(P) \colon  H_{k,w}(K(p^r); \overline{\Q}_p ) \rightarrow \overline{\Q}_p $. Hence by the Remark \ref{cuspform_pairing} $ \f(P) $ represents a Hilbert cuspform $ \f_P \in S_{k,w}(K(p^r),\overline{\Q}_p) $ such that $ \msf{a}_p(y,\f_P) = \f(P)(\msf{T}(y^\infty)) $ for all $ y \in \widehat{\Oo}_F^\times F^\times_{\infty, +} $. In this article we use both $ \f_P $ and $ \f(P) $ to represent this cuspform. Moreover $ \f_p $ is nearly ordinay at $ p $ i.e. $ \f_P $ is normalized eigenform and for all $ \p \mid p$, the $ T(\varpi_{\frak{p}}) $-eigenvalue is a $ p $-adic unit. In fact, for any Hecke eigenform $ f \in S_{k,w}(K(p^r),\overline{\Q}_p) $, there exist some character $ \Psi  $ and a $ \Lambda_{F,\Psi} $-algebra $ \R $ such that the $ p $-stabilization $ f^{(p)} = \f_P $ for some Hida family $ \f \in \overline{\mbf{S}}_F^\mrm{n.o.}(K;\R) $ and $ P \in \W_\Psi(\R) $. That $ \f $ is called a branch of Hida family along $ f^{(p)} $.
	\begin{defn}
		The set of Crystalline point $\W_{\mrm{crys}}(\R)$ is the subset of $ \W_{\mrm{cl}}(\R) $ consisting of the arithmetic points $ P $ such that $ P_{\mid \Lambda_{F,\Psi}} \equiv P_{k,w,\psi,1\mid \Lambda_{F,\Psi}} $ and the eigenform $ \f_P $ is $ p $-old. 
	\end{defn}   
	
	\subsection{Galois Representation Associated to Families of Hilbert Cuspforms}
	\label{2.3}
	Let $ \Psi $ be a character of $ \mrm{cl}_F^+(\frak{N}p) $ which induces a character of $ \bb{G}_F(\frak{N})_{\mrm{tor}} $ and a Galois character. Both are denoted by $ \Psi $. Suppose $ \f \in \overline{\mbf{S}}_F^\mrm{n.o.}(\frak{N},\Psi;\R)$ is an ordinary Hida family passing through $ p $-stabilization $ f^{(p)} $ of a Hilbert cuspform $ f $. The universal weight character 
	$$ \mathbbl{x} \colon G_{F,\Sigma} \rightarrow W \rightarrow O\llbracket W  \rrbracket^\times \rightarrow  \Lambda_{F,\Psi}^\times  $$
	is induced via global class field theory, gives rise to the character
	$$ \mathbbl{x}_{\f} \colon G_{F,\Sigma} \xrightarrow{\mathbbl{x}}  \Lambda_{F,\Psi}^\times  \rightarrow \R^\times. $$ 
	For any $ P \in \W_{\mrm{cl}}(\R) $, let $  \mrm{wt}(P) = P_{\mid \Lambda_{F,\Psi}} $. Hence we have $ P \circ  \mathbbl{x}_{\f} = \mrm{wt}(P) \circ \mathbbl{x}  $. 
	
	Let $ \mathcal{Q} = \operatorname{Frac}(\R) $, then by \cite[Theorem 1]{hida89reps}, there is a unique big Galois representation $ \rho_{\f} $ associated with $ \f $ acting on $ V_\f^{\dagger} \cong \mathcal{Q}^2 $ . In particular, $ \rho_{\f} $ is continuous with $ \operatorname{im}(\rho_{\f}) \subset \operatorname{End}_{\R}(T_\f^\dagger) $ and satisfying 
	\begin{equation}
		\bigwedge^2\rho_{\f} \equiv \R( \mathbbl{x}_{\f}\chi_{\mrm{cycl},F}^{-1}\Psi_{\f})
	\end{equation}  
	where $ T_\f^\dagger$ is an $ \R $-lattice inside $ V_{\f}^{\dagger} $. Moreover $ \rho_{\f} $ also satisfies 
	\begin{equation}
		\label{lambda_adic_char_poly}
		\det(1-\rho_{\f}(\operatorname{Fr}(\mathfrak{q}))X) = 1 -\f(\msf{T}(\varpi_{\mathfrak{q}}))X + \Psi(\varpi_{\mathfrak{q}})[\xi_{\varpi_{\mathfrak{q}}}^{-t_F}]\operatorname{N}_{F/\Q}\mathfrak{q}X^2
	\end{equation}
	for all prime ideals $ \mathfrak{q} \nmid \frak{N}p $. 
	We will assume that	
	\begin{itemize}
		\item[\mylabel{Irr}{\textbf{(Irr)}}]  The residual representation of $ \rho_{\f} $ is absolutely irreducible.
	\end{itemize}
	By \cite[Theorem 2.2.2]{Wil88}, we have for any prime ideal $ \p \mid p $, there is an exact sequence of $ G_{F_{\p}} $-modules
	$$ 0 \rightarrow \scrF^+_{\p} T_{\f}^{\dagger} \rightarrow T_{\f}^{\dagger} \rightarrow \scrF^-_{\p}T_{\f}^{\dagger} \rightarrow 0 $$
	such that each $  F_{\p}^{\pm} V_{\f}^{\dagger} $ is rank one $ \mathcal{R} $-module.  Hence we get	 
	\begin{equation*}
		\rho_{\f}|_{G_{F_{\p}}} \simeq \begin{pmatrix}
			\delta_{\f,\p} &* \\
			0 &\alpha_{\f,\p}
		\end{pmatrix}
	\end{equation*}
	where $ \alpha_{\f,\p} \colon G_{F_{\p}} \rightarrow \R^\times $ is the unramified character such that $ \alpha_{\f,\p}(\mrm{Fr}_{\frak{p}}) $ is a $ p $-adic unit and root of \eqref{lambda_adic_char_poly}.	
	Also assume the following hypothesis
	\begin{itemize}
		\item[\mylabel{DIST}{\textbf{(Dist)}}] $ \delta_{\f,\p} \not\equiv \alpha_{\f,\p}\mod\mathfrak{m}_{\R} $ for all $ \p \mid p $.
	\end{itemize}
	\begin{rmk}
		These two hypothesis \ref{Irr} and \ref{DIST} make sure that $ T_{\f}^\dagger $ is free of rank two as $ \R $-module and $ \scrF_{\p}^\pm T_{\f}^\dagger $ rank 1 free module for all $ \p \mid p $.
	\end{rmk}
	Assume that, there exist a character $ \Psi^{-\frac{1}{2}} $ such that $ (\Psi^{-\frac{1}{2}})^2 = \Psi $. Let $ V_\f = V_\f^\dagger \otimes \mathbbl{x}_\f^{-\frac{1}{2}}\Psi^{-\frac{1}{2}}\chi_{\mathrm{cycl}} $, then there is an isomorphism of $ \operatorname{Frac}(\R)\llbracket G_{F,\Sigma} \rrbracket $-modules
	\begin{equation*}
		V_\f \xrightarrow{\sim} V_\f^*(1) = \Hom(V,\operatorname{Frac}(\R))(1). 
	\end{equation*}
	Therefore, $ V_\f $ is a self-dual representation containing self-dual stable lattice
	\begin{equation} \label{selfdual_lattice}
		T_\f\coloneqq T_\f^\dagger \otimes \mathbbl{x}_\f^{-\frac{1}{2}}\Psi^{-\frac{1}{2}}\chi_{\mathrm{cycl}} .
	\end{equation} Hence the lattice $ T_{\f} $ fits in an exact sequence of $ \R[G_{F_\p}] $-modules
	\begin{equation*}
		0 \rightarrow \scrF^+_{\p} T_{\f} \rightarrow T_{\f} \rightarrow \scrF^-_{\p} T_{\f} \rightarrow 0.
	\end{equation*}
	Any Hida family in our article will satisfy \ref{Irr} and \ref{DIST} and admits a self dual free lattice defined in \eqref{selfdual_lattice}. 
	\subsection{Tensor Induction Representations}
	\label{sec2.4}
	Let $ H $ be an index two subgroup of finite group $ G $. Suppose $ k $ is a field and $ \rho $ is a representation of $H$ acting on a finite $ k $-vector space $ V $. 
	\begin{defn}
		The tensor induction of $ \rho $ is defined as 
		\begin{align*} (\Ind^{\otimes}_{G/H} &\rho) \colon G \rightarrow \operatorname{Aut}_k(V \otimes V) \\
			(\Ind^{\otimes}_{G/H} &\rho)(h)(v_1\otimes v_2)  = \rho(h)(v_1) \otimes \rho(\theta h \theta)(v_2), \\
			(\Ind^{\otimes}_{G/H} &\rho)(\theta)(v_1\otimes v_2)  = v_2 \otimes v_1,
		\end{align*}
		for all $ h \in H $, where $ \{1, \theta \} $ is right coset of $ H $.
	\end{defn}
	Let $ (\rho^{\theta},V^{\theta}) $ is the representation obtained via conjugation by $ \theta $ (i.e. $ \rho^\theta(h) = \rho(\theta h \theta^{-1} ) $ for all $ h \in H $) and as a vector space $ V^\theta =V $. Hence $ (\Ind^{\otimes}_{G/H} \rho)_{\mid H} \equiv \rho\otimes \rho^\theta $.
	\begin{prop}
		Let $ (\rho, V) $ be a self-dual representation, then $ \left((\Ind^{\otimes}_{G/H} \rho),V \otimes V^{\theta}\right) $ is reducible if and only if $V^*(\lambda)  \cong V^{\theta}$ for some character $ \lambda $. Moreover $ V \otimes V^{\theta} \cong \ad^0(V)(\lambda) \oplus k(\lambda) $ for some character $ \lambda $.
	\end{prop}
	The decomposition type of $ (\Ind^{\otimes}_{G/H} \rho)$ is $ (3,1) $ when the former is reducible	\cite[Proposition 2.3]{For19}. 
	Let $ F/\Q $ be a totally real quadratic number field and $ \g $ be a $ \R_\g $-adic Hida family of Hilbert cuspforms of level $ \frak{N}_{\g} $ and character $ \Psi_{\g} $. Suppose 
	$ \Sigma \coloneqq \{\upsilon \mid \frak{d}_F \frak{N}_{\g}p\infty \} $
	be a finite set of places of $ F $, then $ G_{\Q, \Sigma}\coloneqq \Gal(F^{\Sigma}/\Q) $ where $ F^\Sigma $ is the maximal extension of $ F $ unramified outside $ \Sigma$ and $ G_{F, \Sigma}\coloneqq \Gal(F^{\Sigma}/F) $. Therefore we have $ G_{F,\Sigma} \triangleleft G_{\Q, \Sigma}$ and $ G_{\Q, \Sigma}/G_{F,\Sigma} $ is canonically isomorphic with $ \Gal(F/\Q) =\{1,\theta\} $. Let $ \rho_{\g} \colon G_{F,\Sigma} \rightarrow \GL{2}(\R_{\g}) $ be the self dual big Galois representation associated with $ \g $ acting on rank two $ \R_\g $-lattice $ T_\g $. We can extend the notion of tensor induction on Galois representation via projective limit. So we have
	\begin{prop}
		\label{prop : Asai_decomposition}
		Suppose $ \As(\rho_\g)\coloneqq  \Ind^{\otimes}_{G_{\Q, \Sigma}/G_{F, \Sigma}}{}^{}(\rho_\g)(-1) $. If $ \g $ has a trivial central character (i.e. $ \Psi_{\g} \equiv \mathds{1} $) and $ \rho_{\g}^{\theta} \cong \rho_{\g}^* (\dot{\lambda}) $ for some character $ \dot{\lambda} $ (either quadratic or trivial) then $ \As(\rho_\g)|_{G_{F, \Sigma}} \cong \ad^0(\rho_\g)(\lambda) \oplus \lambda$ by setting $ \lambda = \dot{\lambda}(-1) $.
	\end{prop}
	\begin{rmk}
		The representation $ \As(\rho_\g)$ is isomorphic to Asai-representation of $ \g $. When $ \g $ is a base change family of some Hida family of cuspform $ \g^\flat $ over $ \Q $, then there may exist a character $ \lambda \colon G_{\Q, \Sigma} \rightarrow \{\pm 1\} \subset \R_{\g}^\times$ associated with $ F/\Q $ via global class field theory such that $  \As(\rho_\g) $ becomes Asai-decomposable. For more general cases the family $ \g $ need not be decomposable but for some finite set of arithmetic points in $ \W_{\g}= \operatorname{Spf}(\R_\g)(\bb{C}_p) $, the specialization of $ \g $ may be decomposable. (cf. \cite[Section 3.2]{Liu16})
	\end{rmk}
	We have the representation space of $ \As(\rho_\g) $ as rank four $ \R_{\g} $-module $ T_\g \widehat{\otimes}_{O} T_\g^\theta$. By assumptions we have isomorphism of $ \R_{\g}\bbox{G_{\Q,\Sigma}} $-modules  $T_\g \widehat{\otimes}_{O} T_\g^\theta \cong T_\g \widehat{\otimes} T_\g^*(\lambda) \cong \ad(T_\g)(\lambda)\coloneqq\Hom_{\R_{\g}}(T_\g,T_\g)(\lambda) $ since $ T_\g^* \coloneqq \Hom_{\R_\g}(T_\g, \R_{\g}) \cong T_\g $. For $ \p =\p_1, \p_2 $,
	$ \ad(T_\g)(\lambda) $ acquires filtration of $ \R_{\g}\llbracket G_{F_{\p}}\rrbracket $-submodule
	
	\begin{align}
		\begin{aligned}
			\ad(T_\g)(\lambda) \supset \underbrace{\ker\{\ad(T_\g) \rightarrow \Hom(\scrF^{-}_{\p} T_\g , \scrF^{-}_{\p} T_\g)\}(\lambda)}_{=:\mathrm{Fil}^3_{\p}\ad(T_\g)(\lambda)} \supset \underbrace{\Hom(T_\g,\scrF^{+}_{\p} T_\g)(\lambda)}_{=:\mathrm{Fil}^2_{\p}\ad(T_\g)(\lambda)} \qquad \qquad \qquad \\ 
			\supset \underbrace{\Hom(\scrF^{+}_{\p} T_\g,\scrF^{+}_{\p} T_\g)(\lambda)}_{=:\mathrm{Fil}^1_{\p}\ad(T_\g)(\lambda)} \supset \{0\}
		\end{aligned}
	\end{align}
	
	\subsection{Decomposition of Big Galois Modules}
	\label{subsection:Decomposition of Big Galois Modules}
	Let $ \f $ (resp. $ \g $) be an ordinary Hida family of cuspforms on $ \cal G $ of level $ \frak{N}_\f  $ (resp. $ \frak{N}_{\g} $)  with branch $ \R_\f $ (resp. $ \R_\g $) and trivial central character. We also assume the finite prime $ \p \nmid \frak{N}_{\f}\frak{N}_\g $ for $ \p = \p_1,\p_2 $ and $ \f $ and $ \g $ satisfies
	
	Suppose 
	$ \Sigma$ be a finite set of places of $ F $ containing all the places dividing $ \frak{d}_F \frak{N}_{\f} \frak{N}_{\g}p\infty  $. Let $ F^\Sigma $ be the maximal extension of $ F $ unramified outside $ \Sigma $, then we denote 
	\begin{align*}
		G_{F,\Sigma} \coloneqq \Gal(F^\Sigma/F).
	\end{align*}
	Thus we have self-dual, rank two $ \R_{\f}\bbox{G_{F,\Sigma}} $-lattice (resp. $ \R_{\g}\bbox{ G_{F,\Sigma}} $-lattice) $ T_\f $ (resp. $ T_\g $). We denote $ \g^\theta $ as a Hida family of Hilbert cuspforms such that associated Galois representation $ T_{\g^\theta} $ is isomorphic to the $ \theta $-twist $ T_\g^\theta $ (i.e $\rho_{\g^\theta}(h) = \rho_{\g}(\theta h\theta^{-1}) $).
	Let $ G_{F,\Sigma} $ act on $ T_\g \widehat{\otimes}_{O}T_\g^\theta $ via $ \As(\rho_{\g})_{|G_{F,\Sigma}} $. We set the regular, local ring
	$$ \R \coloneqq \R_{\f} \widehat{\otimes}_{O} \R_{\g} .$$
	In this setup, let $ \T_3 = T_{\f} \widehat{\otimes}_{O} T_{\g} \widehat{\otimes}_{O} T_{\g}^{\theta} $ be the $ \R $-lattice of rank eight, then $$ \bs{\rho}_3 \coloneqq \rho_{\f} \otimes \As(\rho_{\g})_{|G_{F,\Sigma}} \colon G_{F, \Sigma} \rightarrow \operatorname{GL}_{\R}(\T_3)\cong \GL{8}(\R) $$ is a self-dual Galois representation. The associated weight space is $$ \W\coloneqq\operatorname{Spf}(\R)(\bb{C}_p)= \W(\R_\f) \times \W(\R_\g). $$
	The subset of arithmetic points is denoted by $ \W_{\mrm{cl}} = \W_{\mrm{cl}}(\R_\f) \times \W_{\mrm{cl}}(\R_\g) \subset \W $. Furthermore $ \W_{\mrm{cl}} $ has following disjoint subsets:
	\begin{itemize}
		\item The set of totally $ \f $-dominated points
		$$ \W_{\mrm{cl}}^\f \coloneqq \{ (P, Q) \in \W_{\mrm{cl}}\colon m(P) > 2m(Q) \}. $$
		\item The set of $ \f $-dominated at $ \tau_1 $ points
		\begin{align*}
			\W_{\mrm{cl}}^{(\f,\b)} \coloneqq \left\lbrace  (P, Q) \in \W_{\mrm{cl}}\colon \begin{aligned}
				k(P)_{\tau_1} \ge 2k(Q)_{\tau_1},\ k(P)_{\tau_2} < 2k(Q)_{\tau_2},\qquad \qquad \\
				m(P) =2m(Q),  
				\text{ and }k(P)_{\tau_2} + 2k(Q)_{\tau_2} \text{ is even} 
			\end{aligned}\right\rbrace .
		\end{align*}
		\item The set of $ \f $-dominated at $ \tau_2 $ points
		\begin{align*}
			\W_{\mrm{cl}}^{(\b,\f)} \coloneqq \left\lbrace  (P, Q) \in \W_{\mrm{cl}}\colon \begin{aligned}
				k(P)_{\tau_2} \ge 2k(Q)_{\tau_2},\ k(P)_{\tau_1} < 2k(Q)_{\tau_1},\qquad \qquad \\
				m(P) =2m(Q),  
				\text{ and }k(P)_{\tau_1} + 2k(Q)_{\tau_1} \text{ is even} 
			\end{aligned}\right\rbrace .
		\end{align*}
		\item The set of totally balanced points
		$$ \W_{\mrm{cl}}^{\mathrm{bal}} \coloneqq \{(P, Q) \in \W_{\mrm{cl}}\colon  m(P) < 2m(Q) \}. $$
	\end{itemize}
	\subsubsection{}
	Let us consider the natural $ O $-algebra homomorphisms
	\begin{equation}
		\begin{split}
			i^*_\f : \R_\f \rightarrow \R  &; \quad a   \mapsto a \otimes 1; \\
			i^*_\g : \R_\g \rightarrow \R  &; \quad b   \mapsto 1 \otimes b;
		\end{split}
	\end{equation}
	which also induces surjective morphisms
	$$ \W \xrightarrow{\pi_\f} \W(\R_\f), \quad  \W \xrightarrow{\pi_\g} \W(\R_\g).  $$
	Suppose $ T_{\g}^\theta \cong T_{\g}^*(\dot{\lambda}) $ for some quadratic character $ \dot{\lambda} \colon G_{F,\Sigma} \rightarrow \{\pm1\} \subset \R_{\g} $. By abuse of notation, we also define the characters $ \dot{\lambda} \colon G_{F,\Sigma} \rightarrow \{\pm1\} \subset \R_{\f} $ and $ \dot{\lambda}\colon G_{F,\Sigma} \rightarrow \{\pm1\} \subset \R $ taking similar value as the first $ \dot{\lambda} $. Therefore we have self-dual $\R \llbracket G_{F, \Sigma}\rrbracket$-module $ \T_1 \coloneqq T_\f \widehat{\otimes}_O \R_\g(\lambda) \cong T_\f(\lambda) \otimes_{i_\g^*} \R  $ of rank two and $ \T_2 \coloneqq T_\f \widehat{\otimes}_{O} \ad^0(T_{\g})(\lambda) $ of rank six where $ \ad^0(T_{\g}) $ is the trace zero matrices of $ \ad(T_\g) $ and $ \lambda = \dot{\lambda}(-1) $. Furthermore, we get that	$ \bs{\rho}_3 $ is decomposable whenever $ \f \text{ and } \g  $ both have a trivial central character. Suppose $ \operatorname{tr} $ is the map explicitly given by 
	$$ \operatorname{tr} \colon T_\f \widehat{\otimes}_{O} \ad(T_{\g})(\lambda) \xrightarrow{\operatorname{id} \otimes \operatorname{trace}} T_\f \widehat{\otimes}_O \R_g(\lambda). $$ 
	We then have an exact sequence
	\begin{equation}
		\label{mainses}
		0 \rightarrow \T_2 \xrightarrow{ \iota} \T_3 \xrightarrow{ \operatorname{tr}} \T_1 \rightarrow 0.
	\end{equation}
	where $ \iota \coloneqq \operatorname{id_{\f} \otimes i} \colon  T_\f \widehat{\otimes}_{O} \ad^0(T_{\g})(\lambda) \hookrightarrow T_\f \widehat{\otimes}_{O} \ad(T_{\g})(\lambda)$ is the canonical embedding. The self-duality of $ \T_1, \T_2 $ and $ \T_3 $ gives us
	\begin{equation}
		\label{mainses_dual}
		0 \rightarrow \T_1 \xrightarrow{\tr} \T_3 \xrightarrow{\pi_{\tr}} \T_2 \rightarrow 0
	\end{equation}
	where the map $ \tr $ is given by transposing the map $ \operatorname{tr} $ and using the self-duality of the modules. Both exact sequences \eqref{mainses} and \eqref{mainses_dual} are split. We denote the representations as $ \bs{\rho_1} $ and $ \bs{\rho_2} $ associated to $ \T_1 $ and $ \T_2 $ respectively.
	\begin{rmk}
		\label{Remark:1.15}
		Suppose $ \omega_{\lambda} $ be a the adelic character of conductor $ \frak{N}_\lambda $ associated with $ \lambda $. Then we can get a Hida family $ \f \otimes \omega_{\lambda} $ such that $ (\f \otimes \omega_{\lambda})_P = \f_P \otimes \omega_{\lambda} $ for all $ P \in \W_{\mrm{cl}}(\R_{\f}) $. We also have $ T_{\f \otimes \omega_{\lambda}} \cong T_\f(\lambda) $ for associated Galois representations \cite[\textsection7F]{Hida91}. We will denote this family as $ \f \otimes \lambda $ for simplicity.
	\end{rmk}
	
	\subsection{The Artin Formalism}
	\label{sec_artin_formalism}
	Let $ (P,Q) \in \W_{\mrm{cl}} $ be a pair of crystalline points. Hence we can get 
	\begin{equation*}
		f \in S_{k(P),w(P)}(\frak{N}_f, \psi_f) \text{ and } g \in S_{k(Q),w(Q)}(\frak{N}_g, \psi_g)
	\end{equation*}
	such that $f = \f_P $ and $ g =  \g_Q $. Thus we can get an automorphic representation $ g^\theta $ as the base change of $ g $ by the field extension $ F/\Q $. Let us denote $ \operatorname{As}(g)$ be the automorphic representation  $ g \otimes g^\theta $ on $ \cal G $. 
	\subsubsection{$ L $-functions and Root numbers}
	The complete Rankin-Selberg $ L $-function satisfies the functional equation
	\begin{equation}
		\label{fuctional_triple}
		L(f \otimes  \operatorname{As}(g),s) = \varepsilon(f \otimes  \operatorname{As}(g))N(f \otimes  \operatorname{As}(g))^{-s}L(f \otimes  \operatorname{As}(g),c(f \otimes  \operatorname{As}(g))-s)
	\end{equation} 
	for the global root number $ \varepsilon(f \otimes  \operatorname{As}(g)) \in \{\pm1\} $, some constant $ N(f \otimes  \operatorname{As}(g)) \in \Z_{>0} $ and the central critical value $ c(f \otimes  \operatorname{As}(g)) $.
	Suppose $ \omega_\lambda $ is the associated quadratic adelic character of $ \operatorname{Res}_{F/\Q}\bb{G}_m $ to $ \lambda $.	Then according to our assumption $ g $ has a trivial central character and is Asai decomposable i.e $ g^\theta \cong g \otimes \omega_\lambda $. Therefor we have \emph{Artin formalism} (in other words the factorization of $ L $-functions)    
	\begin{equation}
		\label{artin_formalism}
		L(f \otimes  \operatorname{As}(g),\cdot) = L(f \otimes \breve{g} ,\cdot) L(f \otimes \omega_\lambda, \cdot )
	\end{equation}
	where $ \breve{g} = \ad^0 g \otimes \omega_\lambda  $. We also have a factorization of global root numbers
	\begin{equation}
		\label{factorization_root_numbers}
		\varepsilon(f \otimes  \operatorname{As}(g))= \varepsilon(f \otimes \breve{g})\varepsilon(f \otimes \omega_\lambda).
	\end{equation}
	Note that following Remark \ref{Remark:1.15}, we have $ f \otimes \omega_\lambda =\f_P \otimes \lambda $. 
	\begin{rmk}
		The global root numbers are defined as the product of local root numbers. The equation \eqref{factorization_root_numbers} also holds in the local cases. In other words we have
		\begin{align*}
			\varepsilon(f \otimes  \operatorname{As}(g))= \varepsilon(f \otimes \breve{g})\varepsilon(f \otimes \omega_\lambda), \quad \forall v \mid \frak N_f\frak N_g \infty.
		\end{align*}		
	\end{rmk}
	\subsubsection{$ p $-adic Artin formalism}
	Our main goal in this present article is to establish some $ p $-adic analogues of the factorization \eqref{artin_formalism} in the algebraic setup. Here we briefly discuss the associated $ p $-adic $  L$-functions. 
	For example in \cite[Theorem 1.2]{salazar2019tripleproductpadiclfunctions}, a triple product $ p $-adic $ L $-functions
	$$\mathscr{L}_p^{\spadesuit}(\f \otimes \As(\g)) \colon \W \rightarrow \bb{C}_p ,$$
	satisfying interpolation properties
	\begin{equation}
		\mathscr{L}_p^{\spadesuit}(\f \otimes \As(\g))(P,Q)\doteq\dfrac{L(f \otimes \As(g))}{\Omega^{\spadesuit}_{\f \otimes \As(\g)}}, \quad (P,Q) \in \W \cap \W_{\mrm{cl}}^{\spadesuit},
	\end{equation} 
	where $ \Omega^{\spadesuit}_{\f \otimes \As(\g)} $ is the associated period, $ \spadesuit \in \{(\f,\b),(\b,\f)\} $ (the interpolation factors are precisely defined in \cite{salazar2019tripleproductpadiclfunctions} and \cite{Molina21}). In this philosophy, we can conjecturally say that (in a broad sense) there exist $ p $-adic $ L $-functions
	$$\mathscr{L}_p^{\bullet}(\f \otimes \As(\g)) \colon \W \rightarrow \bb{C}_p ,$$
	satisfying interpolation properties
	\begin{equation}
		\mathscr{L}_p^{\bullet}(\f \otimes \As(\g))(P,Q)\doteq\dfrac{L(f \otimes \As(g))}{\Omega^{\bullet}_{\f \otimes \As(\g)}}, \quad (P,Q) \in \W \cap \W_{\mrm{cl}}^{\bullet},
	\end{equation} 
	for the suitable period $ \Omega^{\bullet}_{\f \otimes \As(\g)} $, $ \bullet = \f, \b $. In this manner we expect conjectural $ p $-adic $ L $-functions of this forms
	\begin{equation}
		\mathscr{L}_p^{\bullet}(\f \otimes\breve{\g})(P,Q) \doteq \dfrac{L(\f_P \otimes \breve{\g}_Q)}{\Omega^{\bullet}_{\f \otimes \breve{\g}}}, \quad  (P,Q) \in \W \cap \W_{\mrm{cl}}^{\bullet},
	\end{equation}
	for some suitable period $ \Omega^{\bullet}_{\f \otimes \breve{\g}} $ and $ \bullet \in \{\f,\b, (\f,\b), (\b,\f) \} $. We also have the $ p $-adic $ L $-functions for Hida family of cuspforms $ \mathscr{L}_p(\f \otimes \lambda) $ (see \cite[Theorem 0.1]{Dimitrov13} or \cite[\textsection 8.2]{bergdall2021padiclfunctionshilbertmodular}). Hence we have a conjecture
	\begin{conj} 
		In full generalities, we should have a factorization like
		$$\mathscr{L}_p^{\bullet}(\f \otimes \As(\g)) = \mathscr{L}_p^{\bullet}(\f \otimes\breve{\g})\cdot\mathscr{L}_p(\f \otimes \lambda).$$	
	\end{conj}
	\subsubsection{Root numbers and vanishing of $ L $-functions} In this article, we are interested in the case when $ \varepsilon(\f_P \otimes  \lambda) = 0 $. The interpolative properties suggest that 
	$$ \mathscr{L}_p(\f_P \otimes \lambda) = 0 .$$
	In the present paper, we mainly deal with two situations as given by the following data satisfying \eqref{factorization_root_numbers} 
	\begin{align}
		\label{table_rootnumber}
		\begin{tabular}{|c|c|c|}
			\hline
			& $\varepsilon(\f_P \otimes \As(\g_Q))$    & {$\varepsilon(\f_P\otimes \breve{\g}_Q)$}   \\
			\hline\hline
			$(P,Q)\in \W^{\f}$ & $-1$    & \multicolumn{1}{c|}{+1} \\ \hline
			$(P,Q)\in \W^{(\f,\b)}$  & $+1$   & \multicolumn{1}{c|}{$-1$} \\ \hline
		\end{tabular}
	\end{align}
	We remark that for $ (P,Q)\in \W^{(\b,\f)} $ the problem can be addressed in a manner similar to in the similar manner of $ (P,Q)\in \W^{(\f,\b)} $. 
	\subsubsection{}
	We can further say that for $ (P,Q) \in \W_{\mrm{cl}}^{\b} $ one can predict $ \varepsilon(\f_P \otimes \As(\g_Q) ) =- 1 $ (cf. \cite{DR22,BSV22b}). Hence via \eqref{artin_formalism} we can infer that either $ \mathscr{L}_p^\b(\f \otimes \breve{\g} ) =0 $ (if $\varepsilon(\f\otimes \check{\varphi})=+1$)
	or $ \mathscr{L}_p(\f \otimes \lambda) = 0 $ (if $\varepsilon(\f\otimes \check{\varphi})=-1$). 
	In all cases, we have the rather uninteresting factorization of the totally balanced $ p $-adic $ L $-functions
	$$ \mathscr{L}_p^{\b}(\f \otimes \As(\g) ) = 0 = \mathscr{L}_p^{\b}(\f \otimes \breve{\g} )\cdot \mathscr{L}_p(\f \otimes \lambda) .$$
	\subsection{Factorization of $\scr{L}^{(\g,\b)}_p(\f \otimes \As(\g))$ in algebraic setup}
	\label{rank1_scenarion}
	The factorization of the $ \f $-dominant at $ \tau_1 $ $ p $-adic $ L $-functions
	\begin{equation}
		\label{f-dominant_factorization}
		\mathscr{L}_p^{(\f,\b)}(\f \otimes \As(\g) )  = \mathscr{L}_p^{(\f,\b)}(\f \otimes \breve{\g} )\cdot \mathscr{L}_p(\f \otimes \lambda)
	\end{equation}
	reduces, once we are provided with the expected interpolation properties of various $ p $-adic $ L $-functions, to a comparison of periods at the specializations of both sides at $ (P,Q) \in \W_{\mrm{cl}}^{(\f,\b)} $.We are addressing the problem when $ \varepsilon(\f \otimes \lambda) = -1$. In this scenario, at the central critical value $ c = c(P,Q) $, \eqref{artin_formalism} and \eqref{table_rootnumber} give 
	\begin{equation*}
		\operatorname{ord}_{s=c}L(\f_P \otimes \As(\g)_Q) \ge 2, \quad \text{ if } (P,Q) \in \W_{\mrm{cl}}^{(\f,\b)}
	\end{equation*}
	In particular, we get by taking a derivative at the central critical value
	\begin{equation}
		\frac{d^2}{ds^2}	L(\f_P \otimes  \operatorname{As}(\g_Q),s) =\frac{d}{ds} L(\f_P \otimes \breve{\g}_Q ,s)\frac{d}{ds} L(\f_P \otimes \lambda, s )
	\end{equation}
	Inspired by the work of \cite{BCS23} and the fundamental result in \cite{BDP13,DR22,BSV22b}, the guiding principle of our project is
	\begin{itemize}
		\item[ \mylabel{BDP1}{\textbf{(BDP1)}}] the $p$-adic $L$-function $\scr{L}^{(\g,\b)}_p(\f \otimes \As(\g))$ should be thought of as a $p$-adic avatar of the family of second-order derivatives 
		$$ \left\lbrace 	\frac{d^2}{ds^2}	L(\f_P \otimes  \operatorname{As}(\g_Q),s)\mid_{s=c} \colon (P,Q) \in \W_{\mrm{cl}}^{(\f,\b)} \right\rbrace ; $$ 
		\item[ \mylabel{BDP2}{\textbf{(BDP2)}}] the $p$-adic $L$-function $\scr{L}^{(\g,\b)}_p(\f \otimes \breve{\g})$ should be thought of as a $p$-adic avatar of the family of first order derivatives 
		$$ \left\lbrace \frac{d}{ds} L(\f_P \otimes \breve{\g}_Q ,s)\mid_{s=c} \colon (P,Q) \in \W_{\mrm{cl}}^{(\f,\b)} \right\rbrace . $$
	\end{itemize}
	\begin{rmk}
		\label{BDP_philosophy}
		The general philosophy of \emph{Bertolini-Darmon-Prasanna} in \cite{BDP13} the \emph{BDP} $p$-adic $L$-function as a $p$-adic avatar of the derivatives
		of complex analytic $L$-functions at the central critical points using the (generalized) Heegner cycles. Using this circle of ideas, the authors in \cite{BSV22b,DR22} suggest that the unbalanced $p$-adic $L$-functions should be considered as $p$-adic avatars of the derivatives of the Garrett–Rankin L-series at the central critical point. Our forthcoming work will be about further pursuing this philosophy in a totally real number field setup.
	\end{rmk}
	Based on the principles \ref{BDP1} and \ref{BDP2}, we propose the following refinement of the factorization \eqref{f-dominant_factorization}:
	\begin{conj}
		\label{conjecture_1}
		Suppose that $\varepsilon(\f \otimes \lambda)=-1$ and the root number condition holds, we have the following factorization of $p$-adic $L$-functions (up to some $ p $-adic meromorphic function)
		\begin{equation}
			\label{g-dominant_factorization_analytic}
			\scr{L}_p^{(\g,\b)}(\f \otimes \As(\g) )(P,Q)  = \scr{L}_p^{(\g,\b)}(\f \otimes \breve{\g} )(P,Q)\cdot i^*_\f\left(\Log_{\f \otimes \lambda}^1(\mathbf{BK}_{\f\otimes \lambda}^{\rm LZ}) \right) .
		\end{equation} 
		
	\end{conj}
	Here $ \mathbf{BK}^{\rm LZ}_{\f\otimes \lambda} $ is the Kolyvagin system coming from the Euler system constructed by Loeffler and Zerbes in \cite{loeffler2020iwasawatheoryquadratichilbert}. At this moment the conjecture is out of reach due to the scarcity of information related to the explicit \textit{BDP} formulas regarding quadratic Hilbert modular forms and $ p $-adic $ L $-functions in use. However we show the algebraic counterpart of the Conjecture \ref{conjecture_1} that is compatible with the theme of Iwasawa main conjecture, in the sense that we have defined the (higher rank) \emph{modules of leading term} $ \bbdelta_r(X,\Delta) $ (concerning the relevant \emph{Selmer complexes} and their \emph{core rank}) in \cref{sec : module_leading} which can be thought of as an algebraic analogues of $ p $-adic $ L $-functions.
	\begin{introtheorem}[\cref{main_theorem_1} below]
		\label{theoremA}
		Let the root number conditions in Conjecture \ref{conjecture_1} hold. Suppose that the module of (algebraic) $ p $-adic $ L $-functions $ \bbdelta_0(\T_{3}, \Delta_{(\g, \b)}) $ is non vanishing. Under certain hypotheses we have
		\begin{equation}
			\label{g-dominant_factorization}
			\bbdelta_0(\T_{3}, \Delta_{(\g, \b)})  =  \Exp_{\f \otimes\lambda, \p_1}\left( \bbdelta_1(\T_2, \operatorname{tr}^*\Delta_{(\g,\b)})\right)\cdot i^*_\f\left(\Log_{\f \otimes \lambda}^1(\mathbf{BK}^{\rm LZ}_{\f\otimes \lambda})\right) .
		\end{equation} 
	\end{introtheorem}
	Proving this conjecture is the main content of \cref{section : factorization_1}.
	\subsection{Factorization of $\scr{L}^{\g}_p(\f \otimes \As(\g))$ in algebraic setup} 
	\label{rank2_scenario}
	Now we focus on the case $ (P,Q) \in \W_{\mrm{cl}}^{\f} $. Here we have a hypothesis that $  L(\f_P \otimes \lambda) $ has zero of order two and $ L(\f_P \otimes \check{\g_Q}) $ does not vanish at the central critical value. Therefore we get from \eqref{artin_formalism}
	\begin{equation*}
		\operatorname{ord}_{s=c}L(\f_P \otimes \As(\g)_Q) = 2, \quad \text{ if } (P,Q) \in \W_{\mrm{cl}}^{\f}
	\end{equation*}
	Similarly, we take our guiding principle as
	\begin{itemize}
		\item[ \mylabel{BDP3}{\textbf{(BDP3)}}] the $p$-adic $L$-function $\scr{L}^{\g}_p(\f \otimes \As(\g))$ should be thought of as a $p$-adic avatar of the family of second order derivatives 
		$$ \left\lbrace 	\frac{d^2}{ds^2}	L(\f_P \otimes  \operatorname{As}(\g_Q),s)\mid_{s=c} \colon (P,Q) \in \W_{\mrm{cl}}^{\f} \right\rbrace ; $$ 
	\end{itemize}
	We assume that $ \f_P \otimes \lambda $ satisfies the Plectic conjecture. In particular, if we have a rank-two Kolyvagin system (cf. \cite{BSS18}) $ \mathbf{BK}^{\rm Plec}_{\f\otimes \lambda,1} \wedge \mathbf{BK}^{\rm Plec}_{\f\otimes \lambda,2} $ coming from rank two Euler system generated by generalized Kato classes. We remark that these types of rank two generalized Kato classes are more natural to appear in real quadratic number field setup (cf. Remark \ref{remark_panchiskin}). For more details kindly check Remark \ref{gen_kato_class}. Hence similar to \cref{conjecture_1} we can get the following conjecture assuming that $\varepsilon(\f \otimes \lambda)=-1$ and the root number condition holds true
	\begin{conj}
		\label{conjecture:3}
		We have the following factorization of $p$-adic $L$-functions
		\begin{equation}
			\scr{L}_p^{\g}(\f \otimes \As(\g) )(P,Q)  = \scr{L}_p^{\b}(\f \otimes \breve{\g} )(P,Q)\cdot i^*_\f\left(\Log^2_{\f \otimes \lambda}(\mathbf{BK}^{\rm Plec}_{\f\otimes \lambda,1} \wedge \mathbf{BK}^{\rm Plec}_{\f\otimes \lambda,2} )\right) .
		\end{equation} 
	\end{conj}
	\begin{rmk}
		Though the existence of $ \mathbf{BK}^{\rm Plec}_{\f\otimes \lambda,1} \wedge \mathbf{BK}^{\rm Plec}_{\f\otimes \lambda,2}$ is conjectural, but there are some recent works towards the evidence of the higher rank Euler system. In \cite{fornea23plectic}, the authors introduced \emph{plectic} Stark-Heegner points for an elliptic curve $ A/F $ given that $ \lim_{s\to 1 }L(A/F,s) = 2 $ which is related to the local image of $ \mathbf{BK}^{\rm Plec}_{\f\otimes \lambda,1} \wedge \mathbf{BK}^{\rm Plec}_{\f\otimes \lambda,2}$ in a parallel weight 2 specializtion. Also the \enquote{mock plectic points} in \cite[]{Darmon_Fornea_2025}, may be seen an oblique evidence for the plectic philosophy of Nekovář and Scholl, using the mock analogue of Hilbert modular forms.  
	\end{rmk}
	In the view of last remark we can observe that \cref{conjecture:3} is almost unaccessible by the informations we get from recent literature. Nevertheless we have shown the algebraic analogue of \cref{conjecture:3} in \cref{section : factorization_2}.
	\begin{introtheorem}[\cref{main_theorem_2} and Corollary \ref{cor_main_2} below]
		\label{theoremB}
		Let the root number conditions in Conjecture \ref{conjecture:3} hold. Suppose that the module of (algebraic) $ p $-adic $ L $-functions $ \bbdelta_0(\T_{3}, \Delta_{\g}) $ is non vanishing. Under certain hypotheses we have
		\begin{equation}
			\label{rank2_factorization}
			\bbdelta_0(\T_{3}, \Delta_{\g})  =   \bbdelta_0(\T_2, \operatorname{tr}^*\Delta_{\b})\cdot i^*_\f\left(\Log^2_{\f \otimes \lambda}(\mathbf{BK}^{\rm Plec}_{\f\otimes \lambda,1} \wedge \mathbf{BK}^{\rm Plec}_{\f\otimes \lambda,2} )\right).
		\end{equation} 
	\end{introtheorem}
	
	Prooving \cref{theoremA} and \cref{theoremB} is the main part of this paper. In the next section, we discuss about Nekovar's theory of Selmer complexes to define the (algebraic) $p$-adic $L$-functions more precisely.

	\section{Selmer Complexes}
	\label{section : Selmer_complexes}
	We now introduce the algebraic counterparts of the ($ p $-adic) analytic objects briefly discussed in the \textsection \ref{sec_artin_formalism}, \textsection \ref{rank1_scenarion}, and \textsection \ref{rank2_scenario}. We will formulate the first step of algebraic factorization in this section. The ETNC \footnote{abbrv. Equivariant Tamagawa Number Conjecture.} philosophy will be our general guidance (cf \cite[Appendix A]{BCS23}).
	\subsection{Greenberg Conditions}
	\label{sec:greenberg}
	For each $ X \in \{\T_1, \T_2,\T_3 \} $ as in \textsection\ref{subsection:Decomposition of Big Galois Modules}, with the fact $ \f $, $ \g $ are $p$-ordinary families, with the conditions \ref{Irr} and \ref{DIST} gives rise to short exact sequences of $ G_{F_\p} $-modules (for $ \p = \p_1,\p_2 $) of the form
	\begin{equation}
		\label{onestep_filtration}
		0 \rightarrow \scrF^+_{\p} X \rightarrow X \rightarrow \scrF^-_{\p} X \rightarrow 0 
	\end{equation}
	We denote $ \bP\coloneqq \{\p_1,\p_2\} \subset \Sigma $, $ I_v \subset G_{F_v} $ as the inertia subgroup and $ \operatorname{Fr}_v \in G_{F_v}/I_v $ as the (geometric) Frobenius element. Let's define the following local conditions for a given short exact sequences like \eqref{onestep_filtration}(in the sense of \cite{Nek06})
	$$ 
	U_{v}^+(X) =
	\begin{cases}
		C^{\bullet}(G_{F_\p},\scrF^+_{\p} X), & \p \in \bP; \\
		C^{\bullet}(G_{F_v}/I_v,X^{I_{F_{v}}}), & v \in \Sigma \ba \bP ;
	\end{cases}
	$$
	along with a morphism of complexes 
	$$ i_{v}^+ \colon U_{v}^+(X) \rightarrow C^{\bullet}(G_{F_v},X) \text{, for all } v \in \Sigma. $$
	For each $ v \in \Sigma\ba \bP $, we have a quasi-isomorphism of complexes 
	$$ U_{v}^+(X) \cong \left[ X^{I_v} \xrightarrow{\fr_v - 1} X^{I_v} \right], $$
	concentrated in degrees 0 and 1.\\
	We note that even for a given $ X $, there exists a number of choices of one-step filtration such as \eqref{onestep_filtration}, which should be thought of as a reflection of the multiple-choice of $ p $-adic $ L $-functions, depending on the choice of interpolation. \\
	Let $ v \in \Sigma\ba \bP $ and suppose that 
	$$ 0 \rightarrow M_1 \rightarrow M_2 \rightarrow M_3 \rightarrow 0, $$
	is a split short exact sequence continuous $ G_{F_v} $-representations of free finite rank over a complete local Noetherian ring $ R $. Then we have 
	\begin{equation}
		U_{v}^+(M_2) = U_{v}^+(M_3) \oplus U_{v}^+(M_1).
	\end{equation}
	\subsubsection{Important local conditions}The definition goes as follow
	\begin{defn}
		The data $$ \Delta\left( \left\lbrace \scrF^+_{\p}X\right\rbrace_{\p \in \bP} \right)  \coloneqq \{i_{v}^+ \colon U_{v}^+(X) \rightarrow C^{\bullet}(G_{F_v},X) \text{, for all } v \in \Sigma\}$$ is called a Greenberg-local condition on $ X $.
	\end{defn}
	In the following examples, we shall record some useful Greenberg-local conditions for our purpose.  
	\begin{exmp}
		Suppose $ X = \T_3 $. For $ \p = \p_1, \; \p_2 $, lets define
		\begin{enumerate}[(i)]
			\item $\f$-dominated Greenberg-local condition $ \Delta_{\mathbf{f}} = \Delta (\scrF^{+\mathbf{f}}_{\p_1}X,\scrF^{+\mathbf{f}}_{\p_2}X) $ is 
			\begin{equation*}
				\scrF^{+\mathbf{f}}_{\p}\T_{3} \coloneqq \scrF^{+}_{\p}T_{\f} \widehat{\otimes} \ad(T_{\g})(\lambda)  \hookrightarrow \T_3. 
			\end{equation*}
			\item Balanced Greenberg-local condition $\Delta_{\mathrm{bal}} = \Delta (\scrF^{+\mathrm{bal}}_{\p_1}X,\scrF^{+\mathrm{bal}}_{\p_2}X) $ is
			\begin{equation*}
				\scrF^{+\mathrm{bal}}_{\p}\T_{3} \coloneqq \scrF^{+}_{\p}T_{\f} \widehat{\otimes} \mathrm{Fil}^3_{\p} \ad(T_\g)(\lambda)   +T_{\f} \widehat{\otimes} \mathrm{Fil}^1_{\p}\ad(T_\g)(\lambda) \hookrightarrow \T_3.
			\end{equation*}
			\item $ \g $-dominated Greenberg-local condition $\Delta_{\g} = \Delta(\scrF^{+\g}_{\p_1}X, \scrF^{+\g}_{\p_2}X) $ defined as
			\begin{equation*}
				\scrF^{+\g}_{\p}\T_{3} \coloneqq T_{\f} \widehat{\otimes} \mathrm{Fil}^2_{\p}(\ad(T_\g)(\lambda)) \hookrightarrow \T_3.
			\end{equation*}
			\item We will also consider the Greenberg-local condition $ \Delta_{+}\coloneqq\Delta(\scrF^{+\mathrm{bal}^+}_{\p_1}X, \scrF^{+\mathrm{bal}^+}_{\p_2}X) $ given by the following rank 5 direct submodule of $ X $
			\begin{equation*}
				\scrF^{+\mathrm{bal}^+}_{\p} \T_{3} = \scrF^{+\g}_{\p}\T_{3} + \scrF^{+\mathrm{bal}}_{\p}\T_{3}.
			\end{equation*}
		\end{enumerate}
	\end{exmp}
	\begin{exmp}
		For $ \T_{3} $ we can define more intrinsic (mixed) local conditions
		\begin{enumerate}[(i)]
			\item  $ \Delta_{(\g, \f)} = \Delta (\scrF^{+\mbf{g}}_{\p_1}\T_3, \; \scrF^{+\mathbf{f}}_{\p_2}\T_3) $.
			\item $ \Delta_{(\g, \b)} = \Delta (\scrF^{+\mbf{g}}_{\p_1}\T_3, \; \scrF^{+\b}_{\p_2}\T_3) $.
			\item $ \Delta_{(+,\b)}:=\Delta(\scrF^{+\mathrm{bal}^+}_{\p_1}\T_3, \; \scrF^{+\b}_{\p_2}\T_3) $.
			
		\end{enumerate}
	\end{exmp}		
	\begin{rmk}
		We remark that under the self-duality pairing $ \T_{3} \otimes\T_{3} \rightarrow \R(1) $, the following local conditions discussed above are self-orthogonal complements. 
		$$ \Delta_{\g}^{\perp} = \Delta_{\g}, \quad \Delta_{\b}^{\perp} = \Delta_{\b}, \quad \Delta_{(\g,\b)}^{\perp} = \Delta_{(\g,\b)}.$$
		
	\end{rmk}
	Now we shall propagate these local conditions to $ \T_2 $ and $ \T_{1} $ via the exact sequence \eqref{mainses} and \eqref{mainses_dual}.
	\begin{exmp}
		When $ X = \T_2 $ we define Greenberg local condition $$ \tr \Delta_? \coloneqq \Delta( \{\im(F_\p^{+?}\T_{3} \hookrightarrow \T_{3} \rightarrow \T_{2} )\}_{\p \in \bP}) $$ 
		\begin{enumerate}[label=(\roman*)]
			\item $\f$-dominated Greenberg-local condition $ \tr \Delta_{\f}\coloneqq \Delta(\scrF^{+\mathbf{f}}_{\p}\T_{2}) $ is 
			$$ \scrF^{+\mathbf{f}}_{\p}\T_{2} \coloneqq \scrF^{+}_{\p}T_{\f} \widehat{\otimes} \ad^0(T_{\g})(\lambda) $$
			\item Balanced Greenberg-local condition $ \tr \Delta_{\b}\coloneqq \Delta(\scrF^{+\mathrm{bal}}_{\p}\T_{2}) $ is
			\begin{align*}
				\scrF^{+\mathrm{bal}}_{\p}\T_{2} \coloneqq \scrF^{+}_{\p}T_{\f} \widehat{\otimes} \underbrace{\ker\{\ad^0(T_\g) \hookrightarrow \ad(T_\g) \rightarrow \Hom(F_{\p}^+ T_\g,F_{\p}^- T_\g) \}}_{=:\mathrm{Fil}^{+\g}_{\p} \ad^0(T_\g)}(\lambda) \quad \\
				+  \quad T_\f \widehat{\otimes} \underbrace{\Hom(F_{\p}^- T_\g,F_{\p}^+  T_\g)}_{=:\mathrm{Fil}^{\g}_{\p}\ad^0(T_\g)}(\lambda)
			\end{align*} 
			\item $ g $-dominated Greenberg-local condition $ \tr \Delta_{\g} \coloneqq \Delta(\scrF^{+\g}_{\p}\T_{2}) $ is
			$$ \scrF^{+\g}_{\p}\T_{2} \coloneqq T_{\f} \widehat{\otimes} \underbrace{\ker\{\ad^0(T_\g) \hookrightarrow \ad(T_\g) \rightarrow \Hom(F_{\p}^+ T_\g,F_{\p}^- T_\g) \}}_{=:\mathrm{Fil}^{+\g}_{\p} \ad^0(T_\g)}(\lambda) $$
		\end{enumerate}
		
	\end{exmp}
	Similarly, we can get the local conditions $ \tr \Delta_{(\g, \b)}$,  $\tr \Delta_{(\g, \f)}$, $ \tr \Delta_{(+,\b)} $ etcetera.
	\begin{lemma}
		\label{localses}
		The following hold for any $ \p \in \bP $
		\begin{enumerate}
			\item $ \rank_{\R}\scrF^{+\mathbf{f}}_{\p}\T_{2} = \rank_{\R}\scrF^{+\mathrm{bal}}_{\p}\T_{2} = 3$.
			\item $ \rank_{\R} \scrF^{+\g}_{\p}\T_{2} = 4 $
			\item We have the following short exact sequence of $ G_{F_{\p}} $-representations
			\begin{equation}
				\label{ses_sep18}
				0 \rightarrow \scrF^{+\mathrm{bal}}_{\p}\T_{2} \rightarrow \scrF^{+\g}_{\p}\T_{2} \rightarrow \scrF^-_{\p}T_\f \otimes_{O} \R_{\g}(\lambda) \rightarrow 0. 
			\end{equation}
		\end{enumerate}
		
	\end{lemma}
	\begin{proof}
		Proof of (1) and (2) is immediate from the previous example.\\
		$  \scrF^{+\mathrm{bal}}_{\p}\T_{2} $ is a submodule of $ \scrF^{+\g}_{\p}\T_{2} $ and we get 
		$$ \operatorname{coker}\left( \scrF^{+\mathrm{bal}}_{\p}\T_{2} \hookrightarrow \scrF^{+\g}_{\p}\T_{2} \right) \cong \scrF^-_{\p}T_{\f} \widehat{\otimes}_O \rm gr_{\g}^+ \rm Ad^0 (T_{\g})(\lambda) $$
		where $ \mathrm{gr}_{\g}^+ \ad^0 (T_{\g})(\lambda) \coloneqq \left(\mathrm{Fil}^{+\g}_{\p} \ad^0(T_\g) /\mathrm{Fil}^{\g}_{\p}\ad^0(T_\g)\right)(\lambda) $. Hence we have to show 
		\begin{equation}
			\label{eqn_0709}
			\mathrm{gr}_{\g}^+  \ad^0 (T_{\g}) \cong  \R_{\g}  .
		\end{equation}
		Let $ \{ u^+_\p \} \subset F_{\p}^+T_\g $ denote an $ \R_\g $-basis and let $ \{u^+_\p,u^-_\p \} $ be an extension. Let $ \{\breve{u}^+_\p,\breve{u}^-_\p \} \subset T_\g^* $ be the dual basis. Then 
		\begin{equation}
			\mathrm{Fil}^{+\g}_{\p} \ad^0(T_\g) = \ker\{\ad^0(T_\g) \hookrightarrow \ad(T_\g) \rightarrow \Hom(F_{\p}^+ T_\g,F_{\p}^- T_\g) \} = \left\langle u^+_\p \otimes \breve{u}^-_\p, u^+_\p \otimes \breve{u}^+_\p - u^-_\p \otimes \breve{u}^-_\p  \right\rangle 
		\end{equation}
		and 
		$$ \mathrm{Fil}^{\g}_{\p}\ad^0(T_\g) = \Hom(F_{\p}^- T_\g,F_{\p}^+  T_\g) =\left\langle u^+_\p \otimes \breve{u}^-_\p \right\rangle.  $$
		Let denote the $ \R_\g $-valued character $ \chi_\pm $ of $ G_{F_\p} $ as the action of $ \scrF^\pm_\p T_\g$. Then for any $ g \in G_{F_\p} $
		$$ g \cdot u^+_\p = \chi_+(g)u^+_\p, \quad g \cdot \breve{u}^+_\p = \chi_+(g)^{-1}\breve{u}^+_\p $$
		Hence $ g \cdot (u^+_\p \otimes \breve{u}^+_\p) = u^+_\p \otimes \breve{u}^+_\p $
		and $ g \cdot u^-_\p = \chi_-(g)u^-_\p + c(g)u^+_\p  $ for some $ c(g) \in \R_{\g} $ which gives $ g \cdot \breve{u}^-_\p = \chi_-(g)\breve{u}^-_\p $.
		Combining all of these, we get $ g \cdot ( u^+_\p \otimes \breve{u}^+_\p - u^-_\p \otimes \breve{u}^-_\p) =  u^+_\p \otimes \breve{u}^+_\p - u^-_\p \otimes \breve{u}^-_\p $ in $  \mathrm{gr}_{\g}^+  \ad^0 (T_{\g}) $, which proves \eqref{eqn_0709}.
	\end{proof}
	\begin{defn}
		As we know $ \T_1 \cong T_\f \otimes_O \R_\g(\lambda) $, let's define Greenberg-local conditions on $ \T_1 $ as $$ \tr \Delta_? \coloneqq \Delta(\{\ker(F_\p^{+?}\T_{3} \hookrightarrow \T_{3} \rightarrow \T_{2} )\}_{\p \in \bP}).$$
		So we get 
		\begin{enumerate}[label=(\roman*)]
			\item The filtration for Panchiskin local condition $ \scrF^{+\pan}\T_1 \coloneqq \scrF^{+}_{\p}T_{\f} \otimes_O  \R_\g(\lambda)) $ and we have the following correspondence for $ \p = \p_1, \; \p_2  $  
			$$ \scrF^{+\mathbf{f}}_{\p}\T_1= \scrF^{+}_{\p}T_\f \hotimes_\lambda \R_\g = \scrF^{+\mathrm{bal}}_{\p}\T_1.$$ 
			\item Filtration for strict local condition $$ \scrF^{+0}_{\p}\T_1 =\{0\} =\scrF^{+\g}_{\p}\T_1.$$
		\end{enumerate}
	\end{defn}
	Similarly, we can define the local condition $ \tr\Delta_{\f} = \Delta_\pan $, $ \tr \Delta_{\g} = \Delta_0 $, $ \tr \Delta_{(\g, \b)} = \Delta_{(0,\pan)} $ etcetera of $ \T_1 $.
	\begin{defn}
		We define the (relaxed) Greenberg-local condition on $ X $ with the filtration as $ \scrF^{+\emptyset}_\p X = X $.
	\end{defn}
	
	\subsection{Selmer Complexes associated with Greenberg-local conditions}
	Let $ \Delta $ be a Greenberg-local conditions on $ X $, let us put 
	$$ U_{\Sigma}^{+}(X) \coloneqq \bigoplus_{v \in \Sigma}U_{v}^{+}(X), \quad C^{\bullet}_{\Sigma}(X) \coloneqq \bigoplus_{v \in \Sigma}C^{\bullet}(G_{F_v},X)  $$ 
	and consider the diagram
	\begin{center}
		\begin{tikzcd}
			& U_{\Sigma}^{+}(X) \arrow[d, "i_{\Sigma}^{+}"]  \\
			C^{\bullet}(G_{F, \Sigma},X) \arrow[r, "\res_{\Sigma}"]  & C^{\bullet}_{\Sigma}(X)		
		\end{tikzcd}
	\end{center}
	where $ i_{\Sigma}^{+} \coloneqq \bigoplus_{v \in \Sigma}i_{v}^{+} $ and $ \res_{\Sigma} \coloneqq \bigoplus_{v \in \Sigma}\res_{v} $. \newline
	Let's define the Selmer complex associated with $(X,\Sigma , \Delta )$ 
	$$ \widetilde{C}_{\mrm{f}}^{\bullet}(G_{F,\Sigma},X, \Delta)\coloneqq \operatorname{Cone}(C^{\bullet}(G_{F, \Sigma},X) \oplus U_{\Sigma}^{+}(X) \xrightarrow{\res_{\Sigma}-i_{\Sigma}^{+}} C^{\bullet}_{\Sigma}(X))[-1] $$
	define the corresponding object in derived category by $ \widetilde{R\Gamma}_{\mathrm{f}}(G_{F,\Sigma},X, \Delta) $ and it's cohomology by $ \widetilde{H}_{\mathrm{f}}^{\bullet}(G_{F,\Sigma},X, \Delta) $.
	\begin{thm}
		\cite[Theorem 7.8.6]{Nek06}
		\label{euler_characteristisc}
		We have 
		$$ \chi(\widetilde{R\Gamma}_{\mathrm{f}}(G_{F,\Sigma},X, \Delta)) = \sum_{v \mid \infty}\rank(X^{G_v}) - \sum_{\p \mid p}\rank(\scrF^+_{\p} X) $$
		for the Euler-Poincare characteristic of $ \widetilde{R\Gamma}_{\mathrm{f}}(G_{F,\Sigma},X, \Delta) $.	
	\end{thm}
	\subsubsection{}
	We get the following exact triangle, which can be thought of as the first step of our factorization statements
	\begin{equation}\label{comtriangle1}
		\widetilde{R\Gamma}_{\mathrm{f}}(G_{F,\Sigma},\T_{1}, \Delta_{(0,\pan)}) \xrightarrow{\operatorname{id} \otimes \tr}\widetilde{R\Gamma}_{\mathrm{f}}(G_{F,\Sigma},\T_{3}, \Delta_{(\g,\b)}) \xrightarrow{\operatorname{\pi_{\tr}}} \widetilde{R\Gamma}_{\mathrm{f}}(G_{F,\Sigma},\T_{2}, \tr \Delta_{(\g,\b)}) \xrightarrow[+1]{\delta}  
	\end{equation}
	and
	\begin{align}
		\label{comtriangle}
		\widetilde{R\Gamma}_{\mathrm{f}}(G_{F,\Sigma},\T_{1}, \Delta_{0}) \xrightarrow{\operatorname{id} \otimes \tr}\widetilde{R\Gamma}_{\mathrm{f}}(G_{F,\Sigma},\T_{3}, \Delta_{\g}) \xrightarrow{\operatorname{\pi_{\tr}}} \widetilde{R\Gamma}_{\mathrm{f}}(G_{F,\Sigma},\T_{2}, \tr \Delta_{\g}) \xrightarrow[+1]{\delta}
	\end{align}
	which will be studied extensively later. We also define the following:
	\begin{defn} ---
		\begin{enumerate}
			\item Let $ \res_{/\b, \p_1} $ denote the composite map
			\begin{align}\label{resbal_1}
				\begin{aligned}
					\widetilde{H}^1_{\mathrm{f}}(G_{F,\Sigma},\T_{2}, \tr \Delta_{(\g,\b)}) \xrightarrow{\res_{\p_1}}  H^1(G_{F_{\p_1}},\scrF^{+\g}_{\p_1}\T_{2} ) \xrightarrow{\operatorname{pr}_{/\g,\p_1}}  H^1(G_{F_{\p_1}}, T_\f \widehat{\otimes} \operatorname{gr}^{+\g}_{\p_1}\ad^0(T_\g)(\lambda))   \\
					\rightarrow  H^1(G_{F_{\p_1}}, \scrF^{-}_{\p_1}T_\f \widehat{\otimes} \operatorname{gr}^{+\g}_{\p_1} \ad^0(T_\g)(\lambda)) \xrightarrow[(\ref{localses}.)]{\sim} H^1(G_{F_{\p_1}}, \scrF^{+\g}_{\p_1}\T_{2}/\scrF^{+\b}_{\p_1}\T_{2}),
				\end{aligned}
			\end{align}
			\item Let $ \res_{/\b}$ denote the composite map 
			\begin{align}\label{resbal_2}
				\begin{aligned}
					\widetilde{H}^1_{\mathrm{f}}(G_{F,\Sigma},\T_{2}, \tr \Delta_{\g}) \xrightarrow{\res_p} \bigoplus_{\p \mid p} H^1(G_{F_{\p}},\scrF^{+\g}_{\p}\T_{2} ) \xrightarrow{\operatorname{pr}_{/\g}}  \bigoplus_{\p \mid p} H^1(G_{F_{\p}}, T_\f \widehat{\otimes} \operatorname{gr}^{+\g}_{\p}\ad^0(T_\g)(\lambda))   \\
					\rightarrow \bigoplus_{\p \mid p} H^1(G_{F_{\p}}, \scrF^{-}_{\p}T_\f \widehat{\otimes} \operatorname{gr}^{+\g}_{\p} \ad^0(T_\g)(\lambda)) \xrightarrow[(\ref{localses}.)]{\sim} \bigoplus_{\p \mid p} H^1(G_{F_{\p}}, \scrF^{+\g}_{\p}\T_{2}/\scrF^{+\b}_{\p}\T_{2}),
				\end{aligned}
			\end{align}
			where the maps $ \res_p\coloneqq\bigoplus_{\p \mid p} \res_{\p} $ and $  \operatorname{pr}_{/\g} \equiv \bigoplus_{\p \mid p}\operatorname{pr}_{/\g,\p} $ is the map induced from canonical surjection for all $ \p \mid p $ 
			$$ \scrF^{+\g}_{\p}\T_{2} \xrightarrow{\operatorname{pr}_{/\g,\p}} T_\f \widehat{\otimes}\operatorname{gr}^{+\g}_{\p} \ad^0(T_\g)(\lambda). $$
		\end{enumerate}
		
	\end{defn}
	\subsubsection{Panchiskin defect}
	\label{sec_pan}
	We now introduce the invariant called Panchiskin defect which is necessary to show the perfectness of certain Selmer complex. The foundational concept was introduced by Perrin-Riou and Rubin. Later Kazim B\"{u}y\"{u}kboduk incorporated these ideas into the following articles \cite{kazim1,kazim2,kazim3,kazim4,kazim5,kazim6}. With the help of \cite[Section 7]{LZ20local} we will identify the differences between our present factorization problems with those in \cite{Greenberg1982, Palva18, BCS23}.
	\begin{defn}
		\label{panchiskin_defect}
		Suppose that $ X $ is a free $ O $-module of ﬁnite rank, endowed with a continuous $ \G $-action. We also have that $ X $ is equipped with a Greenberg local conditions $ \Delta \coloneqq \Delta(\scrF_{\p_1}^+X,\scrF_{\p_2}^+X) $ given by $ G_{F_{\p_i}} $-stable submodule $ \scrF_{\p_i}^+X \subset X $ for $ i=1,2$. We define the Panchiskin defect $ \delta_\pan(X,\Delta) $ attached to the pair $ (X, \Delta) $ on setting
		\[\delta_\pan(X,\Delta) = \abs{ \sum_{v \mid \infty}\rank(X^{c_v =1}) - \sum_{i=1,2}\rank(\scrF^+_{\p_i} X)} \]
		where $ X^{c_v =1} \subset X $ is the $ +1 $-eigenspace of complex conjugation at $ v $. We say that the pair $ (X, \Delta)$ is $ r $-Panchiskin if $ \delta_\pan(X,\Delta) = r $.
	\end{defn}
	Note that $ X $ is Panchishkin ordinary if there exist a Greenberg local condition $ \Delta $ such that $ (X,\Delta) $ is $ 0 $-Panchiskin and all Hodge–Tate weights of $ \scrF_{\p_i}^+X $ (resp. of $ X/\scrF_{\p_i}^+X $) are positive (resp. non-positive) for $ i =1,2 $.
	\begin{exmp}
		---
		\begin{enumerate}[(i)]
			\item Let $ x \in \W_{\mrm{cl}}^{(\f,\b)}(O) $ be an $ O $-valued arithmetic point and let us denote $X(x) := X \otimes_x O$ for $ X = \T_{1}, \T_{2}, $ and $ \T_{3} $. Then
			\begin{align*}
				\delta_\pan(\T_{3}(x), \Delta_{(\g, \b)}) &=0; \\
				\delta_\pan(\T_{2}(x), \tr\Delta_{(\g, \b)}) &=1 = \delta_\pan(\T_{1}(x), \tr\Delta_{(\g, \b)}).
			\end{align*}
			\item Let $ x \in \W_{\mrm{cl}}^{\f}(O) $ be an $ O $-valued arithmetic point and let us denote $X(x) := X \otimes_x O$ for $ X = \T_{1}, \T_{2}, $ and $ \T_{3} $. Then
			\begin{align*}
				\delta_\pan(\T_{3}(x), \Delta_{\g}) &=0= \delta_\pan(\T_{2}(x), \Delta_{\b}) ; \\
				\delta_\pan(\T_{2}(x), \tr\Delta_{\g}) &=2= \delta_\pan(\T_{1}(x), \tr\Delta_{\g}).
			\end{align*}
		\end{enumerate}
	\end{exmp}
	\begin{rmk}
		\label{remark_panchiskin}
		Using the Panchiskin defect we can indicate why the factorization problem for algebraic $p$-adic $L$-functions we consider in this article is considerably more challenging than those problems considered earlier in \cite{Greenberg1982,Palva18,BCS23}. All the factorizations result in op. cite are obtained when the associated pairs are either $ 0 $-Panchiskin or $ 1 $-Panchiskin while we also consider the $ 2 $-Panchiskin pairs (cf. \cite[\textsection 4.3.2]{BCS23}). This discussion also gives evidence towards the existence of higher rank Euler systems assumed in \cref{section : factorization_2} (cf. \cite{LZ20local}). 
	\end{rmk}
	\subsection{Remarks on Tamagawa Factors}
	\label{sec:tamagawa}
	Let $ v \nmid p\infty $ be a place in $ F $. Suppose $ \varphi \colon G_{F_v} \rightarrow \GL{R}(V) $ be a continuous $ G_{F_v} $-representation, where $ V $ is a free module of finite rank over a complete Noetherian local ring $ R $ with finite residue field of characteristic $ p $. 
	\begin{rmk}
		When $ R = \Z_p $, $ F = \Q $, the $ p $-adic valuation of the order of $ H^1(I_{v},V)^{\mathrm{Fr}_{v}=1}_{R-\mathrm{tors}} $ is the local Tamagawa factor at $ v $ (see \cite[\textsection I.4.2.2]{FPR94}). So it vanishes if $ H^1(I_{v},V) $ is a free $ R $-module. In this philosophy, we would concentrate on the following conditions:
		\begin{itemize}
			\item The $ R $-module $ H^1(I_{v},V) $ is free.
		\end{itemize}
	\end{rmk}
	This discussion about the Tamagawa factor is needed for verification that our Selmer complexes are perfect. Our arguments are inspired by the observations made by Büyükboduk et al. in \cite{BCS23}.
	\subsubsection{}
	Let $ I_{v}^{(p)} \triangleleft I_v $ be the unique subgroup satisfying $ I_v/I_{v}^{(p)} \cong \Z_p$. Let us fix a topological generator $ t \in I_v/I_{v}^{(p)} $. (cf. \cite[\textsection 7.5.2]{NSW08})
	\begin{lemma} 
		\label{lemma : 2.9}
		---
		\begin{enumerate}[(1)]
			\item The $R$-module $V^{I_{v}^{(p)}}$ is free of finite rank.
			\item The complex $C^{\bullet}(I_v , V )$ is quasi-isomorphic to the perfect complex
			$$ \cdots \rightarrow 0 \rightarrow V^{I_{v}^{(p)}} \xrightarrow{t-1} V^{I_{v}^{(p)}} \rightarrow 0 \rightarrow \cdots  $$ concentrated in degrees $0$ and $1$.
			\item For any ring homomorphism $R \rightarrow S$, the induced map $ V^{I_{v}^{(p)}} \otimes_R S \rightarrow (V \otimes_R S)^{I_{v}^{(p)}}$ is an isomorphism. Therefore in the derived category 
			\begin{align*}
				R\Gamma(I_{v},V) \otimes^{\mathbb{L}}_R S \xrightarrow{\sim}  R\Gamma(I_{v},V \otimes_R S)
			\end{align*}
		\end{enumerate}
	\end{lemma}
	\begin{proof}
		The statement (1) and (2) follows from \cite[\textsection 7.5.8]{Nek06}. \\
		To prove (3), \cite[\textsection 7.5.8]{Nek06} gives the inclusion $V^{I_v^{(p)}} \to V$ is split, the $R$-module $V/V^{I_v^{(p)}}$ is flat. hence we have an exact sequence of $S$-modules  
		\[
		0 \longrightarrow   V^{I_v^{(p)}}\otimes_{R} S \longrightarrow V \otimes_{R} S \longrightarrow V/V^{I_v^{(p)}} \otimes_{R} S \longrightarrow 0. 
		\]
		Moreover, the group $I_v^{(p)}$ acts trivially on the ring $S$. Since $\varphi(I_v^{(p)})$ is finite and $p \nmid \#\varphi(I_v^{(p)})$, we have 
		\[
		e := \frac{1}{\# \varphi(I_v^{(p)})}\sum_{g \in \varphi(I_v^{(p)})}g \in O[\varphi(I_v^{(p)})]
		\] 
		Therefore we have 
		\[
		(V/V^{I_v^{(p)}} \otimes_{R} S)^{I_v^{(p)}} = e(V/V^{I_v^{(p)}} \otimes_{R} S) = e(V/V^{I_v^{(p)}}) \otimes_{R} S = 0\,. 
		\]
		This shows that the canonical homomorphism $V^{I_v^{(p)}} \otimes_R S \to (V \otimes_R S)^{I_v^{(p)}}$ is an isomorphism.  
	\end{proof}
	
	\begin{cor}
		\label{cor : 2.10}
		Suppose $ \varphi(I_{v}) $ is finite and $ p \nmid \#\varphi(I_{v}) $ then $ R $-module $ H^1(I_{v},V) $ is free.
	\end{cor}
	\begin{proof}
		Follows from Lemma \ref{lemma : 2.9} since in this scenario we have isomorphism of $ R $-modules $ H^1(I_v, V) \cong V^{I_v^{(p)}} $.
	\end{proof}
	\begin{prop}
		\label{prop_unramified_complex}
		Assume that the $R$-module $H^1(I_v, V)$ is free. Then the following are true. 
		\begin{enumerate}
			\item[(1)] $V^{I_v}$ is a finite rank free $R$-module.
			\item[(2)] The complex $U_v^+(V)$ is a perfect complex, and $R\Gamma_{\mathrm{ur}}(G_v, V) \in D^{[0,1]}_{\rm parf}(_R\mathrm{Mod})$.  
			\item[(3)]  For any ring homomorphism $R \to S$, we have an isomorphism $V^{I_v} \otimes_R S \to (V \otimes_R S)^{I_v}$. 
			In particular, $R\Gamma_{\mathrm{ur}}(G_v, V) \otimes^{\mathbbl{L}}_{R} S \stackrel{\sim}{\longrightarrow} R\Gamma_{\mathrm{ur}}(G_v, V \otimes_R S)$. 
		\end{enumerate}
	\end{prop}
	\begin{proof}
		---
		\begin{enumerate}[(1)]
			\item From Lemma \ref{lemma : 2.9} we get $H^1(I_v, V) = V^{I_v^{(p)}}/ (t-1)V^{I_v^{(p)}} $ which is free by assumption. Hence $ V^{I_v} = (V^{I_v^{(p)}})^{t=1} $ is also free.
			\item Follows from (1).
			\item Since $ R\Gamma(I_v,V) $ is a perfect complex represented by
			\[\cdots\rightarrow 0 \rightarrow V^{I_v^{(p)}} \xrightarrow{0} H^1(I_v,V) \rightarrow 0 \cdots \]
			Lemma \ref{lemma : 2.9}.3 implies that
			\[\cdots\rightarrow 0 \rightarrow V^{I_v^{(p)}} \otimes_R S \xrightarrow{0} H^1(I_v,V)\otimes_R S \rightarrow 0 \cdots \] Therefore we have   $H^0(I_v,V\otimes_R S) = (V \otimes_R S)^{I_v} = V^{I_v} \otimes_R S$ 
		\end{enumerate}
	\end{proof}

	\subsubsection{Tamagawa numbers in families}
	Now we concentrate on the Tamagawa factor of Hida families of Hilbert cuspforms. We consider the case when $ \varphi $ is a $ p $-adic representation $ \rho_\g $ associated with a Hida family $ \g $ of Hilbert modular forms over a real quadratic number field $F$. Suppose $ v \nmid p\infty$ be any place of $ F $, for simplicity, lets denote $ \rho_\g|_{v} $ be the composition of following homomorphisms
	$$ G_{F_v} \hookrightarrow G_F \twoheadrightarrow G_{F,\Sigma} \xrightarrow{\rho_\g} \operatorname{Aut}_{\R_{\g}}(T_\g)$$
	For any arithmetic prime $ P $, let $ T_\g(P)\coloneqq T_\g \otimes_{R_\g} S_P $, where $ S_P $ is the integral closure of $ R_\g/P $. The next Proposition follows from \cite[ \textsection12.4.4, \textsection12.4.5, \textsection12.7.14]{Nek06}.
	\begin{prop}\label{prop:4.16}
		---
		\begin{enumerate}[(1)]
			\item If $\ord_{v}(\frak{N}_\g)=1 $,then $ \rho_\g(I_{v}) $ is infinite.
			\item If $ \#(\rho_\g(I_{v}))= \infty $, then there is a short exact sequence of $ \R_{\g}[G_{F_{v}}] $-modules
			$$ 0 \rightarrow \R_{\g}(1) \otimes \mu \rightarrow T_\g \rightarrow \R_{\g} \otimes \mu \rightarrow 0 $$
			where $ \mu \colon G_{F_v} \rightarrow \{\pm 1\} $ is an unramified quadratic character.
		\end{enumerate}
	\end{prop}
	
	Suppose $ \lambda \colon G_{F,\Sigma} \rightarrow R_\g^{\times} $ is a quadratic character. Then we have
	\begin{cor}
		\label{cor:4.17}
		Let $ \#(\rho_\g(I_{v}))= \infty $ and $ \g  $ has trivial central character. Then the following are equivalent.
		\begin{enumerate}[(1)]
			\item $ H^1(I_{v}, T_\g \widehat{\otimes}_{R_\g} T_\g (\lambda)) $ is free.
			\item $ H^1(I_{v}, T_\g(P) \widehat{\otimes}_{S_{P}} T_\g(\lambda)(P)) $ is free for some arithmetic prime $ P $.
			\item The Tamagawa factor for $ T_\g(P) \widehat{\otimes}_{S_{P}} T_\g(\lambda)(P) $ equals 0 for some arithmetic prime $ P $.
			\item The Tamagawa factor for $ T_\g(\lambda)(P) $ equals 0 for some arithmetic prime $ P $.
		\end{enumerate}
		Moreover, whenever these equivalent conditions are satisfied, then $ H^1(I_{\omega}, T_\g) $ is free for any finite index subgroup $ I_\omega < I_v $ with $ p \nmid [I_v \colon I_\omega] $.
	\end{cor}
	\begin{proof}
		--- \newline	
		$ (1) \Rightarrow (2) $ by Lemma \ref{lemma : 2.9}. \newline
		$ (2) \Rightarrow (3)  $ by definition. \newline	
		From the previous proposition, we have a short exact sequence of $ G_{F_v} $-stable modules.
		$$ 0 \rightarrow \R_{\g}(1) \otimes \mu \rightarrow T_\g \rightarrow \R_{\g} \otimes \mu \rightarrow 0 $$
		\textit{Claim I:} (3) and (4) are equivalent.  
		Let choose a basis $ \{e_1, e_2\} $ of $ T_\g $ such that the matrix of 
		$$ \rho_\g|_{v}(t) =
		\begin{pmatrix}
			1 &a\\
			&1	
		\end{pmatrix} $$
		for some non zero $ a \in \R_{\g} $. Similarly we have basis $ \{e'_1, e'_2\} $ of $ T_\g(\lambda) $. Then 
		\begin{equation}
			\rho_{\g \otimes \g(\lambda)}|_{v}(t) = \lambda(t)
			\begin{pmatrix}
				1 &a &a &a^2\\
				&1 &  &2a\\
				&  &1 &2a\\
				&  &  &1
			\end{pmatrix}
		\end{equation} 
		with respect to the basis $ \{e_1 \otimes e_1',e_1 \otimes e_2',e_2 \otimes e_1',e_2 \otimes e_2' \} $. Let us define the $ G_{F_v} $-module 
		$$ C\coloneqq \operatorname{Span}_{\R_{\g} \widehat{\otimes} \R_{\g}}\{e_1 \otimes e_1',e_1 \otimes e_2'+e_2 \otimes e_1'\} $$
		Let us take an arithmetic prime $P$. Since $\ord_{v} (\frak{N}_\g ) = 1$, the image of composition
		$$ I_{v} \rightarrow \GL{2}(\R_{\g}) \rightarrow \GL{2}(\R_{\g}/P) \subset \GL{2}(S_P) $$
		has infinite cardinality (see also \cite[\textsection 12.4.4.2,\textsection 12.7.14.1]{Nek06}). Hence $ a \notin P $ and 
		$$ H^1(I_{v}, T_\g(P) \widehat{\otimes}_{S_{P}} T_\g(\lambda)(P))_{\mathrm{tors}} = (C \otimes S_{P}/aS_{P})(-1) $$
		The exact sequence of $ Gal(F_{v}^{ur}/F_{v}) $-modules
		$$ 0 \rightarrow \R_{\g} \widehat{\otimes} \R_{\g}(1) \rightarrow C(-1 ) \rightarrow \R_{\g} \widehat{\otimes} \R_{\g} \rightarrow 0 $$ gives 
		$$ \det((C \otimes S_{P}/aS_{P})(-1) \xrightarrow{Fr_{v}-1} (C \otimes S_{P}/aS_{P})(-1)) =0 $$
		Since $S_{P}/aS_{P}$ is a finite ring (as $a \notin P$), the condition (3) is equivalent to the requirement that $a \in \R_\g^{\times}$. \newline
		Similarly, we can also show that condition (4) is also equivalent to $a \in \R_\g^{\times}$ which shows that $(3) \Longleftrightarrow (4)$. \newline
		Condition (3) gives that $a \in \R_\g^{\times}$ which implies $H^1(I_{v}, T_\g \widehat{\otimes}_{R_\g} T_\g (\lambda)) $ is free of rank 2. Hence $(3) \Rightarrow (1)$.
	\end{proof}
	Let $ \f $ be a Hida family of Hilbert modular forms over $ F $. The following lemma follows from \cite[\textsection 12.4.4, Proposition 12.7.14.2]{Nek06},
	\begin{lemma}
		\label{lemma : 2.13}
		Let $ 1 < \rho_{\f}(I_{v}) < \infty$, then for any arithmetic prime $ P $, one of the following holds.
		\begin{enumerate}
			\item There is a (ramiﬁed) character $\mu \colon G_{F_v} \rightarrow S_P$ such that $T_\f(P)|_{I_v} =  \mu \oplus \mu^{-1}$.
			\item $T_\f(P )$ is monomial, namely, $T_\f(P ) = \Ind^{G_{E_\omega}}_{G_{F_v}}(\mu)$, where $E_{\omega} /F_v$ is a ramiﬁed quadratic extension with absolute Galois group $ G_{E_\omega} $, and $ \mu \colon G_{E_\omega} \rightarrow S_{P}^{\times} $ is a ramified character.
			\item $v \mid 2$ and there exists a Galois extension of $F_v$ with Galois group isomorphic to $A_3$ or $S_3$ , over which $ T_\f(P ) $ becomes monomial.
		\end{enumerate} 
	\end{lemma}
	\begin{cor}
		\label{cor : 2.14}
		Let $ 1 < \rho_{\f}(I_{v}) < \infty$. Then the following are equivalent
		\begin{enumerate}
			\item The $ \R_{\f} $-module $ H^1(I_v, T_\f) $ is free.
			\item The $ S_P $-module $ H^1(I_v, T_\f(P)) $ is free for some arithmetic prime $ P $.
			\item $ H^1(I_v, T_\f(P)) =0 $ for some arithmetic prime $ P $.
			\item $ H^1(I_v, T_\f) =0$
		\end{enumerate}
	\end{cor}
	\begin{proof}
		From the previous lemma, we have $ T_\f(P)^{I_v}= 0 $. Therefore, by the Lemma \ref{lemma : 2.9}, $ H^1(I_v, T_\f(P)) $ is a torsion the $ S_P $-module. This shows that the conditions (2) and (3) are equivalent. (1) $\Leftrightarrow$ (2) and (3) $\Leftrightarrow$ (4) follow from Lemma \ref{lemma : 2.9}. Lastly (4) $ \Rightarrow $ (1) is trivial.
	\end{proof}
	\begin{cor}
		Let $ 1 < \rho_{\f}(I_{v}) < \infty$. If the $ \R_{\f} $-module $ H^1(I_v, T_\f) $ is free then $ T_\f^{I_v^{(p)}}= 0 $
	\end{cor}
	\begin{proof}
		By Lemma \ref{lemma : 2.9}. its enough to show $ T_\f(P)^{I_v^{(p)}}= 0 $ for some arithmetic prime $P$. \\
		Lets consider the case $ (i) $ of Lemma \ref{lemma : 2.13}, so that there is a (ramiﬁed) character $ \mu \colon G_{F_v} \rightarrow S_{P}^{\times} $ with the property of $T_\f(P)|_{I_v} =  \mu \oplus \mu^{-1}$.\\
		Suppose that $ \mu(I_v)= \{1\} $. So by Corollary \ref{cor : 2.14}. 
		$$ 0 = H^1(I_v, T_\f(P)) = S_{P}/(\mu(t)-1)S_{P} $$
		So $ (\mu(t)-1) \in S_{P}^{\times} $. Since $ t \in I_{v}/I_{v}^{(p)} \cong \Z_p $ and $\#\rho_{\f}(I_{v})$ is finite, $ \mu(t) $ is a $p$-power root of unity. This contradicts the fact that $ (\mu(t)-1) \in S_{P}^{\times} $ and $ \mu(I_{v}^{(p)}) \ne 1 $, which implies $ T_\f(P)^{I_v^{(p)}}= 0 $. \\
		In the remaining two cases of lemma 2.13, $I_v^{(p)}$ acts non-trivially on $T_\f(P)$. So, if $ T_\f(P)^{I_v^{(p)}} \ne 0 $, then $ I_{v}/I_{v}^{(p)} $ on $ T_\f(P)^{I_v^{(p)}} $ by a non trivial character $ \nu $. Since $ H^1(I_v, T_\f(P)) =0 $ by the Corollary \ref{cor : 2.14}. we have $ (\nu(t)-1) \in (S_P)^\times $. But $ \nu(t) $ is a $ p $-power root of unity, this contradicts the assumption $ T_\f(P)^{I_v^{(p)}} \ne 0 $.
	\end{proof}
	\begin{prop}\cite[\textsection 12.4.10.3]{Nek06}
		Let $ f \in S_k(\frak{N}, \mathbb{C}) $ be a newform and $ L $ be the number field generated by all the Hecke eigenvalues of $ f $. Suppose $ 1< \#\rho_f(I_v)<\infty $ for a place $ v \nmid p\infty $ and $ p\nmid \#( H^0(G_L, \Q/\Z(2)) )$, then $ H^1(I_v,T(f)) =0$ and $ \zeta_p + \zeta_p^{-1} \notin L $.
	\end{prop}
	\begin{prop}
		Let $ 1< \#\rho_{\f}(I_v)<\infty $, $ p\nmid v-1 $ and $ v\ne 2 $. Suppose for some arithmetic prime $ P $, $ T_{\f}(P) $ is supercuspidal then $ p \nmid \#\rho_{\f}(I_v) $ and $ H^1(I_v,T_{\f}) =0 $.
	\end{prop}
	\begin{prop}
		Let $ 1< \#\rho_{\f}(I_v)<\infty $, $ p \ne 3,7 $ and $ v=2 $. Suppose for some arithmetic prime $ P $, $ T_{\f}(P) $ is supercuspidal then $ p \nmid \#\rho_{\f}(I_v) $ and $ H^1(I_v,T_{\f}) =0 $.	
	\end{prop}
	\begin{proof}
		The last two propositions are directly derived from Corollary \ref{cor : 2.10} and Lemma \ref{lemma : 2.13}.
	\end{proof}
	Now we summarize the results
	\begin{cor}
		\label{cor_tam_assump}
		Let $ v $ be non-archimedean place such that $ 1< \#\rho_{\f}(I_v)<\infty $, $ v \nmid p\infty $. Suppose that at least one of the following conditions holds,
		\begin{enumerate}
			\item $ p\nmid v-1 $ and $ v\ne 2 $ and for some arithmetic prime $ P $, $ T_{\f}(P) $ is supercuspidal.
			\item $ p \ne 3,7 $ and $ v=2 $ and for some arithmetic prime $ P $, $ T_{\f}(P) $ is supercuspidal.
			\item for some arithmetic prime $ P $, $ \f(P) $ be a newform and $ L $ be the number field generated by all the Hecke eigenvalues of $ \f(P) $ and $ \zeta_p + \zeta_p^{-1} \notin L $.
		\end{enumerate}
		Then $ p \nmid \#\rho_{\f}(I_v) $ and $ H^1(I_v,T_{\f}) =0 $.
	\end{cor}

	\begin{thm}
		Let the assumptions of the last corollary hold for both families of Hilbert cuspforms $ \f $, $ \g $ over $ F $. Assume also that the $p$-part of the Tamagawa factor for a single member of the Hida family $\f$ and a single member of the family $\g$ is trivial. Then $H^1(I_v , V )$ is free for any $ G_{\Q,\Sigma} $-representation $V \in \{\T_1, \T_2, \T_3\}$ and $ v \nmid p\infty $.
	\end{thm}
	\begin{proof}
		As  $ \T_3 = \T_1 \oplus \T_2 $, it is enough to show the theorem for $ V = \T_3 $. \\
		If $ v \notin \Sigma $ then $ I_v $ act trivially on $ \T_3 $ (converse is also true). Hence $H^1(I_v , \T_3 )$ is free. The following cases will arise when $ v \in \Sigma, \: v \nmid p\infty $. 
		\begin{enumerate}[\text{Case} 1.]
			\item If $ 1 < \#\rho_{\f}(I_v),\#\rho_{\g}(I_v) < \infty $ then $ p \nmid \#\rho_{\f}(I_v),\#\rho_{\g}(I_v) $. Therefore we have $ 1< \#\rho_{\T_3}(I_v) < \infty $ and $ p \nmid \#\rho_{\T_3}(I_v) $ , implies $ H^1(I_v , \T_3) $ is free via corollary 2.10.
			\item Let $ 1 < \#\rho_{\f}(I_v) < \infty $ and $ \#\rho_{\g}(I_v) = \infty $. Therefore we have  $ I_v^{(p)} $ acts trivially on $ \As(T_\g)(\lambda) $ (the character $ \mu$ in Proposition 2.11. is always quadratic)
			$$ \T_3^{I_v^{(p)}} = (\T_3^{I_v^{(p)}})^{I_v^{(p)}/I_v^{(p)}} = (T_\f^{I_v^{(p)}} \otimes \As(T_\g)(\lambda))^{I_v^{(p)}/I_v^{(p)}} = 0 $$ 
			as $T_\f^{I_v^{(p)}} = 0$ by corollary 2.15. In particular $ H^1(I_v , \T_3) =0$. 
			\item Let $ \#\rho_{\f}(I_v) = \infty $ and $ 1 < \#\rho_{\g}(I_v) < \infty $. Then define $ I_\omega \coloneqq \rho_{\g}^{-1}(1) $. As $ p \nmid \#\rho_{\g}(I_v) $, then $ p \nmid [I_v \colon I_\omega] $. Hence $ I_\omega $ acts freely on $ \As(\rho_{\g})(\lambda) $ and the module $ H^1(I_\omega, T_\f) $ is free, and hence $ H^1(I_v,  \T_3) =(H^1(I_\omega, T_\f)\widehat{\otimes}\As(T_\g)(\lambda))^{I_v/I_\omega}  $ is also free. 
			\item Let us consider $\#\rho_{\f}(I_v) = \infty$ and $\#\rho_{\g}(I_v) = \infty$. Since $(\T_3)^{I_v^{(p)}} = (T_{\f})^{I_v^{(p)}} \widehat{\otimes}_{O} T_{\g}\widehat{\otimes}_{O} T_{\g}^\theta$ (as $I_v^{(p)}$ acts trivially on $T_{\g} \widehat{\otimes}_{O}T_{\g}^\theta$ by Proposition \ref{prop:4.16}), without loss of generality we can take 
			$I_v^{(p)}$ acts trivially on $T_{\f}$ and $T_{\g}$ by Proposition \ref{prop:4.16}. Hence we get 
			\[
			\rho_{\spadesuit}(t) = 
			\begin{pmatrix}
				1 & a_{\spadesuit}
				\\
				0 & 1
			\end{pmatrix}  \,,\qquad \spadesuit \in \{\f, \g, \g^\theta\}
			\]
			in $\GL{\R_{\spadesuit}}(T_{\spadesuit}) \cong \GL{2}(\R_{\spadesuit})$. 
			Moreover, via the proof of Corollary \ref{cor:4.17} we have $a_{\spadesuit} \in \R_{\spadesuit}^\times$. Hence the double coset 
			$\GL{8}(\R)\ (\bs{\rho}_3(t) - 1)\ \GL{8}(\R)$ is given by the matrix 
			\[
			\begin{pmatrix}
				0 & 1 & 0&&&&&
				\\
				&  0& 0&1&&&&
				\\
				& & 0&0&1&&&
				\\
				& & &0&0&1&&
				\\
				&  & &&0&0&1&
				\\
				&&&&&0&0&1
				\\
				&& &&&&0&0
				\\
				&  & &&&&&0
			\end{pmatrix}\,. 
			\]
			This gives us $H^1(I_v,\T_3)$ is free, as required. 
			
			\item Suppose $\#\rho_{\f}(I_v) = 0$, then $H^1(I_v, \T_3) = T_\f^{\dagger} \widehat{\otimes}_{O} H^1(I_v, T_{\g}^{\dagger}\widehat{\otimes}_{O} T_{\g^{c}}^{\dagger})$ is free by Corollary \ref{cor:4.17}. 
			If $\#\rho_{\g}(I_v) = 0$, then $H^1(I_v, \T_3) = H^1(I_v, T_\f)\widehat{\otimes}_{O} T_{\g}\widehat{\otimes}_{O} T_{\g}^{\theta}$ is free also by Corollary \ref{cor:4.17}. This completes the proof of our proposition. 
		\end{enumerate}
		
	\end{proof}
	\begin{rmk}
		We remark that almost all of our results in this section (\textsection \ref{sec:tamagawa}) can be extended for a general totally real number field $ L/F $ and $ \T_3 = T_\f \hotimes T_\g \hotimes T_{\mbf h} $ where $ \f, \ \g $, and $ \mbf h $ are Hida families of Hilbert cuspforms over $ L $. For simplicity, we are avoiding doing it here in full generality.
	\end{rmk}
	
	\subsection{Algebraic $ p $-adic $ L $-functions}
	Our goal in this section is to introduce the module of algebraic $ p $-adic $ L $-functions in terms of our Selmer complexes.
	\subsubsection{Characteristic ideal}
	Let $ R $ be a normal, regular local ring, and $ M $ be a torsion $ R $-module of finite type. We write $ R_{\frak P} $ as the localisation of $ R $ at $ \frak P \in \operatorname{Spec}(R)$, $ M_{\frak P} \coloneqq M \otimes_R  R_{\frak P} $. Then we have
	\begin{align*}
		\Char_{R}(M) \coloneqq \prod_{\operatorname{ht}(\frak P)=1} \frak P^{\operatorname{length}_{R_{\frak P}}(M_{\frak P}) }; \qquad \frak P \in \operatorname{Spec}(R).
	\end{align*} 
	\subsubsection{Algebraic $ p $-adic $ L $-functions}
	Let R be a Noetherian, complete local ring with finite residue field of characteristic $ p\ge 5 $. Suppose $ T $ be a free $ R $-module of finite rank with a continuous $ G_{F,\Sigma} $-action. We also assume that there exist an free $ R\bbox{G_{F_\p}} $-submodule $ \scrF_{\p}^+T $ such that the quotient $ \scrF_{\p}^-T \coloneqq T/\scrF_{\p}^-T $ is free as an $ R $-module so that we can attach a Greenberg local condition $ \Delta $, cf. \textsection \ref{sec:greenberg}.
	
	Let $ \mathfrak{m}_\R $ denote the maximal ideal of $ \R $ and let $ \overline{T}\coloneqq T \otimes_{\R} \R/\mathfrak{m}_\R $. Let assume that 
	\begin{itemize}
		\item[\mylabel{H=0}{\textbf{($ H^0 $)} }]$ \overline{T}^{G_F} $ and $( \overline{T}^\vee(1))^{G_F} $ are trivial. 
		\item[\mylabel{Tam}{\textbf{(Tam)}}]$ H^1(I_v,T) $ is free for any prime $ v \in \Sigma\ba \bP $.
	\end{itemize}
	\begin{prop} ---
		\label{dercat}
		\begin{enumerate}[(1)]
			\item For any ideal $ I $ of $ \R $, the complex $ \widetilde{R\Gamma}_{\mathrm{f}}(G_{F,\Sigma},T/IT, \Delta) $ is perfect and we have a natural isomorphism 
			$$ \widetilde{R\Gamma}_{\mathrm{f}}(G_{F,\Sigma},T, \Delta) \otimes^{\mathbbl{L}}_\R \R/I \xrightarrow{\sim} \widetilde{R\Gamma}_{\mathrm{f}}(G_{F,\Sigma},T/IT, \Delta).$$
			\item $ \widetilde{R\Gamma}_{\mathrm{f}}(G_{F,\Sigma},T, \Delta) \in D^{[1,2]}_{\mathrm{parf}}(\prescript{}{\R}{\mathrm{Mod}}) $.
		\end{enumerate}
	\end{prop}
	\begin{proof}
		The proof is inspired from \cite[Lemma A.2, Corollary A.5]{BS22}.
		\begin{enumerate}[(1)]
			\item  Suppose $I$ is an ideal of $R$. 
			Let us denote
			\[
			U^{-}_{\Sigma}(T/IT) := \mathrm{Cone}\left(U^{+}_{\Sigma}(T/IT) \xrightarrow{-i^+_{\Sigma}} C^{\bullet}_{\Sigma} (T/IT) \right). 
			\]
			Since $H^1(I_v, T)$ is free for any prime $v \in \Sigma \ba \bP$ by our assumption, $U_{\Sigma}^{-}(T/IT)$ is a perfect complex by Proposition \ref{prop_unramified_complex} for any $I$, and 
			$U_{\Sigma}^{-}(T) \otimes_R^{\mathbbl{L}}R/I \stackrel{\sim}{\longrightarrow} U_{\Sigma}^{-}(T/IT)$. 
			Therefore we have an exact triangle 
			\[
			\widetilde{R\Gamma}_{\rm f}(G_{F, \Sigma}, T/IT, \Delta) \longrightarrow R\Gamma(G_{F, \Sigma}, T/IT) \longrightarrow U^{-}_{\Sigma}(T/IT) \xrightarrow{+1}\,,
			\]
			which gives that the homomorphism $\widetilde{R\Gamma}_{\rm f}(G_{F, \Sigma}, T, \Delta) \otimes_R^{\mathbb{L}} R/I \longrightarrow \widetilde{R\Gamma}_{\rm f}(G_{F, \Sigma}, T/IT, \Delta)$ is an isomorphism.

			\item
			We have  $\widetilde{H}^{i}_{\rm f}(G_{F, \Sigma}, T, \Delta) = 0$ for any $i \geq 4$ by the definition of Selmer complexes. So at first we have to prove that $\widetilde{H}_{\rm f}^3(G_{F, \Sigma}, T, \Delta) = 0$. 
			Since $(\overline{T}^\vee(1))^{G_F} = 0$, global duality shows that the homomorphism 
			\[
			H^2(G_{F, \Sigma}, T) \longrightarrow \bigoplus_{v \in \Sigma} H^2(G_{v}, T)
			\]
			is surjective. 
			Since $H^2(G_{\p}, T) \longrightarrow H^2(G_{\p}, \scrF^-T)$ is surjective for $ \p =\p_1,\p_2 $, we have 
			\[
			\widetilde{H}_{\rm f}^3(G_{F, \Sigma}, T, \Delta) = \mathrm{coker}\left(H^2(G_{F, \Sigma}, T) \longrightarrow \bigoplus_{\p \mid p} H^2(G_{\p}, \scrF^-T) \oplus \bigoplus_{v \in \Sigma \setminus \{p\}} H^2(G_{v}, T)\right) = 0\,.
			\]
			This shows that $\widetilde{R\Gamma}_{\rm f}(G_{F, \Sigma}, T, \Delta) \in D^{[a,2]}_{\mathrm{parf}}(_{R}\mathrm{Mod})$ 
			for some integer $a \leq 2$. 
			To complete the proof, we have to show that 
			\[
			H^{i}(\widetilde{R\Gamma}_{\rm f}(G_{F, \Sigma}, T, \Delta) \otimes^{\mathbbl{L}}_{R} R/\mathfrak{m}_R) = 0 
			\]
			for any $i \leq 0$. By part (1), we have 
			\[
			H^{i}(\widetilde{R\Gamma}_{\rm f}(G_{F, \Sigma}, T, \Delta) \otimes^{\mathbbl{L}}_{R} R/\mathfrak{m}_R)
			= 
			\widetilde{H}_{\rm f}^i(G_{F, \Sigma}, \overline{T}, \Delta). 
			\]
			Therefore, 
			$H^{i}(\widetilde{R\Gamma}_{\rm f}(G_{F, \Sigma}, T, \Delta) \otimes^{\mathbbl{L}}_{R} R/\mathfrak{m}_R) = 0$ for any $i < 0$ and 
			\[
			H^{0}(\widetilde{R\Gamma}_{\rm f}(G_{F, \Sigma}, T, \Delta) \otimes^{\mathbbl{L}}_{R} R/\mathfrak{m}_R) 
			= \widetilde{H}_{\rm f}^0(G_{F, \Sigma}, \overline{T}, \Delta) \subset \overline{T}^{G_F} = 0
			\]
			by \ref{H=0}. The proof of our proposition is complete. 
		\end{enumerate}
	\end{proof}
	\begin{defn}
		Whenever the second cohomology group of $ \chi(\widetilde{R\Gamma}_{\mathrm{f}}(G_{F,\Sigma},T, \Delta)$ is $ R $-torsion, we call the (invertible) module $ \Char(\widetilde{H}^2_{\mathrm{f}}(G_{F,\Sigma},T, \Delta)) $ the module of algebraic $ p $-adic $ L $-functions for $ (T,\Delta) $.
	\end{defn}
	The following corollary justifies the definition of algebraic $ p $-adic $ L $-functions.
	\begin{cor}
		\label{cor : sep10}
		Suppose $ R $ in normal and $ \chi(\widetilde{R\Gamma}_{\mathrm{f}}(G_{F,\Sigma},T, \Delta)) = 0 $. If $ \widetilde{H}^2_{\mathrm{f}}(G_{F,\Sigma},T, \Delta) $ is the $ R $-torsion then the first cohomology group vanishes.
	\end{cor}
	\begin{proof}
		By Proposition \ref{dercat}, we have an exact sequence
		$$0 \rightarrow \widetilde{H}^1_{\mathrm{f}}(G_{F,\Sigma},T, \Delta) \rightarrow P \rightarrow P \rightarrow \widetilde{H}^2_{\mathrm{f}}(G_{F,\Sigma},T, \Delta) \rightarrow 0 $$
		for some finitely generated projective $ R $-module $ P $. Since $ \widetilde{H}^2_{\mathrm{f}}(G_{F,\Sigma},T, \Delta) $ is torsion, we get $ P \otimes \operatorname{Frac}R \rightarrow P \otimes \operatorname{Frac}R $ is an isomorphism and hence $ \widetilde{H}^1_{\mathrm{f}}(G_{F,\Sigma},T, \Delta) = 0 $. 
	\end{proof}
	\begin{rmk}
		Whenever the conditions of Corollary \ref{cor : sep10} satisfies, $ \widetilde{H}^2_{\mathrm{f}}(G_{F,\Sigma}, T, \Delta) $ is the same as the Pontryagin dual of dual Selmer group of $ T $ by the theory of Nekovar. This is the exact case of some previous work discussed in \textsection \ref{sec:comparison}.
	\end{rmk}


	\section{The modules of leading terms (higher rank)}
	\label{sec : module_leading}
	In this section, we will introduce the module of leading terms which we give an equivalent definition of the module of algebraic $ p $-adic $ L $-functions. The motivation of this idea goes to the general philosophy of the Kolyvagin system. First, we will discuss some notions and properties of the exterior bi-dual of a module.
	\subsection{Exterior power bi-duals}
	Let $R$ be a complete Gorenstein local ring with finite residue field of characteristic $p > 3$ and $M$ be a $R$-module. We set $$ M^*= \Hom_R(M,R) .$$ For most of the result in this subsection, we refer to \cite[\textsection 2]{BSS18}.
	\begin{defn}
		For each $R$-module $M$, each positive integer $r$ and each map $\varphi$ in $ M^* $, there exists a
		unique homomorphism of $R$-modules 
		$$ \bigwedge^r_R M \rightarrow \bigwedge^{r-1}_R M  $$
		with the property that $$x_1 \wedge \cdots \wedge x_r \mapsto \sum_{i=1}^r (-1)^{i+1} \varphi(x_i) x_1 \wedge \cdots \wedge \hat{x}_i \wedge \cdots \wedge x_r$$
		for each subset $ \{x_i\}_{1\le i \le r} $ of $ M $. By abuse of notation, we also denote this map $ \varphi $.
	\end{defn}
	\begin{defn}
		For any non-negative integer $r$, we define the ‘r-th exterior bi-dual’ of $ M $ to be the $R$-module obtained by setting 
		$$ \bigcap_R^r M \coloneqq \left(\bigwedge_R^r M^*\right)^* $$
	\end{defn}
	\begin{rmk} ---
		\begin{enumerate} 
			\item Note that there is a natural homomorphism 
			$$ \bigwedge_R^r M \rightarrow \bigcap_R^r M; \quad x \mapsto (\Phi \mapsto \Phi(x)),  $$
			which is isomorphism when $ M $	is finitely generated projective $R$-module. 
			\item For each $\varphi \in M^*$, there is a canonical homomorphism of $R$-modules 
			$$\bigcap_R^r M \rightarrow \bigcap_R^{r-1} M, \quad \Phi \mapsto \Phi(\varphi \wedge \cdot) ,$$
			such that the following diagram commutes:
			\begin{center}
				\begin{tikzcd}
					\bigwedge_R^r M \arrow[r] \arrow[d, "\varphi"] & \bigcap_R^r M \arrow[d, "\varphi"] \\
					\bigwedge_R^{r-1} M \arrow[r] & \bigcap_R^{r-1} M
				\end{tikzcd}
			\end{center}
		\end{enumerate}
	\end{rmk}
	Let us recall some basic facts concerning the exterior bi-duals.	
		Let $\varphi \colon M \longrightarrow N$ be any $R$-homomorphism and any integer $t \geq 1$. We define a homomorphism also denoted by
		\begin{align*}
			\varphi \colon {\bigwedge}^t_S M &\longrightarrow  N \otimes_{S} {\bigwedge}^{t-1}_S M \\
			m_1 \wedge \cdots \wedge m_t &\mapsto \sum_{i=1}^t(-1)^{i-1}\varphi(m_i) \otimes m_1 \wedge \cdots \wedge m_{i-1 } \wedge m_{i+1} \wedge \cdots \wedge  m_t. 
		\end{align*}
		\begin{prop}\cite[Lemma B.2]{BS22}
			\label{proposition_exterior_wedge}
			For an exact sequence of $R$-modules 
			\[
			0 \longrightarrow M \longrightarrow P^1 \stackrel{\varphi}{\longrightarrow} P^2, 
			\]
			where $P^1$ and $P^2$ are finitely generated free $R$-modules, the canonical map 
			\[
			{\bigcap}^t_{R}M \longrightarrow {\bigcap}^t_{R}P^1 = {\bigwedge}^t_{R}P^1
			\]
			induces an isomorphism 
			\[
			{\bigcap}^t_{R}M \stackrel{\sim}{\longrightarrow} \ker\left({\bigwedge}^t_{R}P^1 \stackrel{\varphi}{\longrightarrow} P^2 \otimes_R  {\bigwedge}^{t-1}_{R}P^1\right). 
			\]
			where $t \in \Z_{>0}$.
		\end{prop}
		Many of the results in the following subsection are inspired by the work done in \cite{BCS23}.

		\subsection{Construction of special elements in the extended Selmer modules}
		Let $R$ be a complete Gorenstein local ring with a finite residue field of characteristic $p$ and $T$ be a free $R$-module of finite rank with a continuous $G_{F,\Sigma}$-action. Throughout this section, lets assume we have an $R\llbracket G_{F_\p}\rrbracket$-submodule $\scrF^+_\p T$ such that the quotient $T/\scrF^+_\p T$ is free as an $R$-module for $ \p = \p_1, \; \p_2 $. The condition \ref{Tam} also holds. \\
		We also have a Greenberg local condition $\Delta\coloneqq \Delta_{\scrF^+}$ with respect to $R[G_{F_{\p}}]$ submodule (free) $\scrF^+_{\p}T$ of $T$ for $ \p = \p_1, \; \p_2 $. Let us put
		$$ r= r(T,\Delta) \coloneqq -\chi(\widetilde{R\Gamma}_{\mathrm{f}}(G_{F,\Sigma},T,\Delta)) \in \Z $$
		As $ \widetilde{R\Gamma}_{\mathrm{f}}(G_{F,\Sigma},T,\Delta) \in D^{[1,2]}_{\mathrm{parf}}(\prescript{}{\R}{\mathrm{Mod}}) $ by Proposition \ref{dercat}.2, we have a short exact sequence 
		\begin{equation}
			\label{ses}
			0 \rightarrow \widetilde{H}^1_{\mathrm{f}}(G_{F,\Sigma},T,\Delta) \rightarrow R^{a+r} \xrightarrow{\varPhi} R^a \rightarrow \widetilde{H}^2_{\mathrm{f}}(G_{F,\Sigma},T,\Delta) \rightarrow 0
		\end{equation}
		for some integer $ a >0 $. We put 
		\[ \varPhi = (\varPhi_1, \cdots , \varPhi_a) \colon  R^{a+r} \rightarrow R^a \]
		and obtain a homomorphism 
		\[ \varPhi_a\circ \cdots \circ \varPhi_1 \colon \bigwedge_R^{a+r} R^{a+r} \to \bigwedge_R^{r} R^{a+r} \]
		Moreover, $ \varPhi \circ \varPhi_a\circ \cdots \circ \varPhi_1 =0 $. Hence by Proposition \ref{proposition_exterior_wedge}, we get 
		\[ \mathrm{im}(\varPhi_a\circ \cdots \circ \varPhi_1) \subset \bigcap_R^r \widetilde{H}^1_{\mathrm{f}}(G_{F,\Sigma},T,\Delta). \]
		When $ r=0 $, we also have $ \mathrm{im}(\varPhi_a\circ \cdots \circ \varPhi_1) \subset R $ and $ \bigcap_R^0 \widetilde{H}^1_{\mathrm{f}}(G_{F,\Sigma},T,\Delta) =R $ by definition.
		\begin{defn}
			We define the rank $ r $ module of leading terms as the cyclic $ R $-module $ \bbdelta_r(T,\Delta) $ to be the image of
			\[ \det(R^{a+r}) = \bigwedge_R^{a+r} R^{a+r} \rightarrow \bigcap_R^r \widetilde{H}^1_{\mathrm{f}}(G_{F,\Sigma},T,\Delta) \]
			induced by the exact sequence \ref{ses}.
		\end{defn}
		The definition of $\bbdelta_r(T,\Delta)$ is inspired from the definition of $\delta(T,\Delta)$ given in \cite[Definition 5.3]{BCS23}. 
		For any $0 \le i \le r$, we also have 
		$$ \bbdelta_i(T,\Delta) \coloneqq\left\lbrace  \delta_\phi \in  \bigcap^i_R\widetilde{H}^1_{\mathrm{f}}(G_{F,\Sigma},T,\Delta) \mid \delta \in  \bbdelta_r(T,\Delta), \; \phi \in \bigwedge^{r-i}_R \Hom_R(\widetilde{H}^1_{\mathrm{f}}(G_{F,\Sigma},T,\Delta),R) \right\rbrace  $$
		where $\delta_\phi \colon \bigwedge^{i}_R \Hom(\widetilde{H}^1_{\mathrm{f}}(G_{F,\Sigma},T,\Delta),R) \rightarrow R$ defined as $ \psi \mapsto \delta(\phi \wedge \psi) $.
		\begin{prop}
			The module $ \bbdelta_r(T,\Delta) $ is independent of the choice of the representation
			\[ [R^{a+r} \xrightarrow{\varPhi} R^a] \]
			of the Selmer complex $ \widetilde{R\Gamma}_{\mathrm{f}}(G_{F,\Sigma},T,\Delta) $.
		\end{prop}
		\begin{proof}
			Let $ n \in \Z_{\ge 0} $ be the smallest integer such that $ \cohom^2( G_{F,\Sigma},T,\Delta)  $ is generated by $ n $ elements. Consider the exact sequence
			\[R^m \xrightarrow{\phi} R^n \rightarrow \cohom^2( G_{F,\Sigma},T,\Delta) \rightarrow 0,\] 
			Without loss of generality, we can assume a splitting $ R^m = R^n \oplus R^{m-n} $ such that $ \phi = \mathrm{pr}_1 $. Let us denote
			\[ K \coloneqq \ker (R^m \rightarrow \cohom^2(\G,T,\Delta)). \]
			Therefore we get $ K \oplus R^{a-n} = \ker(R^a \rightarrow \cohom^2(\G,T,\Delta)) $. As $ R^{a +r} \rightarrow  K \oplus R^{a-n} $ is surjective, there exist a splitting $ R^{a+r} = R^{n+r} \oplus R^{a-n} $ such that $ \im(R^{n+r} \to K\oplus R^{a-n}) = K $. Hence $ [R^{n+r} \to R^n] $ is a representative of $ \dcom(\G,T,\Delta) $ and 
			\[[R^{a+r} \xrightarrow{\varPhi} R^a] = [R^{n+r} \to R^n] \oplus [R^{a-n} \xrightarrow{\sim} R^{a-n}]. \]
			Therefore, the diagram 
			\[\xymatrix{
				\det(R^{a+r}) \ar[r] \ar[d]^{\cong}  & \cap^r_R \cohom^1(\G,T,\Delta) \ar@{=}[d] \\
				\det(R^{m+r}) \ar[r]   & \cap^r_R \cohom^1(\G,T,\Delta)
			}\]
			is commutative, thus completing the proof.
		\end{proof}
		\begin{prop}
			\label{prop : sep10}
			The $R$-module $\widetilde{H}^2_{\mathrm{f}}(G_{F,\Sigma},T,\Delta)$ is torsion if and only if $\bbdelta_r(T,\Delta)$ is generated by an $R$-regular element of $ \bigcap_R^r\widetilde{H}^1_{\mathrm{f}}(\G,T,\Delta) $.
		\end{prop}
		\begin{proof}
			Let $Q := \mathrm{Frac}(R)$. If $\widetilde{H}^2_{\rm f}(G_{F, \Sigma}, T, \Delta)$  is torsion, then surjectivity of the homomorphism $Q^{a+r} \to Q^a$ implies $
			\bbdelta_r(T, \Delta) \otimes_R Q   \cong Q 
			$
			by construction. Hence $\bbdelta_r(T, \Delta)$ is generated by an $R$-regular element of ${\bigcap}^{r}_{R}\widetilde{H}^1_{\rm f} (G_{F, \Sigma}, T, \Delta)$. 
			
			To show the converse, we get injective homomorphism 
			\[
			\varPhi_a \circ \cdots \circ \varPhi_1  \colon {\bigwedge}^{r+a}_Q Q^{r+a} \longrightarrow {\bigwedge}^r_{Q} Q^{r+a}
			\]
			since $\bbdelta_r(T, \Delta)$ is generated by an $R$-regular element. Then $\varPhi_1 \colon {\bigwedge}^{r+a}_Q Q^{r+a} \to {\bigwedge}^{r+a-1}_Q Q^{r+a}$ is injective, and hence $\varPhi_1 \colon Q^{r+a} \to Q$ is surjective. 
			Then there is a splitting $Q^{r+a} =  \ker(\varPhi_1)\oplus Q$ such that 
			\[
			\varPhi_a \circ \cdots \circ \varPhi_2  \colon \bigwedge^{r+a-1}_Q \ker(\varPhi_1) \rightarrow {\bigwedge}^r_{Q} \ker(\varPhi_1)
			\]
			is injective. Similarly, we can prove that $ \varPhi_i $ is surjective for $ 2 \le i \le a $. So we conclude that $\varPhi \colon Q^{r+a} \to Q^r$ is surjective, that is, 
			$\widetilde{H}^2_{\rm f}(G_{F, \Sigma}, T, \Delta) \otimes_R Q = 0$.  
		\end{proof}
		The next one is inspired by \cite[Theorem 5.6]{BCS23}.
		\begin{thm}
			\label{H2tors}
			If $\bbdelta_r(T,\Delta)$ is not trivial, then $$ \mathrm{char}_R\left(\bigcap^r_R \widetilde{H}^1_{\mathrm{f}}(G_{F,\Sigma},T,\Delta)/\bbdelta_r(T,\Delta)\right)  = \mathrm{char}_R(\widetilde{H}^2_{\mathrm{f}}(G_{F,\Sigma},T,\Delta)). $$
		\end{thm}
		\begin{proof}
			Let $\fp$ be a height-1 prime ideal of $R$. 
			Since $R_\fp$ is a discrete valuation ring and $\widetilde{H}^2_{\rm f}(G_{F, \Sigma}, T, \Delta)$  is torsion, 
			we may assume that $\varPhi_i = \pi^{n_i} \cdot \mathrm{pr}_i$ for some integer $n_i \in \Z_{\ge0} $. 
			Here $\pi$ denotes an uniformizer of $R_\fp$ and $ \mathrm{pr}_i$ projection onto the $i$th factor. 
			Lets denote $\{e_1, \ldots, e_{r+a}\}$ as the standard basis of $R_\fp^{r+a}$, then we have 
			$\widetilde{H}^1_{\rm f}(G_{F, \Sigma}, T, \Delta) \otimes_{R} R_\fp = R_\fp e_{a+1} + \cdots + R_\fp e_{a+r}$ and 
			\begin{align*}
				\bbdelta_r(T, \Delta)R_\fp 
				&= \mathrm{im}\left( \det(R_\fp^{r+a}) \longrightarrow {\bigwedge}^r_{R_\fp}R_\fp^{r+a} \right) 
				\\
				&= \varPhi_a \circ \cdots \circ \varPhi_1(e_1 \wedge \cdots \wedge e_{r+a})R_\fp
				\\
				&= \pi^{\sum n_i} {\bigwedge}^r_{R_\fp} (\widetilde{H}^1_{\rm f}(G_{F, \Sigma}, T, \Delta) \otimes_R R_\fp). 
			\end{align*}
			We therefore conclude that 
			\begin{equation*}
				\mathrm{length}_{\fp} \left( \bigcap^r_R \widetilde{H}^1_{\rm f}(G_{F, \Sigma}, T, \Delta)   /  \bbdelta_r(T, \Delta)  \right) 
				= \sum n_i		
				= \mathrm{length}_{\fp}\left( \widetilde{H}^2_{\rm f}(G_{F, \Sigma}, T, \Delta) \right), 
			\end{equation*}
			which completes the proof. 
		\end{proof}
		\begin{lemma}\cite[Lemma 5.8]{BCS23}
			Let $S$ be a 0-dimensional Gorenstein local ring, and $M$ a finitely generated $S$-module. Let $\delta \in M^*$ . Then by Maltis duality, we have a natural isomorphism 
			$$ (S\delta)^* \rightarrow \rm im (\delta), \quad \phi \mapsto \phi(\delta) $$
		\end{lemma}

		\begin{prop}
			\label{Fitt}
			If $R/I$ is a $0$-dimensional Gorenstein local ring, then any generator $\delta$ of $\bbdelta_r(T /IT, \Delta)$ induces an isomorphism
			$$ \bbdelta_r (T /IT, \Delta)^* \xrightarrow{\sim} \Fitt^0_{R/I}\widetilde{H}^2_{\mathrm{f}}(G_{F,\Sigma},T /IT,\Delta); \quad \phi \mapsto \phi(\delta) $$
		\end{prop}
		\subsubsection{}
		When $r=0$, by definition we have $ \cap_R^0 M = S$ for any commutative ring S and any finitely generated $S$-module M. As a result, $ \bbdelta_0(T,\Delta) \subset R $. We have the following theorem 
		\begin{thm}
			\label{r0}
			Suppose that $ r=0 $
			\begin{enumerate}
				\item $ \bbdelta_0(T,\Delta) = \Fitt^0_R \widetilde{H}^2_{\mathrm{f}}(G_{F,\Sigma},T,\Delta) $.
				\item If $ \widetilde{H}^1_{\mathrm{f}}(G_{F,\Sigma}, T, \Delta)=0 $, then $ \widetilde{H}^2_{\mathrm{f}}(G_{F,\Sigma},T,\Delta) $ is torsion and $$ \bbdelta_0(T,\Delta)= \Char_R(\widetilde{H}^2_{\mathrm{f}}(G_{F,\Sigma}, T, \Delta)) .$$
				In other words, $ \bbdelta(T,\Delta) $ coincides with the module of algebraic $p$-adic $L$-functions for $(T,\Delta)$.
				\item If $ \bbdelta_0(T,\Delta) $ is generated by a regular element of $R$, then $ \widetilde{H}^1_{\mathrm{f}}(G_{F,\Sigma}, T, \Delta)=0 $
			\end{enumerate}
		\end{thm}  
		\begin{proof} ---
			\begin{enumerate}
				\item By Propositions \ref{dercat} and \ref{Fitt} we have
				\[
				\begin{split}
					\bbdelta_0(T,\Delta) &= \varprojlim_I \bbdelta_0(T/IT,\Delta) \\
					&= \varprojlim_I \Fitt^0_{R/I}\widetilde{H}^2_{\mathrm{f}}(G_{F,\Sigma},T /IT,\Delta) = \Fitt^0_{R}\widetilde{H}^2_{\mathrm{f}}(G_{F,\Sigma},T,\Delta).
				\end{split}
				\]
				\item This is consequence of Corollary \ref{cor : sep10} and Theorem \ref{H2tors}.
				\item Follows from Propossition \ref{prop : sep10}. 
			\end{enumerate}
		\end{proof}
		\subsubsection{}
		When $r=1$ we note that for any $0$-dimensional Gorenstein local ring $S$ and any finitely generated $S$-module $M$ that the canonical homomorphism $M \rightarrow \bigcap^r_R M= M^{**}$ is an isomorphism by Maltis duality. Hence we have 
		$$ \bbdelta_1(T,\Delta) = \varprojlim_I \bbdelta_1(T/IT,\Delta) \subset \varprojlim_I \widetilde{H}^1_{\mathrm{f}}(G_{F,\Sigma},T/IT,\Delta)) =\widetilde{H}^1_{\mathrm{f}}(G_{F,\Sigma},T,\Delta)) $$ we have 
		\begin{thm}
			\label{leading_term_rank1}
			Suppose that $ \bbdelta_1(T, \Delta)  $ is generated by a regular element of $ R $. Then we have an exact sequence of $ R $-modules
			$$ 	0 \rightarrow \Fitt^0_R \widetilde{H}^2_{\mathrm{f}}(G_{F,\Sigma},T,\Delta) \rightarrow \bbdelta_0(T, \Delta) \rightarrow \Ext^2_R(\widetilde{H}^2_{\mathrm{f}}(G_{F,\Sigma},T,\Delta),R) \rightarrow 0 .$$
			As a consequences we have $$\Char(\widetilde{H}^1_{\mathrm{f}}(G_{F,\Sigma},T,\Delta)/\bbdelta_1(T,\Delta)) =  \Char\widetilde{H}^2_{\mathrm{f}}(G_{F,\Sigma},T,\Delta).$$	
		\end{thm}
		\begin{proof}
			Let $I \subset R$ be an ideal such that $R/I$ is a 0-dimensional Gorenstein local ring. By Proposition 2.5. we have an exact sequence $R/I$-modules
			\begin{multline*}
				0 \rightarrow \Tor_2^R(\widetilde{H}^2_{\mathrm{f}}(G_{F,\Sigma},T,\Delta), R/I) \to \widetilde{H}^1_{\mathrm{f}}(G_{F,\Sigma},T,\Delta) \otimes_R R/I \\
				\rightarrow \widetilde{H}^1_{\mathrm{f}}(G_{F,\Sigma},T/IT,\Delta) \rightarrow \Tor_1^R(\widetilde{H}^2_{\mathrm{f}}(G_{F,\Sigma},T,\Delta), R/I) \rightarrow 0
			\end{multline*}
			Since $R/I$ is injective, we obtain the following exact sequence
			\begin{multline*}
				0 \rightarrow \Ext^1_R(\widetilde{H}^2_{\mathrm{f}}(G_{F,\Sigma},T,\Delta), R/I) \rightarrow \Hom_{R/I}(\widetilde{H}^1_{\mathrm{f}}(G_{F,\Sigma},T/IT,\Delta),R/I) \\
				\rightarrow \Hom_R(\widetilde{H}^1_{\mathrm{f}}(G_{F,\Sigma},T,\Delta), R/I) \rightarrow \Ext^2_R(\widetilde{H}^2_{\mathrm{f}}(G_{F,\Sigma},T,\Delta), R/I) \rightarrow 0
			\end{multline*}
			Let $\xi_I$ denote the image of 
			$$  \Hom_{R/I}(\widetilde{H}^1_{\mathrm{f}}(G_{F,\Sigma},T/IT,\Delta),R/I) \rightarrow \Hom_R(\widetilde{H}^1_{\mathrm{f}}(G_{F,\Sigma},T,\Delta), R/I) $$
			Then by Proposition \ref{Fitt}. $$ \Fitt^0_{R/I}(\widetilde{H}^2_{\mathrm{f}}(G_{F,\Sigma},T/IT,\Delta)) = \{\phi(\delta) \mid \delta \in \bbdelta_2(T,\Delta), \; \phi \in \xi_I\} $$
			Let $\bbdelta_1(T,\Delta)$ is generated by a regular element $\delta $, we have 
			\begin{equation*}
				\begin{split}
					\Hom_R(\widetilde{H}^1_{\mathrm{f}}(G_{F,\Sigma},T,\Delta), R) \xrightarrow{\sim} \bbdelta_0(T,\Delta) &; \quad \phi \mapsto \phi(\delta), \\
					\varprojlim_I \xi_I \xrightarrow{\sim} \Fitt^0_{R}(\widetilde{H}^2_{\mathrm{f}}(G_{F,\Sigma},T,\Delta)) &; \quad \phi \mapsto \phi(\delta)
				\end{split}
			\end{equation*}
			which gives the required short exact sequence. The consequence follows from Theorem \ref{H2tors}.
		\end{proof}
		\subsubsection{}
		Let's study the case when $r=2$. 
		\begin{lemma}
			Suppose $R$ is regular and $\bbdelta_2(T,\Delta) \ne 0 $. Additionaly if we have $ \widetilde{H}^1_{\mathrm{f}}(G_{F,\Sigma},T,\Delta) $ is free of rank two, then 
			$$ \left(\mathrm{char}_R\left(\bigcap^2_R\widetilde{H}^1_{\mathrm{f}}(G_{F,\Sigma},T,\Delta)/\bbdelta_2(T,\Delta)\right) \right)^2 \cong \mathrm{char}_R\left(\widetilde{H}^1_{\mathrm{f}}(G_{F,\Sigma},T,\Delta)/\bbdelta_1(T,\Delta)\right).$$
		\end{lemma}
		\begin{proof}
			Let $\mathfrak{P}$ be a height 1 prime of $ R $. So $ R_\mathfrak{P} $ is a discrete valuation ring and $ \pi \in R_\mathfrak{P} $ is a uniformizer. We can choose a homomorphism $ \varPhi = (\varPhi_1, \cdots \varPhi_a) $ such that we have a short exact sequence 
			$$ 0 \rightarrow \widetilde{H}^1_{\mathrm{f}}(G_{F,\Sigma},T,\Delta) \otimes R_\mathfrak{P} \rightarrow R_\mathfrak{P}^{a+2} \xrightarrow{\varPhi} R_\mathfrak{P}^{a} \rightarrow \widetilde{H}^2_{\mathrm{f}}(G_{F,\Sigma},T,\Delta) \otimes R_\mathfrak{P} \rightarrow 0 .$$
			In particular, we can choose $ \varPhi $ such that $ \varPhi_i = \pi^{n_i}\cdot \mathrm{pr}_i,\: n_i \in \Z_{\ge0}$ for all $ \: 1 \le i \le a $, where $\mathrm{pr}_i \colon R_\mathfrak{P}^{a+2} \rightarrow R_p$ is projection onto $ i$-th component. If we set $\{ e_1, \cdots, e_{a+2}\}$ as the standard basis of $R_\mathfrak{P}^{a+2} $, we get $\widetilde{H}^1_{\mathrm{f}}(G_{F,\Sigma},T,\Delta) \otimes R_\mathfrak{P} \cong R_\mathfrak{P} e_{a+1} + R_\mathfrak{P} e_{a+2}$ and 
			\begin{equation*}
				\begin{split}
					\bbdelta_2(T,\Delta) \otimes R_\mathfrak{P} &= \mathrm{im}\left(\det(R_\mathfrak{P}^{a+2}) \rightarrow \bigwedge_{R_\mathfrak{P}}^2 R_\mathfrak{P}^{a+2}\right) \\
					&= \pi^{\sum n_i} \bigwedge_{R_\mathfrak{P}}^2(\widetilde{H}^1_{\mathrm{f}}(G_{F,\Sigma},T,\Delta) \otimes_R R_\mathfrak{P}).
				\end{split}
			\end{equation*}   
			Hence by the definition, we have $ \bbdelta_1(T,\Delta) \otimes_R R_\mathfrak{P} = \pi^{\sum n_i}(R_\mathfrak{P} e_{a+1} + R_\mathfrak{P} e_{a+2}) $. Therefore we conclude that 
			\begin{equation*}
				\begin{split}
					\mathrm{char}_R\left(\widetilde{H}^1_{\mathrm{f}}(G_{F,\Sigma},T,\Delta)/\bbdelta_1(T,\Delta)\right) \otimes_R R_\mathfrak{P} &= \pi^{2 \cdot \sum n_i} R_\mathfrak{P} \\
					&=\left(\mathrm{char}_R\left(\bigcap^2_R\widetilde{H}^1_{\mathrm{f}}(G_{F,\Sigma},T,\Delta)/\bbdelta_2(T,\Delta)\right) \otimes_R R_\mathfrak{P}\right)^2 \\
				\end{split}
			\end{equation*}
			which gives the required result.
		\end{proof}
		Moreover, we can conclude that,
		\begin{cor}
			\label{r2}
			If $ R $ is regular and $ \bbdelta_2(T,\Delta) \ne 0 $ then we have $$\Char_R(\widetilde{H}^1_{\mathrm{f}}(G_{F,\Sigma},T,\Delta)/\bbdelta_1(T,\Delta)) =  \Char_R(\widetilde{H}^2_{\mathrm{f}}(G_{F,\Sigma},T,\Delta))^2.$$
		\end{cor}
		\begin{thm}
			\label{nontorsion}
			Suppose that $\bbdelta_2(T, \Delta)$ is generated by a $R$-regular element, we then have a short exact sequence  
			\begin{equation*}
				0 \rightarrow \Fitt^0_R \widetilde{H}^2_{\mathrm{f}}(G_{F,\Sigma},T,\Delta) \rightarrow \bbdelta_0(T, \Delta) \rightarrow \Ext^2_R(\widetilde{H}^2_{\mathrm{f}}(G_{F,\Sigma},T,\Delta),R) \rightarrow 0
			\end{equation*}
		\end{thm}
		\begin{proof}
			The proof is similar to the proof of Theorem \ref{leading_term_rank1}.
		\end{proof}
		\begin{rmk}
			\label{koly_ralation}
			Let us assume the standard assumption of the theory of \emph{Euler-Kolyvagin systems}, like:
			\begin{enumerate}[(i)]
				\item The Galois representation $ \G \to \GL{R}(T) $ has big image.
				\item For each $ p \in \bP $, we have $H^0(G_{F_{\p}}, \scrF^{\pm}_{\p}\overline{T}) = 0 =H^2(G_{F_{\p}}, \scrF^{\pm}_{\p}\overline{T}) $ .
				\item The complex $ R\Gamma(G_{F_v},T) $ is acyclic for each $ v \in \Sigma \ba \bP $.
			\end{enumerate}
			The theory gives us that the module of Kolyvagin systems of rank $ r $, $ \mrm{KS}_r(T,\Delta) $ associated with the pair $ (T,\Delta) $ is free of rank one (cf. \cite[Theorem 5.2.10]{MR_kolysys}, \cite[Theorem A]{kazim_kolysys}, and \cite[Theorem 5.2]{BSS18}). Moreover, we have a homomorphism,
			\[ \mrm{KS}_r(T,\Delta) \xrightarrow{\operatorname{Reg}_r} \cohom^1(\G,T,\Delta); \quad \kappa \mapsto \kappa_1 \ \text{(the leading class of the Kolyvagin system $ \kappa $)}. \]
			The assumption that $ \cohom^2(\G,T,\Delta) $ is $ R $-torsion implies that $ \operatorname{Reg}_r $ is injective and 
			\[ \Char_R(\bigcap_R^r\cohom^1(\G,T,\Delta)/ \kappa_r(T,\Delta)) = \Char_R(\cohom^2(\G,T,\Delta)); \quad \kappa_r(T,\Delta)\coloneqq \operatorname{Reg}_r(\mrm{KS}_r(T,\Delta)). \]
			and hence $ \kappa_r(T,\Delta) = \bbdelta_r(T,\Delta) $ by Thorem \ref{H2tors}. In particular, we can say that the module $ \bbdelta_r(T,\Delta) $ is generated by the leading classes of Kolyvagin systems. We also state that the construction of $ \bbdelta_r(T,\Delta) $ is done with weaker assumptions concerning Tamagawa factors, than the standard assumptions of the theory of Kolyvagin systems. 
			
			We can expect that the module of leading terms is generated by suitable Euler system classes at the ground level. Hence via a suitable large Perrin-Riou logarithm map, it recovers analytic $ p $-adic $ L $-functions.
		\end{rmk}
		\subsubsection{}
		Let $ R $ be a regular local ring for the rest of the section. 
		Suppose for each $\p \mid p$, we consider a pair of $R[G_{F_{\p}}]$-submodules $\scrF^{+1}_{\p}T \subset \scrF^{+2}_{\p}T$ of $T$ such that the quotients $T/\scrF^{+1}_{\p}T $ and $T/\scrF^{+2}_{\p}T $ are free as $R$-modules. We then have the Greenberg local conditions
		$\Delta_1 \coloneqq \Delta(\{\scrF^{+1}_\p\}_\p)$ and $\Delta_2 \coloneqq \Delta(\{\scrF^{+2}_\p\}_\p)$. Also, assume that 
		\begin{equation*}
			\begin{split}
				r(T,\Delta_1) \coloneqq -\chi(\widetilde{R\Gamma}_{\mathrm{f}}(G_{F,\Sigma},T,\Delta_1)) = \sum_{w \mid \infty} \rank_R(T^{c_w = -1}) - \sum_{\p \mid p} \rank_R(T/\scrF^{+1}_{\p}T) =0 
			\end{split}
		\end{equation*}
		When $$ 	r(T,\Delta_2) \coloneqq -\chi(\widetilde{R\Gamma}_{\mathrm{f}}(G_{F,\Sigma},T,\Delta_2)) = \sum_{w \mid \infty} \rank_R(T^{c_w = -1}) - \sum_{\p \mid p} \rank_R(T/\scrF^{+2}_{\p}T) = 1 $$
		Let's also assume that for all $\p \mid p$
		\begin{itemize}
			\item[\mylabel{LT}{\textbf{(LT)}}] 	$H^0(G_{F_{\p}}, \scrF^{+2}_{\p}\overline{T}/\scrF^{+1}_{\p}\overline{T}) = 0 =H^2(G_{F_{\p}}, \scrF^{+2}_{\p}\overline{T}/\scrF^{+1}_{\p}\overline{T}). $ 
		\end{itemize}
		WLOG we assume that the $R$-module $H^1(G_{F_{\p_1}}, \scrF^{+2}_{\p_1}T/\scrF^{+1}_{\p_1}T)$ is free of rank one by the Euler characteristic formula and $H^1(G_{F_{\p_2}}, \scrF^{+2}_{\p_2}T/\scrF^{+1}_{\p_2}T)$ is trivial. Let us fix a trivialization
		$$ \LOG_{\Delta_2/\Delta_1, \p_1 } \colon H^1(G_{F_{\p_1}}, \scrF^{+2}_{\p_1}T/\scrF^{+1}_{\p_1}T) \xrightarrow{\sim} R $$
		We denote the composition of the following maps as $ \LOG^1_{\Delta_2/\Delta_1} $
		$$  \widetilde{H}^1_{\mathrm{f}}(G_{F,\Sigma},T,\Delta_2) \xrightarrow{ \res_{\p_1}}  H^1(G_{F_{\p_1}}, \scrF^{+2}_{\p_1}T) \rightarrow  H^1(G_{F_{\p_1}}, \scrF^{+2}_{\p_1}T/\scrF^{+1}_{\p_1}T) \xrightarrow{\LOG_{\Delta_2/\Delta_1 , \p_1} } R^2 $$
		\begin{lemma}
			\label{lemma_log_H^1_vanishes}
			If $\LOG^1_{\Delta_2/\Delta_1}(\bbdelta(T, \Delta_2))$ is generated by a regular element of $R$,  
			then $\widetilde{H}^1_{\rm f}(G_{F, \Sigma}, T, \Delta_1) = 0$. 
		\end{lemma}
		\begin{proof}
			By Theorem \ref{H2tors}, the $R$-module $\widetilde{H}^2_{\rm f}(G_{F, \Sigma}, T, \Delta_2)$ is torsion. 
			By assumption $R/\LOG^1_{\Delta_2/\Delta_1}(\bbdelta(T, \Delta_2))$ is torsion gets 
			$\widetilde{H}^2_{\rm f}(G_{F, \Sigma}, T, \Delta_1)$ is $ R $-torsion. 
			As $\chi(\widetilde{R\Gamma}_{\rm f}(G_{F, \Sigma}, T, \Delta_1)) = 0$, we have $\widetilde{H}^1_{\rm f}(G_{F, \Sigma}, T, \Delta_1) = 0$. 
		\end{proof}
		\begin{thm}
			\label{old_log_relation}
			Suppose that $ R $ is regular, then we have 
			$$ \LOG^1_{\Delta_2/\Delta_1} (\bbdelta_1(T, \Delta_2)) = \Char_R(\widetilde{H}^2_{\mathrm{f}}(G_{F,\Sigma},T,\Delta_1)) $$
			whenever $ \widetilde{H}^2_{\mathrm{f}}(G_{F,\Sigma},T,\Delta_1) $ is torsion.
		\end{thm}
		\begin{proof}
			If $\LOG^1_{\Delta_2/\Delta_1}(\bbdelta(T, \Delta_2)) \neq 0$, then Lemma \ref{lemma_log_H^1_vanishes} shows that 
			$\widetilde{H}^2_{\rm f}(G_{F, \Sigma}, T, \Delta_1) = 0$ and we have an exact sequence  
			\begin{align*}
				0 \longrightarrow \widetilde{H}^1_{\rm f}(G_{F, \Sigma}, T, \Delta_2)/\bbdelta(T, \Delta_2)  
				\longrightarrow R/ \LOG^1_{\Delta_2/\Delta_1}(\bbdelta(T, \Delta_2)) \qquad \qquad \qquad \qquad \qquad \\
				\longrightarrow \widetilde{H}^2_{\rm f}(G_{F, \Sigma}, T, \Delta_1)
				\longrightarrow \widetilde{H}^2_{\rm f}(G_{F, \Sigma}, T, \Delta_2) 
				\longrightarrow 0. 
			\end{align*}
			of torsion $R$-modules. Theorems \ref{r0} and \ref{leading_term_rank1}
			imply 
			\begin{align*}
				\LOG^1_{\Delta_2/\Delta_1}(\bbdelta(T, \Delta_2)) &= \Char_R(R/ \LOG^1_{\Delta_2/\Delta_1}(\bbdelta(T, \Delta_2)))
				\\
				&= \Char_R(\widetilde{H}^2_{\rm f}(G_{F, \Sigma}, T, \Delta_1)). 
			\end{align*}	
		\end{proof}
		\subsubsection{}
		Now suppose	
		$$
		r(T,\Delta_2) \coloneqq -\chi(\widetilde{R\Gamma}_{\mathrm{f}}(G_{F,\Sigma},T,\Delta_2)) = \sum_{w \mid \infty} \rank_R(T^{c_w = -1}) - \sum_{\p \mid p} \rank_R(T/\scrF^{+2}_{\p}T) =2 $$
		For all $\p= \p_1, \p_2 $ we assume the condition \ref{LT}. We also assume 
		\begin{itemize}
			\item[\mylabel{R1}{\textbf{(R1)}}] $ \scrF^{+2}_{\p}T/\scrF^{+1}_{\p}T $ is free $ R $-module of rank 1 for $ \p= \p_1, \p_2  $.
		\end{itemize}
		Then the $R$-module $H^1(G_{F_{\p}}, \scrF^{+2}_{\p}T/\scrF^{+1}_{\p}T)$ is free of rank one by the Euler characteristic formula. Let us fix a trivialization 
		$$ \LOG_{\Delta_2/\Delta_1 , \p} \colon H^1(G_{F_{\p}}, \scrF^{+2}_{\p}T/\scrF^{+1}_{\p}T) \xrightarrow{\sim} R $$
		Let the composition of the following maps be defined as $\LOG^2_{\Delta_2/\Delta_1}$
		$$  \widetilde{H}^1_{\mathrm{f}}(G_{F,\Sigma},T,\Delta_2) \xrightarrow{\bigoplus_{\p \mid p} \res_{\p}} \bigoplus_{\p \mid p} H^1(G_{F_{\p}}, \scrF^{+2}_{\p}T) \rightarrow \bigoplus_{\p \mid p} H^1(G_{F_{\p}}, \scrF^{+2}_{\p}T/\scrF^{+1}_{\p}T) \xrightarrow{\bigoplus_{\p \mid p}\LOG_{\Delta_2/\Delta_1 , \p} } R^2 $$
		We then have an exact sequence 
		\begin{multline*}
			0 \rightarrow \widetilde{H}^1_{\mathrm{f}}(G_{F,\Sigma},T,\Delta_1) \rightarrow \widetilde{H}^1_{\mathrm{f}}(G_{F,\Sigma},T,\Delta_2) \xrightarrow{\LOG^2_{\Delta_2/\Delta_1}} R^2
			\rightarrow \widetilde{H}^2_{\mathrm{f}}(G_{F,\Sigma},T,\Delta_1) \rightarrow \widetilde{H}^2_{\mathrm{f}}(G_{F,\Sigma},T,\Delta_2) \rightarrow 0.
		\end{multline*} 
		
		\begin{lemma}
			\label{logim}
			If for all $ \p \mid p $, $ \LOG^2_{\Delta_2/\Delta_1 , \p} (\bbdelta_1(T, \Delta_2)) $ is generated by some $ R $-regular element, then  $$ \widetilde{H}^1_{\mathrm{f}}(G_{F,\Sigma},T,\Delta_1) = 0.$$ 
		\end{lemma}	
		\begin{proof}
			
			By the Theorem \ref{H2tors}, $ \widetilde{H}^2_{\mathrm{f}}(G_{F,\Sigma},T,\Delta_2) $ is torsion. Since the quotient $$ R^2/\LOG_{\Delta_2/\Delta_1}(\bbdelta_1(T, \Delta_2)) $$	is torsion, $ \widetilde{H}^2_{\mathrm{f}}(G_{F,\Sigma},T,\Delta_1) $ is torsion. As $ \chi(\widetilde{R\Gamma}_{\mathrm{f}}(G_{F,\Sigma},T,\Delta_1)) =0 $, from the exact sequence \eqref{ses} we get $ \widetilde{H}^1_{\mathrm{f}}(G_{F,\Sigma},T,\Delta_1) = 0$.
		\end{proof}

		\begin{thm}
			\label{logrln}
			Suppose that $R$ is regular. If $ \LOG^2_{\Delta_2/\Delta_1 , \p} (\bbdelta_1(T, \Delta_2)) \ne 0 $ for all $ \p \mid p $, then we have 
			$$ \bigotimes_{\p \mid p}  \LOG_{\Delta_2/\Delta_1 , \p} (\bbdelta_1(T, \Delta_2)) = \Char_R(\widetilde{H}^2_{\mathrm{f}}(G_{F,\Sigma},T,\Delta_1)) \Char_R(\widetilde{H}^2_{\mathrm{f}}(G_{F,\Sigma},T,\Delta_2)) $$
		\end{thm}
		\begin{proof}
			If $ \LOG_{\Delta_2/\Delta_1 , \p} (\bbdelta_1(T, \Delta_2)) \ne 0 $ then Lemma \ref{logim}. shows that $ \widetilde{H}^1_{\mathrm{f}}(G_{F,\Sigma},T,\Delta_1) = 0 $. Hence we have the following exact sequence of torsion $ R $-modules 
			\begin{align*}
				0 \rightarrow \widetilde{H}^1_{\mathrm{f}}(G_{F,\Sigma},T,\Delta_2)/ \bbdelta_1(T,\Delta_2) \rightarrow R^2/\LOG^2_{\Delta_2/\Delta_1}(\bbdelta_1(T,\Delta_2)) \qquad \qquad \qquad\qquad \\
				\rightarrow \widetilde{H}^2_{\mathrm{f}}(G_{F,\Sigma},T,\Delta_1) \rightarrow \widetilde{H}^2_{\mathrm{f}}(G_{F,\Sigma},T,\Delta_2) \rightarrow 0
			\end{align*} 
			Then by Theorem \ref{r0}. and Corollary \ref{r2}. we conclude that\\
			$\bigotimes_{\p \mid p}  \LOG_{\Delta_2/\Delta_1 , \p} (\bbdelta_1(T, \Delta_2)) = \Char_R(\widetilde{H}^2_{\mathrm{f}}(G_{F,\Sigma},T,\Delta_1)) \Char_R(\widetilde{H}^2_{\mathrm{f}}(G_{F,\Sigma},T,\Delta_2)) $.
		\end{proof}
		\begin{rmk}
			\cref{old_log_relation} and \cref{logrln} are formal variants of implications of Perrin-Riou's conjecture about comparing Beillinson-Kato elements and Heegner points.
		\end{rmk}


		\section{Factorization of algebraic $p$-adic $L$-functions (The rank (1,1) case)}
		\label{section : factorization_1}
		In this section, our objective is to establish one of our main factorization results stated in \cref{theoremA}. 
		\subsection{Modules of leading terms in rank (1,1) case}
		The factorization formula will be proved by comparing the following modules of leading terms.
		\subsubsection{Hypotheses}
		\label{rank1_hypo}
		We will work throughout the \cref{section : factorization_1} under the hypotheses as follows. 
		
		The assumptions of Corollary \ref{cor_tam_assump} are enforced both with $ \f $ and $ \g $. Also, note that if the tame conductors of the Hida families $ \f $ and $ \g $ are square free then this condition trivially holds. Along with that, the condition \ref{DIST} holds for holds for both  $ \f $ and $ \g $. The Hida family $ \f $ satisfies \ref{Irr}, but for $ \g $ we assume the following stronger version i.e.
		\begin{itemize}
			\item[\mylabel{Irr+}{\textbf{(Irr+)}}] The residual representation $ \bar{\rho}_{\g} $ is absolutely irreducible when restricted to $ G_{F(\zeta_p)} $.
			\item[\mylabel{TAM}{\textbf{(TAM)}}] We further assume that the  $ p $-part of a Tamagawa factor for a single member of the Hida family $ \f $ and a single member of the family $ \g $ is trivial.
		\end{itemize}
		\begin{rmk}
			\label{hypo_rmk}
			These hypotheses ensure the validity of \ref{H=0} and \ref{Tam} for $ \T_{1},  \T_{2} $, and $ \T_{3} $. Therefore our main conclusions in \cref{section : Selmer_complexes} are applicable and the Selmer complexes associated with these Galois representations are perfect.
		\end{rmk}
		In this section we work in the scenario of \textsection\ref{rank1_scenarion}, so that we have 
		\begin{itemize}
			\item[\mylabel{Sign1}{\textbf{(S1)}}] $ \varepsilon(\f_P \otimes \lambda) = -1 $ and  $ \varepsilon(\f_P \otimes \ad^0(\g_Q)(\lambda) )= 
			-1$ whenever $(P,Q) \in \W_{\mrm{cl}}^{(\f,\b)}	$ for the global root number.
		\end{itemize}
		We also need the following conditions to compare the module of leading terms associated with $(T_{\f}(\lambda),\Delta_{(\emptyset, \pan)})$:
		\begin{itemize}
			\item[\mylabel{ord}{\textbf{(Ord)}}] $\f \otimes \lambda$ is ordinary at $ \p_1 $ and $ \p_2 $ both.
			\item[\mylabel{BI}{\textbf{(BI)}}] $\f \otimes \lambda$ satisfies the \enquote{big image} conditions of \cite[Definition 2.5.1]{loeffler2020iwasawatheoryquadratichilbert}. 
			\item[\mylabel{MC}{\textbf{(MC)}}] The family $\f \otimes \lambda $ verifies the conditions of \cite[Theorem 101]{Wan}.
			\item[\mylabel{NA}{\textbf{(NA)}}] $ \alpha_{\f \otimes \lambda,\p}(\mrm{Fr}_{\frak{p}}) -1 \in \R_{\f}^\times $ for each $ \p = \p_1, \p_2 $.
		\end{itemize}
		
		\begin{rmk}
			\label{remark_hypo}
			We note that the hypotheses \ref{ord}, \ref{BI} and \ref{MC} are required to validate the Iwasawa main conjecture for $ \f \otimes \lambda $ (cf. \cite{loeffler2020iwasawatheoryquadratichilbert}, \cite{Wan}). The \ref{NA} hypothesis is there to ensure certain local cohomology groups are finite.
		\end{rmk}

		\subsubsection{}
		\label{6.1.2}
		We have $ \chi(\widetilde{R\Gamma}_{\mathrm{f}}(G_{F, \Sigma}, T_\f(\lambda), \Delta_{\emptyset,\pan})) = -1 $. Hence there will be a cyclic submodule 
		$$\bbdelta_1(T_{\f \otimes \lambda}, \Delta_{\emptyset,\pan}) \subset \widetilde{H}^1_{\mathrm{f}}(G_{F,\Sigma},T_\f(\lambda),\Delta_{\emptyset,\pan})$$
		\begin{rmk}
			$\bbdelta_1(T_{\f \otimes \lambda}, \Delta_{(\emptyset,\pan)})$ is expected to be generated by Kolyvagin system $\mathbf{BK}^{\mrm LZ}_{\f\otimes \lambda}$ coming from the big Euler system classes constructed by Loeffler et al. and hence non zero. (cf. \cite[Section 8]{loeffler2020iwasawatheoryquadratichilbert}, \cite{loeffler2024blochkatoconjecturegsp4})
		\end{rmk}
		
		\begin{prop} We have
			$$\widetilde{R\Gamma}_{\mathrm{f}}(G_{F,\Sigma},T_\f(\lambda), \Delta_{(0,\pan)}), \; \widetilde{R\Gamma}_{\mathrm{f}}(G_{F,\Sigma},T_\f(\lambda), \Delta_{(\emptyset,\pan)}) \in D_{\mathrm{parf}}^{[1,2]}(\prescript{}{\R_\f}{\mathrm{Mod}})$$
			If $\bbdelta_1(T_\f(\lambda), \Delta_{(\emptyset,\pan)}) \neq 0$, then:
			\begin{enumerate}
				\item If the image of $\bbdelta_1(T_\f(\lambda), \Delta_{(\emptyset,\pan)})$ under the composite map 
				$$  \res_{\p_1}^-  \colon H^1(G_{F, \Sigma}, T_\f(\lambda)) \rightarrow H^1(F_{\p_1}, T_\f(\lambda)) \rightarrow  H^1(F_{\p_1}, \scrF^-_{\p_1} T_\f(\lambda)) $$ is zero then $$ \widetilde{H}^1_{\mathrm{f}}(G_{F,\Sigma},T_\f(\lambda),\Delta_{\pan}) \cong \widetilde{H}^1_{\mathrm{f}}(G_{F,\Sigma},T_\f(\lambda),\Delta_{(\emptyset,\pan)}) $$ are both free $ \R_{\f} $-modules of rank one. 
				\item $\widetilde{H}^2_{\mathrm{f}}(G_{F,\Sigma},T_\f(\lambda),\Delta_{(\emptyset,\pan)})$ is torsion, Moreover,
				\begin{align}
					\begin{aligned}
						\Char_R(\widetilde{H}^1_{\mathrm{f}}(G_{F,\Sigma},T_\f(\lambda),\Delta_{(\emptyset,\pan)})/\R_{\f}\bbdelta_1(T_\f(\lambda), \Delta_{(\emptyset,\pan)})) \\ \noindent = \Char_R(\widetilde{H}^2_{\mathrm{f}}(G_{F,\Sigma},T_\f(\lambda),\Delta_{(\emptyset,\pan)})) .
					\end{aligned}
				\end{align}
				\item If in addition $ \res_{\p_1}(\bbdelta_1(T_\f(\lambda), \Delta_{(\emptyset,\pan)})) \ne 0 $, then $ \widetilde{H}^1_{\mathrm{f}}(G_{F,\Sigma},T_\f(\lambda),\Delta_{(0,\pan)}) = 0 $ and the $ \R_{\f} $-module $  \widetilde{H}^2_{\mathrm{f}}(G_{F,\Sigma},T_\f(\lambda),\Delta_{(0,\pan)}) $ has rank one.
			\end{enumerate}
		\end{prop}
		\begin{proof}
			The fact that the Selmer complexes are perfect and that they are concentrated in the indicated degrees follows from Proposition \ref{dercat}. \\
			Since $\bbdelta_1(T_\f(\lambda), \Delta_{(\emptyset,\pan)}) \neq 0$, the $ \R_{\f} $-module $ \widetilde{H}^2_{\mathrm{f}}(G_{F,\Sigma},T_\f(\lambda),\Delta_{(\emptyset,\pan)}) $ is torsion by Theorem \ref{H2tors}. Therefor $ \widetilde{R\Gamma}_{\mathrm{f}}(G_{F,\Sigma},T_\f(\lambda), \Delta_{(\emptyset,\pan)}) \in D_{\mathrm{parf}}^{[1,2]}(\prescript{}{\R}{\mathrm{Mod}}) $ and $ \chi(\widetilde{R\Gamma}_{\mathrm{f}}(G_{F,\Sigma},T_\f(\lambda), \Delta_{(\emptyset,\pan)}))=-1 $ imply that rank of $ \widetilde{H}^1_{\mathrm{f}}(G_{F,\Sigma},T_\f(\lambda),\Delta_{(\emptyset,\pan)}) $ equals two. Moreover, since $ \R_{\f} $ is a finite-dimensional regular local ring, for any finitely generated $ \R_{\f} $-module $ M $ we have $$ \Ext^1_{\R_{\f}}(\widetilde{H}^1_{\mathrm{f}}(G_{F,\Sigma},T_\f(\lambda),\Delta_{(\emptyset,\pan)}), M) = 0 = \Ext^2_{\R_{\f}}(\widetilde{H}^2_{\mathrm{f}}(G_{F,\Sigma},T_\f(\lambda),\Delta_{(\emptyset,\pan)}), M). $$ 
			Thence, $ \widetilde{H}^1_{\mathrm{f}}(G_{F,\Sigma},T_\f(\lambda),\Delta_{(\emptyset,\pan)}) $ is free. Since $ \varepsilon(\f \otimes \lambda) =-1 $, the vanishing $ \res^-_{\p_1}(\bbdelta_1(T_\f(\lambda), \Delta_{(\emptyset,\pan)})) = 0 $ should come from comparison of $\bbdelta_1(T_\f, \Delta_{(\emptyset,\pan)})$ with $\mathbf{BK}^{\mrm LZ}_{\f\otimes \lambda}$. Hence we have the equality $$ \widetilde{H}^1_{\mathrm{f}}(G_{F,\Sigma},T_\f(\lambda),\Delta_{\pan}) \cong \widetilde{H}^1_{\mathrm{f}}(G_{F,\Sigma},T_\f(\lambda),\Delta_{(\emptyset,\pan)}). $$ 
			Part (2) is a consequence of Corollary \ref{old_log_relation}.\\
			In addition, if $ \res_{\p_1}(\bbdelta_1(T_\f(\lambda), \Delta_{(\emptyset,\pan)})) \ne 0 $, then $ \widetilde{H}^1_{\mathrm{f}}(G_{F,\Sigma},T_\f(\lambda),\Delta_{(0,\pan)}) = 0 $ is a consequence of conclusion (1). As $ \chi(\widetilde{R\Gamma}_{\mathrm{f}}(G_{F,\Sigma},T_\f(\lambda), \Delta_{(0,\pan)})) = -1  $, we get $ \R_{\f} $-rank of $ \widetilde{H}^2_{\mathrm{f}}(G_{F,\Sigma},T_\f(\lambda),\Delta_{(0,\pan)}) $ is one.
		\end{proof}
		By the theory of Perrin-Riou, we get an isomorphism (non-canonically), induced from a big Perrin-Riou logarithm (cf. \cite{perrin-riou_padic})
		\begin{align}
			\label{big_logarithm}
			\LOG_{(\f \otimes \lambda,\p_1)} \colon H^1(G_{F_{\p_1}}, \scrF^+_{\p_1} T_\f(\lambda)) \xrightarrow{\sim} \mbf D_{\rm crys} (\scrF^+_{\p_1} T_\f(\lambda\chi_{\mathrm{cycl}}\bbl{x}_\f^{-1}).
		\end{align}
		Thus we can choose an isomorphism
		\begin{align}
			\Log_{(\f \otimes \lambda,\p_1)}^1 \colon H^1(G_{F_{\p_1}}, \scrF^+_{\p_1} T_\f(\lambda)) \xrightarrow{\LOG_{(\f \otimes \lambda,\p_1)}} \mbf D_{\rm crys} (\scrF^+_{\p_1} T_\f(\lambda\chi_{\mathrm{cycl}}\bbl{x}_\f^{-1}) \xrightarrow{\sim} \R_\f.
		\end{align}
		Here the choice of isomorphism $ \mbf D_{\rm crys} (\scrF^+_{\p_1} T_\f(\lambda\chi_{\mathrm{cycl}}\bbl{x}_\f^{-1}) \xrightarrow{\sim} \R_\f $  is compatible with the choice of lattice $ T_\f $ (cf. \cite[Proposition 10.1.1]{kings_loeffler_zerbes17}). We fix this choice throughout the section. 
		\subsubsection{}
		\label{6.1.3}
		We now focus on $\T_{2} $. Since we have $ \chi(\widetilde{R\Gamma}_{\mathrm{f}}(G_{F,\Sigma}, \T_2, \operatorname{tr}^*\Delta_{(\g,\b)})) = -1 $, we get a cyclic submodule
		$$ \bbdelta_1(\T_2, \operatorname{tr}^*\Delta_{(\g,\b)}) \subset  \widetilde{H}^1_{\mathrm{f}}(G_{F,\Sigma}, \T_2, \operatorname{tr}^*\Delta_{(\g,\b)}) $$
		On the other hand, note that $ \chi(\widetilde{R\Gamma}_{\mathrm{f}}(G_{F,\Sigma},\T_2, \operatorname{tr}^*\Delta_{\mathrm{bal}})) = 0 $.
		Thus there exists a submodule 
		$$ \bbdelta_0(\T_2, \operatorname{tr}^*\Delta_{\mathrm{bal}}) \subset \R$$
		The short exact sequence \ref{ses_sep18} gives a map
		\begin{equation}
			\res_{/\mathrm{bal}, \p_1} \colon \widetilde{H}^1_{\mathrm{f}}(G_{F,\Sigma}, \T_2, \operatorname{tr}^*\Delta_{(\g,\b)})\rightarrow H^1(G_{F_{\p_1}},\scrF^-_{\p_1} T_{\f}(\lambda)) \otimes_O \R.
		\end{equation}
		Let us fix a trivialization\footnote{This choice of trivialization is given in terms of Perrin-Riou's big dual exponential map compatible with the choice of $ \LOG_{(\f \otimes \lambda,\p_1)} $ in \eqref{big_logarithm}.} 
		$$ \operatorname{EXP}^*_{\scrF^-_{\p_1} T_{\f}(\lambda)} \colon H^1(G_{F_{\p_1}},\scrF^-_{\p_1} T_{\f}(\lambda)) \xrightarrow{\sim} \R_\f.$$
		We define the following composition as
		\begin{align}
			\begin{aligned}
				\Exp_{\scrF^-_{\p_1}\f (\lambda)}^* \colon \widetilde{H}^1_{\mathrm{f}}(G_{F,\Sigma}, \T_2, \operatorname{tr}^*\Delta_{(\g,\b)})\xrightarrow{\res_{/\mathrm{bal}, \p_1}} H^1(G_{F_{\p_1}},\scrF^-_{\p_1} T_{\f}(\lambda)) \otimes_O \R \qquad \\
				\xrightarrow{\operatorname{EXP}^*_{\scrF^-_{\p_1} T_{\f}(\lambda)} \otimes \operatorname{id}} \R.		\end{aligned}
		\end{align}
		\begin{rmk} \label{rmk : Exp_T2_rank1}
			If $ \bbdelta_0(\T_2, \operatorname{tr}^*\Delta_{\mathrm{bal}}) \ne 0 $, Theorem \ref{old_log_relation}. tells 
			\begin{equation}
				\Exp^*_{\scrF^-_{\p_1} T_{\f}(\lambda)} (\bbdelta_1(\T_{2}, \tr \Delta_{(\g,\b)})) = \Char(\widetilde{H}^2_{\mathrm{f}}(G_{F,\Sigma}, \T_2, \operatorname{tr}^*\Delta_{\mathrm{bal}}))
			\end{equation}
			is the module of (algebraic) $ p $-adic $ L $-functions, which is expected to be generated by the $ p $-adic $ L $-function $ \scr{L}_p^{(\g,\b)}(\f \otimes \breve{\g} ) $. 
		\end{rmk}
		We note that $\widetilde{R\Gamma}_{\mathrm{f}}(G_{F,\Sigma},\T_2, \operatorname{tr}^*\Delta_?) \in D_{\mathrm{parf}}^{[1,2]}(\prescript{}{\R}{\mathrm{Mod}}) $ for $? \in \{(\g, \b), \b\} $ and hence 
		\begin{equation}
			\widetilde{H}^1_{\mathrm{f}}(G_{F,\Sigma}, \T_2, \operatorname{tr}^*\Delta_{\mathrm{bal}}) = \{0\}, \quad \widetilde{H}^2_{\mathrm{f}}(G_{F,\Sigma}, \T_2, \operatorname{tr}^*\Delta_{\mathrm{bal}}) \text{ is torsion,}
		\end{equation}
		\begin{align}
			\begin{aligned}
				&\widetilde{H}^1_{\mathrm{f}}(G_{F,\Sigma}, \T_2, \operatorname{tr}^*\Delta_{(\g,\b)}) \text{ is torsion-free of rank one,} \\ 
				&\widetilde{H}^2_{\mathrm{f}}(G_{F,\Sigma}, \T_2, \operatorname{tr}^*\Delta_{(\g,\b)}) \text{ is torsion,} \\
				&\widetilde{H}^1_{\mathrm{f}}(G_{F,\Sigma}, \T_2, \operatorname{tr}^*\Delta_{(\g,\b)})\rightarrow H^1(G_{F_{\p_1}},\scrF^-_{\p_1} T_{\f}(\lambda)) \otimes \R \xrightarrow{\sim} \R \text{ is injective}
			\end{aligned}
		\end{align}
		whenever $  \bbdelta_1(\T_{2}, \tr \Delta_{(\g,\b)})) \ne 0 $.

		\subsubsection{}
		Note that we have $\chi(\widetilde{R\Gamma}_{\mathrm{f}}(G_{F,\Sigma},\T_{3},\Delta_{(\g,\b)})) =0$. Therefore we get submodule of leading terms
		\begin{equation}
			\bbdelta_0(\T_{3},\Delta_{(\g,\b)})  \subset \R.
		\end{equation}
		Suppose $ \bbdelta_0(\T_{3},\Delta_{(\g,\b)}) \ne 0 $, then by \cref{r0} 
		\begin{align}
			\label{6.10}
			\bbdelta_0(\T_{3},\Delta_{(\g,\b)}) = \Char_R\left(\cohom^2(G_{F,\Sigma},\T_{3},\Delta_{(\g,\b)})\right).
		\end{align} 
		Also note that for the local conditions $ \Delta_{(+,\b)} $, we have $ \chi(\widetilde{R\Gamma}_{\mathrm{f}}(G_{F,\Sigma}, \T_{3}, \Delta_{(+,\b)})) =-1 $. Hence we get a submodule $$ \bbdelta_1(\T_{3},\Delta_{(+,\b)}) \subset \widetilde{H}^1(G_{F,\Sigma}, \T_{3}, \Delta_{(+,\b)}) $$
		Let us consider the natural map $ \res^{/\g}_{\p_1} $ as the composition of 
		\begin{equation*}
			\widetilde{H}^1(G_{F,\Sigma}, \T_{3}, \Delta_{(+,\b)}) \rightarrow  H^1(G_{F_{\p_1}}, \scrF^+_{\p_1}T_{\f} \otimes \scrF^-_{\p_1}T_{\g} \otimes \scrF^+_{\p_1}T_{\g}^*(\lambda)) \xrightarrow{\sim}  H^1(G_{F_{\p_1}}, \scrF^+_{\p_1}T_{\f}(\lambda)) \widehat{\otimes} \R ,
		\end{equation*}
		induced by the inclusion $ \scrF^{+\g}_{\p_1} \T_{3} \rightarrow \scrF^{+\mathrm{bal}^+}_{\p_1} \T_{3} $.
		\begin{rmk}
			\label{remark1_T3}
			If we assume that $ \bbdelta_0(\T_{3},\Delta_{\g}) \ne 0 $
			$$  \Log_{\f \otimes \lambda,\p_1}\circ\res^{/\g}_{\p_1}(\bbdelta_1(\T_{3},\Delta_{+,\b})) = \Char (\widetilde{H}^2_{\mathrm{f}}(G_{F,\Sigma}, \T_3, \Delta_{(\g,\b)}))  $$
			by Theorem \ref{old_log_relation}.
		\end{rmk}
		\begin{prop}
			If $ \bbdelta_0(\T_{3}.\Delta_{(\g,\b)}) \ne 0 $, then also 
			$$ \res_{\p_1}(\bbdelta_1(\T_{1}, \Delta_{(\emptyset,\pan)})) \ne 0 \ne \bbdelta_1(\T_{2}, \tr\Delta_{(\g,\b)}). $$
		\end{prop}
		\begin{proof}
			We have the natural map
			$$ \widetilde{H}^1_{\rm f}(G_{F,\Sigma}, \T_1, \Delta_{(0,\pan)}) \rightarrow \widetilde{H}^1_{\rm f}(G_{F,\Sigma}, \T_3, \Delta_{(\g,\b)}) $$
			is injective via \ref{comtriangle1}. Therefore by \eqref{6.10}, Remark \ref{remark1_T3} and \cref{r0} we get $ \widetilde{H}_{\rm f}^1(G_{F,\Sigma}, \T_1, \Delta_{(0,\pan)}) = 0$. The self-duality of Selmer complexes gives $ \widetilde{H}_{\rm f}^2(G_{F,\Sigma}, \T_1, \Delta_{(\emptyset,\pan)}) $ is torsion. Hence first part of the Theorem \ref{leading_term_rank1} tells us that $ \bbdelta_1(\T_{1},\Delta_{(\emptyset,\pan)}) \ne 0 $. Moreover, since 
			$$ \{0\} =  \widetilde{H}^1_{\rm f}(G_{F,\Sigma}, \T_1, \Delta_{(0,\pan)}) = \ker\left(  \widetilde{H}^1_{\rm f}(G_{F,\Sigma}, \T_1, \Delta_{(\emptyset,\pan)}) \xrightarrow{\res_p} \bigoplus_{\p \mid p} H^1(G_{F_\p}, \T_{1}) \right),  
			$$
			we conclude that $ \res_{\p_1}(\bbdelta_1(\T_{1},\Delta_{(\emptyset,\pan)})) \ne 0 $ as required.\\
			We note that \ref{comtriangle1} and the discussion above give rise to an exact sequence on $ \R $-modules
			\begin{equation}\label{exact1_sep28}
				0 \rightarrow 	\widetilde{H}^1_{\rm f}(G_{F,\Sigma}, \T_2, \tr \Delta_{(\g,\b)}) \rightarrow\underbrace{ \widetilde{H}^2_{\rm f}(G_{F,\Sigma}, \T_1, \Delta_{(0,\pan)})}_{\rm rank =1} \rightarrow \underbrace{\widetilde{H}^2_{\rm f}(G_{F,\Sigma}, \T_3, \Delta_{(\g,\b)})}_{\rm torsion}. 	
			\end{equation} 
			Hence \ref{exact1_sep28} shows that $ \widetilde{H}^1_{\rm f}(G_{F,\Sigma}, \T_2, \tr \Delta_{(\g,\b)}) $ is of rank one. therefore we conclude that $ \bbdelta_1(\T_{2}, \tr\Delta_{(\g,\b)}) \ne 0 $, which gives the second asserted non vanishing.
		\end{proof}
		
		\subsubsection{}
		Similarly we may also consider the natural map $  \res^{/\b}_{\p_1} $ as the composition of
		$$  \widetilde{H}^1_{\mathrm{f}}(G_{F,\Sigma}, \T_{3}, \Delta_+) \rightarrow  H^1(G_{F_{\p_1}}, \scrF^-_{\p_1}T_{\f} \otimes \scrF^+_{\p_1}T_{\g} \otimes \scrF^-_{\p_1}T_{\g}^*(\lambda)) \xrightarrow{\sim} H^1(G_{F_{\p_1}}, \scrF^-_{\p_1}T_{\f}(\lambda)) \widehat{\otimes}_O \R .$$
		\begin{rmk}
			If $ \bbdelta_0(\T_{3},\Delta_{\g}) \ne 0 $ then $ 	\operatorname{EXP}^*_{\scrF^-_{\p_1}} \circ  \res^{/\b}_{\p_1} (\bbdelta_1(\T_{3},\Delta_{+})) = 0 $ and expected that $\bbdelta_0(\T_{3},\Delta_{\mathrm{bal}}) = 0$ which is reflection of the fact balanced $ p $-adic $ L $-function is zero.
		\end{rmk}

		\subsection{Factorization}
		To establish the \cref{main_theorem_1}, we need to study the connecting morphism $ \delta^1 $ coming from \eqref{comtriangle1}.
		\subsubsection{The connecting morphisms}	
		We will explicitly describe the map 
		$$ \delta^1 \colon \widetilde{H}^1_{\mathrm{f}}(G_{F,\Sigma}, \T_2, \operatorname{tr}^*\Delta_{(\g,\b)}) \longrightarrow \widetilde{H}^2_{\mathrm{f}}(G_{F,\Sigma}, \T_1, \Delta_{(0,\pan)}) $$
		arising from \eqref{comtriangle1}. \\
		Let \begin{equation}
			\label{cocyclex_f_rank1}
			x_{\mathrm{f}} \coloneqq (x,(x_v), (\xi_v)) \in \widetilde{Z}^1_{\mathrm{f}}\left( G_{F, \Sigma}, \T_{2}, \tr \Delta_{(\g,\b)}  \right)  
		\end{equation} 
		be any cocycle: 
		\begin{align}
			\begin{aligned}
				x \in Z^1(G_{F,\Sigma}, \T_{2}), \ x_v \in U_v(\tr \Delta_{\g}, \T_{2}), \ \xi_v \in \T_{2},   \\
				dx_v = 0 , \quad \res_v(x)- i^+_v(x_v) = d\xi_v \quad \forall v \in \Sigma. 
			\end{aligned}
		\end{align}
		Since $ U^{\bullet}_{\p_1}(\tr \Delta_{(\g,\b)}, \T_{2}) \coloneqq C^{\bullet}(G_{F_{\p_1}}, F_{\p_1}^{+\g}\T_{2}) $ and $ F_{\p_1}^{+\g}\T_{2} $ is isomorphic image of $ F_{\p_1}^{+\g}\T_{3} $, we have  
		$$ U^{\bullet}_{\p_1}(\tr \Delta_{(\g,\b)}, \T_{2}) \cong U^{\bullet}_{\p_1}(\Delta_{(\g,\b)}, \T_{3}). $$
		Hence we can view $ x_{\p_1} =: y_{\p_1} $ as an element of $ U^1_{\p_1}(\Delta_{(\g,\b)}, \T_{3}) $.	But for $ \p_2 $, we have $ F_{\p_2}^{+\b}\T_{1} = \scrF^+_{\p_2} T_{\f} \hotimes_O \R_\g(\lambda) = \ker\{ F_{\p_2}^{+\b}\T_{3} \rightarrow \T_{2} \} $, which gives us 
		\begin{equation}\label{local_conditions_p2}
			F_{\p_2}^{+\b}\T_{1} \oplus F_{\p_2}^{+\b}\T_{2} \xrightarrow[\tr \oplus \iota]{\sim} F_{\p_2}^{+\b}\T_{3} .
		\end{equation} 
		Hence we get an embedding 
		\begin{align}
			U^{1}_{\p_2}(\tr \Delta_{(\g,\b)}, \T_{2}) \hookrightarrow U^{1}_{\p_2}(\Delta_{(\g,\b)}, \T_{3}).
		\end{align}
		For each $ v \in \Sigma\ba \bP $, we have a canonical identification 
		\begin{align}
			U^{1}_v(\tr \Delta_{(\g,\b)}, \T_{2}) \oplus U^{1}_v( \Delta_{(0,\pan)}, \T_{1}) \cong U^{1}_v(\Delta_{(\g,\b)}, \T_{3}) ,
		\end{align}
		which arises from the split short exact sequence \eqref{mainses} i.e.
		$ \T_{2} \oplus \T_{1} \xrightarrow{\sim} \T_{3}$. We also have the following identification 
		\begin{align}
			C^{\bullet}(G_{F, \Sigma}, \T_{2}) \oplus C^{\bullet}(G_{F, \Sigma}, \T_{1}) \cong C^{\bullet}(G_{F, \Sigma}, \T_{3}).
		\end{align}
		Put $ y = \iota(x) \; + \tr(0)  $, which is a cocycle in $ Z^1(G_{F,\Sigma}, \T_{3}) $ with $ dy =0 $. Similarly we can define $$ \mu_v = \iota(\xi_v) + \tr(0) \in \T_{3} \quad \forall v \in \Sigma .$$ Then the cochain $$ y_\mathrm{f} \coloneqq (y,(y_v),(\mu_v)) \in \widetilde{C}^1_{\mathrm{f}}\left( G_{F, \Sigma}, \T_{3}, \Delta_{\g}  \right) $$
		maps to $ x_\f $ under the morphism $ \pi_{\tr} $ since we have $ \pi_{\tr} \circ \iota = \operatorname{id}_{\T_{2}} $. Moreover, for all $ v \in \Sigma\backslash\mathbbl{P} $, we have 
		\begin{equation*}
			d\mu_v- \res_{v}(y)+ i_v^{+}(y_v) = \iota (d\xi_v - \res_{v}(x) + i_v^{+}(x_v))=0.
		\end{equation*}
		Also for $ \p_2 $ we have $ d\mu_{\p_2}- \res_{\p_2}(y)+ i_{\p_2}^{+}(y_{\p_2}) = \iota (d\xi_{\p_2} - \res_{\p_2}(x) + i_{\p_2}^{+}(x_{\p_2}))=0 $ thanks to \eqref{local_conditions_p2}.
		
		Now lets focus on $ \p_1 \in \mathbbl{P} $. Observe that the image of $ 	d\mu_{\p_1}- \res_{\p_1}(y)+ i_{\p_1}^{+}(y_{\p_1}) \in C^1(G_{F_{\p_1}}, \T_{3}) $ under the morphism 
		\begin{equation*}
			C^1(G_{F_{\p_1}}, \T_{3}) \xrightarrow{\pi_{\tr}}C^1(G_{F_{\p_1}}, \T_{2})
		\end{equation*}
		equals $ d\xi_{\p_1} - \res_{\p_1}(x) + i_{\p_1}^{+}(x_{\p_1}) =0  $, hence 
		\begin{equation*}
			\begin{split}
				d\mu_{\p_1}- \res_{\p_1}(y)+ i_{\p_1}^{+}(y_{\p_1}) & \in \ker\left( C^1(G_{F_{\p_1}}, \T_{3}) \xrightarrow{\pi_{\tr}} C^1(G_{F_{\p_1}}, \T_{2}) \right) \\
				& = \operatorname{im}\left( C^1(G_{F_{\p_1}}, \T_{1}) \xhookrightarrow{\tr} C^1(G_{F_{\p_1}}, \T_{3}) \right) .
			\end{split}
		\end{equation*}
		If we denote $ z_{\p_1} \in C^1(G_{F_{\p_1}}, \T_{1}) $ such that 
		$$ \tr(z_{\p_1}) = d\mu_{\p_1}- \res_{\p_1}(y)+ i_{\p_1}^{+}(y_{\p_1}) ,$$
		we get $$ dy_{\mathrm{f}} = \left( dy = 0, (dx_{v} = 0 )_{v \in \Sigma}, (\{0\}_{v \in \Sigma \backslash\{\p_1\}},\{\tr z_{\p_1}\} ) \right) \in \widetilde{C}^2_{\mathrm{f}}\left( G_{F, \Sigma}, \T_{3}, \Delta_{\g}  \right). $$
		As $ d\mu_{\p_1}- \res_{\p_1}(y)+ i_{\p_1}^{+}(y_{\p_1}) \in C^1(G_{F_{\p_1}}, \T_{3}) $ is a cocycle, we get $ d\tr (z_{\p_1}) =0  $ and so $ z_{\p_1} \in C^1(G_{F_{\p_1}}, \T_{1}) $ is also a cocycle. We conclude that 
		$$ \tr(0,0,z_{\p_1})_{\mathrm{f}}= (0, (0)_{v \in \Sigma}, (\{0\}_{v \in \Sigma \backslash\{\p_1\}},\{\tr z_{\p_1}\})) \in  \widetilde{Z}^2_{\mathrm{f}}\left( G_{F, \Sigma}, \T_{1}, \tr \Delta_{0}  \right), $$
		and $\tr(0,0,z_p)_{\mathrm{f}} = dy_{\mathrm{f}}$, where $ (0,0,z_p)_{\mathrm{f}} \coloneqq (0, (0)_{v \in \Sigma}, (\{0\}_{v \in \Sigma\backslash\{\p_1\}},\{z_{\p_1}\})) $. Therefore we deduce the following proposition:
		\begin{prop}
			For a cocycle given as in (\ref{cocyclex_f_rank1}), we have 
			$$ \delta^1([x_{\mathrm{f}}]) = [(0,0,z_{\p_1})_{\mathrm{f}}] \in  \widetilde{H}^2_{\mathrm{f}}(G_{F,\Sigma}, \T_1, \Delta_{0}). $$
		\end{prop}
		\begin{rmk}
			We note that the diagram
			\[\xymatrix{
				&\scrF_{\p_1}^{+\g}\T_3  \ar@{^{(}->}[r]^{i_{\p_1}^+} \ar[rd]_{\sim}^{\pi_{\tr}} &\T_3 \\
				& &\scrF_{\p_1}^{+\g}\T_2 \ar@{^{(}->}[u]_{\iota}
			}\]
			does not commute in contrast to \eqref{local_conditions_p2}. In consequence, $ z_{\p_1} $ comes about to account for this failure.
		\end{rmk}
		\subsubsection{}
		Now we will describe the cocycle $ z_{\p_1} $ in a more explicit manner.
		Let $\{ u^+_{\p_1} \} \subset \scrF^+_{\p_1} T_\g$ denote an $\R_\g$-basis of and let $\{ u^+_{\p_1} , u^-_{\p_1}  \} \subset  T_\g $ be an extended basis. Let us denote by $ \{ \breve{u}^+_{\p_1} , \breve{u}^-_{\p_1}  \} \subset  T_\g^*  $ the dual basis. Recall that 
		$$ \mathrm{Fil}^2_{\p_1}\ad(T_\g)(\lambda) \coloneqq \Hom(T_\g,\scrF^{+}_{\p_1} T_\g)(\lambda) = \scrF^{+}_{\p_1} T_\g \widehat{\otimes} T_\g^* (\lambda) =\left\langle u^+_{\p_1} \otimes \breve{u}^-_{\p_1}, u^+_{\p_1} \otimes \uu^+_{\p_1} \right\rangle $$
		\begin{equation*}
			\begin{split}
				\mathrm{Fil}^{+\g}_{\p_1} \ad^0(T_\g)(\lambda) \coloneqq & \ker \{\ad^0(T_\g) \hookrightarrow \ad(T_\g) \rightarrow \Hom(F_{\p_1}^+ T_\g,F_{\p_1}^- T_\g) \}(\lambda) \\
				= & \left\langle u^+_{\p_1} \otimes \breve{u}^-_{\p_1}, u^+_{\p_1} \otimes \uu^+_{\p_1}  - u^-_{\p_1} \otimes \uu^-_{\p_1} \right\rangle
			\end{split}
		\end{equation*}
		$$ \scrF^{+\g}_{\p_1} \T_{2} = T_\f \widehat{\otimes} \Fil^{+\g}_{\p_1} \ad^0(T_\g)(\lambda)  $$
		$$ \Fil^{\g}_{v}\ad^0(T_\g)(\lambda) \coloneqq \Hom(F_{\p_1}^- T_\g,F_{\p_1}^+  T_\g)(\lambda) = \left\langle u^+_{\p_1} \otimes \breve{u}^-_{\p_1} \right\rangle \subset 	\mathrm{Fil}^{+\g}_{\p_1} \ad^0(T_\g)(\lambda) $$
		all are $G_{F_{\p_1}}$-stable submodule such that the $G_{F_{\p_1}}$ action on the quotient  
		$$ \mathrm{gr}^{+\g}_{\p_1} \ad^0(T_\g) \coloneqq \mathrm{Fil}^{+\g}_{\p_1} \ad^0(T_\g)(\lambda) / \Fil^{\g}_{v}\ad^0(T_\g)(\lambda) $$
		is via $\lambda$ (c.f. the proof of Lemma \ref{localses}). \\
		Let us fix a basis $\{ w_1,w_2 \}$ of $T_\f$. Given $g \in G_{F_{\p_1}}$, let us define $a_i(g), b_i(g)  \in \R$ ($i = 1,2$) such that 
		$$ i_{\p_1}^+(x_{\p_1})(g) = \sum_{i=1}^{2} a_i(g)(u^+_{\p_1} \otimes \breve{u}^-_{\p_1})\otimes w_i + \sum_{i=1}^{2} b_i(g)( u^+_{\p_1} \otimes \uu^+_{\p_1}  - u^-_{\p_1} \otimes \uu^-_{\p_1})\otimes w_i . $$
		Hence the cocycle $i_{\p_1}^+(y_{\p_1}) \in C^1(G_{F_{\p_1},\T_{3}}) $ is then given explicitly by  
		$$ i_{\p_1}^+(y_{\p_1})(g) = \sum_{i=1}^{2} a_i(g)(u^+_{\p_1} \otimes \breve{u}^-_{\p_1})\otimes w_i + \sum_{i=1}^{2} b_i(g)( u^+_{\p_1} \otimes \uu^+_{\p_1} )\otimes w_i . $$
		By definition, the cocycle \( \operatorname{tr} \circ i_{\p_1}^+(y_{\p_1}) \in C^1(G_{F_{\p_1}}, \T_{1}) \) is the cocycle 
		$$ \operatorname{tr} \circ i_{\p_1}^+(y_{\p_1})(g) =  \sum_{i=1}^{2} b_i(g)w_i . $$
		Hence we get 
		\begin{prop}
			The cocycle $z_{\p_1} = \operatorname{tr} \circ i_{\p_1}^+(y_{\p_1}) $ agrees with the image of $i_{\p_1}^+(x_{\p_1}) $	under the natural projection map
			\begin{equation}
				C^1(G_{F_{\p_1}},\scrF^{+\g}_{\p_1} \T_{2}) \xrightarrow{\operatorname{pr}_{/\g,\p_1}} C^1(G_{F_{\p_1}}, T_\f \widehat{\otimes}\mathrm{gr}^{+\g}_{\p_1} \ad^0(T_\g)(\lambda)) \xrightarrow[\mrm{Lemma}\ \ref{localses}]{\sim} C^1(G_{F_{\p_1}}, T_\f(\lambda)) \otimes \R,
			\end{equation}
			In particular the morphism $\delta^1$ factors as
			\begin{equation}
				\begin{tikzcd}[row sep= 1ex, column sep=small]
					[x_\mathrm{f}]=[(x,(x_{v}), (\xi_{v}))]\arrow[rr,mapsto] \arrow[d, symbol = \in ] \arrow[ddddddd,mapsto, bend right =75]  &          &\delta^1([x_\mathrm{f}])=[(0,0,\operatorname{pr}_{/\g}\circ i_{\p}^+(x_{\p}))] \arrow[d, symbol = \in ]\\
					\widetilde{H}^1_{\mathrm{f}}(G_{F,\Sigma}, \T_2, \operatorname{tr}^*\Delta_{\g})\arrow[rr, "\delta^1"]\arrow[ddddd]                              &            &\widetilde{H}^2_{\mathrm{f}}(G_{F,\Sigma}, \T_1, \Delta_{0})\\
					& & \\
					& & \\
					& & \\
					& & \\
					H^1(G_{F_{\p_1}},\scrF^{+\g}_{\p_1} \T_{2})\arrow[r,"\operatorname{pr}_{/\g}" ']                                      &	H^1(G_{F_{\p_1}}, T_\f \widehat{\otimes}\mathrm{gr}^{+\g}_{\p_1} \ad^0(T_\g)(\lambda))\arrow[r, "\sim" ']  &	H^1(G_{F_{\p_1}}, \T_{1})\arrow[uuuuu, "\partial^1_{\f \otimes \lambda} " ']\\
					\left[ (i_{\p_1}^+(x_{\p_1}))_{\p_1} \right] =\left[ (\res_{\p_1}(x))_{\p_1} \right]  \arrow[rr,mapsto]  \arrow[u, symbol = \in ]                   &            &\left[ (\operatorname{pr}_{/\g,\p_1}(i_{\p_1}^+(x_{\p_1})))_{\p_1} \right] \arrow[u, symbol = \in ]\arrow[uuuuuuu,mapsto, bend right= 67]  
				\end{tikzcd}
			\end{equation}
			where the morphism $\partial_{\f \otimes \lambda}^1$ is induced from the exact sequence
			$$ \widetilde{R\Gamma}_{\mathrm{f}}(G_{F,\Sigma},\T_{1}, \Delta_{(0,\pan)}) \rightarrow \widetilde{R\Gamma}_{\mathrm{f}}(G_{F,\Sigma},\T_{1}, \Delta_{(\emptyset,\pan)}) \xrightarrow{\res_{\p_1}}   R\Gamma(G_{F_{\p_1}}, \T_{1}) \xrightarrow{\partial_{\f \otimes \lambda}}[+1] $$
			In particular 
			$$ H^1(G_{F_{\p_1}}, \T_{1}) \ni [a_{\p_1}] \xmapsto{\partial_{\f \otimes \lambda}^1} [(0,0,(\{0\}_{v \in \Sigma\backslash\{\p_1\}},\{a_{\p_1}\}))] \in \widetilde{H}_{\mathrm{f}}^2(G_{F,\Sigma},\T_{1}, \Delta_{(0,\pan)}). $$
		\end{prop}
		\begin{cor}
			\label{cor : delta_and_delta_rank1}
			Suppose $\bbdelta_0(\T_{3},\Delta_{(\g,\b)}) \ne 0 $. then the map $ \delta^1 $ factors as 
			\begin{equation}
				\begin{tikzcd}[column sep = small]
					\widetilde{H}^1_{\mathrm{f}}(G_{F,\Sigma}, \T_2, \operatorname{tr}^*\Delta_{(\g,\b)}) \arrow[rr, hookrightarrow, "\delta^1"] \arrow[dr, hookrightarrow,"\res_{/\mathrm{pan},\p_1}" ']&&\widetilde{H}^2_{\mathrm{f}}(G_{F,\Sigma}, \T_1, \Delta_{(0,\pan)})\\
					&\dfrac{H^1(G_{F_{\p_1}}, \T_{1})}{\res_{\p_1}(\widetilde{H}^1_{\mathrm{f}}(G_{F,\Sigma},\T_{1}, \Delta_{(\emptyset, \pan)}))} 
					\arrow[ur, hookrightarrow, "\partial_{\f \otimes \lambda}^1" '] &
				\end{tikzcd}
			\end{equation}
			and all the $ \R $-modules that appear are of rank one.
		\end{cor}
		
		\begin{proof}
			$ \delta^1 $ is Injective because $ \widetilde{H}^1_{\mathrm{f}}(G_{F,\Sigma}, \T_3, \Delta_{(\g,\b)}) =0 $ since $\bbdelta_0(\T_{3},\Delta_{(\g,\b)}) \ne 0 $ via Theorem \ref{r0}. \\
			Previous Proposition tells that $ \delta^1 $ factors as 
			\begin{equation}
				\begin{tikzcd}
					\widetilde{H}^1_{\mathrm{f}}(G_{F,\Sigma}, \T_2, \operatorname{tr}^*\Delta_{(\g,\b)}) \arrow[r, hook, "\delta^1"]
					\arrow[d, hook, "\operatorname{pr}_{/\g} \circ \res_{\p_1} " '] & \widetilde{H}^2_{\mathrm{f}}(G_{F,\Sigma}, \T_1, \Delta_{(0,\pan)}) \\
					H^1(G_{F_{\p_1}}, \T_{1}) \arrow[ur, "\partial_{\f \otimes \lambda}^1 " '] & 
				\end{tikzcd}
			\end{equation}
			Also $ \partial_{\f \otimes \lambda}^1$ factors as 
			\begin{equation}
				\begin{tikzcd}
					H^1(G_{F_{\p}}, \T_{1}) \arrow[r, "\partial_{\f \otimes \lambda}^1 "] \arrow[dr, twoheadrightarrow] & \widetilde{H}^2_{\mathrm{f}}(G_{F,\Sigma}, \T_1, \Delta_{(0,\pan)}) \\
					&  \dfrac{H^1(G_{F_{\p_1}}, \T_{1})}{\res_{\p_1}(\widetilde{H}^1_{\mathrm{f}}(G_{F,\Sigma},\T_{1}, \Delta_{(\emptyset,\pan)}))} \arrow[u, hook, " \partial_{\f \otimes \lambda}^1 " ']
				\end{tikzcd}
			\end{equation}
			Hence we have a injective homomorphism $ \res_{/\pan, \p_1} $ which is the composition of 
			$$ \widetilde{H}^1_{\mathrm{f}}(G_{F,\Sigma}, \T_2, \operatorname{tr}^*\Delta_{\g}) \xrightarrow{\operatorname{pr}_{/\g} \circ \res_{\p_1}} H^1(G_{F_{\p_1}}, \T_{1})\twoheadrightarrow \frac{H^1(G_{F_{\p_1}}, \T_{1})}{\res_{\p_1}(\widetilde{H}^1_{\mathrm{f}}(G_{F,\Sigma},\T_{1}, \Delta_{(\emptyset,\pan)}))} .$$
			Since $ \widetilde{H}^1_{\mathrm{f}}(G_{F,\Sigma},\T_{1}, \Delta_{(\emptyset,\pan)}) = \widetilde{H}^1_{\mathrm{f}}(G_{F,\Sigma},\T_{1}, \Delta_{\pan}) $ we have the  required result.
		\end{proof}
		Recall our map $ \res_{/\b, \p_1} $ (see \ref{resbal_1}) given as composite of 
		$$ \widetilde{H}^1_{\mathrm{f}}(G_{F,\Sigma}, \T_2, \operatorname{tr}^*\Delta_{(\g,\b)}) \xrightarrow{\operatorname{pr}_{/\g} \circ \res_{\p_1}} H^1(G_{F_{\p_1}}, \T_{1}) \rightarrow  H^1(G_{F_{\p_1}},\scrF^-_{\p_1} T_\f( \lambda)) \hotimes_O \R .$$
		Hence $  \res_{/\b, \p_1} $ factors as 
		\begin{align*}
			\res_{/\b, \p_1} \colon \widetilde{H}^1_{\mathrm{f}}(G_{F,\Sigma}, \T_2, \operatorname{tr}^*\Delta_{(\g,\b)}) \xrightarrow{\res_{/\pan,\p_1}} \frac{H^1(G_{F_{\p_1}}, \T_{1})}{\res_{\p_1}(\widetilde{H}^1_{\mathrm{f}}(G_{F,\Sigma},\T_{1}, \Delta_{(\emptyset,\pan)}))} \qquad \qquad \qquad \\
			\rightarrow  H^1(G_{F_{\p_1}},\scrF^-_{\p_1} T_\f \hotimes_O \R_{\g}(\lambda)) .
		\end{align*}
		Since $ \widetilde{H}^1_{\mathrm{f}}(G_{F,\Sigma},\T_{1}, \Delta_{(\emptyset,\pan)}) = \widetilde{H}^1_{\mathrm{f}}(G_{F,\Sigma},\T_{1}, \Delta_{\pan}) $ in our present set-up, we have the following exact sequence for $ \p_1 $
		$$ 0 \rightarrow \frac{H^1(G_{F_\p},\scrF^+_{\p_1} \T_{1})}{\res_{\p_1}(\widetilde{H}^1_{\mathrm{f}}(G_{F,\Sigma},\T_{1}, \Delta_{(\emptyset,\pan)})) } \rightarrow \frac{H^1(G_{F_{\p_1}},\T_{1})}{\res_{\p_1}(\widetilde{H}^1_{\mathrm{f}}(G_{F,\Sigma},\T_{1}, \Delta_{(\emptyset,\pan)})) } \rightarrow H^1(G_{F_{\p_1}},\scrF^-_{\p_1} \T_{1}) \rightarrow 0 .$$
		As the map $ \res_{/\pan} $ is injective and $ \widetilde{H}^1_{\mathrm{f}}(G_{F,\Sigma}, \T_2, \operatorname{tr}^*\Delta_{\g}) $ is torsion-free, 
		$$ \res_{/\pan, \p_1}(\bbdelta_1(\T_2, \operatorname{tr}^*\Delta_{\g})) \subset \frac{H^1(G_{F_{\p_1}},\scrF^+_{\p_1} \T_{1})}{\res_{\p_1}(\widetilde{H}^1_{\mathrm{f}}(G_{F,\Sigma},\T_{1}, \Delta_{(\emptyset,\pan)})) } $$ 
		is also torsion-free.
		Since $$ \frac{H^1(G_{F_{\p_1}},\scrF^+_{\p_1} \T_{1})}{\res_{\p_1}(\widetilde{H}^1_{\mathrm{f}}(G_{F,\Sigma},\T_{1}, \Delta_{(\emptyset,\pan)})) } $$ is torsion
		$\res_{/\pan, \p_1}(\bbdelta_1(\T_2, \operatorname{tr}^*\Delta_{(\g,\b)}))$ maps isomorphically onto its image 
		$$ \res_{/\b, \p_1}(\bbdelta_1(\T_2, \operatorname{tr}^*\Delta_{(\g,\b)})) \subset H^1(G_{F_{\p_1}},\scrF^-_{\p_1} \T_{1}). $$
		Hence we have an exact sequence
		\begin{align}
			\label{exact_seq_rank1}
			\begin{aligned}
				0 \rightarrow \frac{H^1(G_{F_{\p_1}},\scrF^+_{\p_1} \T_{1})}{\res_{\p_1}(\widetilde{H}^1_{\mathrm{f}}(G_{F,\Sigma},\T_{1}, \Delta_{(\emptyset,\pan)})) } \rightarrow \dfrac{\frac{H^1(G_{F_{\p_1}},\T_{1})}{\res_{\p_1}(\widetilde{H}^1_{\mathrm{f}}(G_{F,\Sigma},\T_{1}, \Delta_{(\emptyset,\pan)})) }}{\res_{/\pan, \p_1}(\bbdelta_1(\T_2, \operatorname{tr}^*\Delta_{(\g,\b)}))} \qquad \qquad \qquad \\
				\rightarrow \frac{H^1(G_{F_{\p_1}},\scrF^-_{\p_1} \T_{1})}{ \res_{/\b, \p_1}(\bbdelta_1(\T_2, \operatorname{tr}^*\Delta_{(\g,\b)}))} \rightarrow 0.
			\end{aligned}
		\end{align} 
		\begin{prop}\label{prop : final_rank1}
			In the setting of Corollary \ref{cor : delta_and_delta_rank1}, we have 
			\begin{align*}
				&\Char\left( \dfrac{\widetilde{H}^2_{\mathrm{f}}(G_{F,\Sigma}, \T_1, \Delta_{(0,\pan)})}{\delta^1(\bbdelta_1(\T_{2}, \tr\Delta_{(\g,\b)}))} \right) \\
				&= \Char\left(  \frac{H^1(G_{F_{\p_1}},\scrF^+_{\p_1} \T_{1})}{\res_{\p_1}\bbdelta_1(\T_{1},\Delta_{(\emptyset,\pan)}) }\right) \Char\left( \frac{H^1(G_{F_{\p_1}},\scrF^-_{\p_1} \T_{1})}{ \res_{/\b, \p_1}(\bbdelta_1(\T_2, \operatorname{tr}^*\Delta_{(\g,\b)}))} \right).
			\end{align*} 
			
		\end{prop}
		\begin{proof}
			We have an exact sequence
			\begin{align}
				\begin{aligned}
					0 \rightarrow \widetilde{H}^1_{\mathrm{f}}(G_{F,\Sigma}, \T_1, \Delta_{(\emptyset,\pan)}) \xrightarrow{\res_{\p_1}} H^1(G_{F_{\p_1}},\T_{1}) \xrightarrow{\partial_{\f \otimes \lambda}^1} \widetilde{H}^2_{\mathrm{f}}(G_{F,\Sigma}, \T_1, \Delta_{(0,\pan)}) \qquad \qquad \\ \rightarrow \widetilde{H}^2_{\mathrm{f}}(G_{F,\Sigma}, \T_1, \Delta_{(\emptyset,\pan)}) \rightarrow 0.
				\end{aligned}
			\end{align}
			Hence we get via Corollary \ref{cor : delta_and_delta_rank1}
			\begin{align}
				0 \rightarrow \dfrac{\frac{H^1(G_{F_{\p_1}},\T_{1})}{\res_{\p_1}(\widetilde{H}^1_{\mathrm{f}}(G_{F,\Sigma},\T_{1}, \Delta_{{\emptyset, \pan}})) }}{\res_{/\pan, \p_1}(\bbdelta_1(\T_2, \operatorname{tr}^*\Delta_{(\g,\b)}))} \xrightarrow{\partial_{\f \otimes \lambda}^1} \dfrac{\widetilde{H}^2_{\mathrm{f}}(G_{F,\Sigma}, \T_1, \Delta_{(0,\pan)})}{\delta^1(\bbdelta_1(\T_{2}, \tr\Delta_{(\g,\b)}))} \rightarrow \widetilde{H}^2_{\mathrm{f}}(G_{F,\Sigma}, \T_1, \Delta_{(\emptyset,\pan)}) \rightarrow 0.
			\end{align}
			Thanks to equation \ref{exact_seq_rank1} and Theorem \ref{leading_term_rank1}, we have 
			\begin{align}
				\begin{aligned}
					&\Char\left( \dfrac{\widetilde{H}^2_{\mathrm{f}}(G_{F,\Sigma}, \T_1, \Delta_{0})}{\delta^1(\bbdelta_1(\T_{2}, \tr\Delta_{\g}))} \right)\qquad \qquad \qquad \qquad \qquad \qquad \qquad \qquad \\
					&= \Char \left(\dfrac{\frac{H^1(G_{F_{\p_1}},\T_{1})}{\res_{\p_1}(\widetilde{H}^1_{\mathrm{f}}(G_{F,\Sigma},\T_{1}, \Delta_{{\emptyset, \pan}})) }}{\res_{/\pan, \p_1}(\bbdelta_1(\T_2, \operatorname{tr}^*\Delta_{(\g,\b)}))} \right) \Char(\widetilde{H}^2_{\mathrm{f}}(G_{F,\Sigma}, \T_1, \Delta_{(\emptyset,\pan)})) \\
					&= \Char\left(  \frac{H^1(G_{F_{\p_1}},\scrF^+_{\p_1} \T_{1})}{\res_{\p_1}(\widetilde{H}^1_{\mathrm{f}}(G_{F,\Sigma},\T_{1}, \Delta_{(\emptyset,\pan)})) }\right) \Char\left( \frac{H^1(G_{F_{\p_1}},\scrF^-_{\p_1} \T_{1})}{ \res_{/\b, \p_1}(\bbdelta_1(\T_2, \operatorname{tr}^*\Delta_{(\g,\b)}))} \right) \\
					&\qquad \qquad \qquad \qquad \qquad \qquad \qquad \qquad \qquad \qquad \times \Char\left( \dfrac{\widetilde{H}^1_{\mathrm{f}}(G_{F,\Sigma}, \T_1, \Delta_{(\emptyset,\pan)})}{\R\bbdelta_1(\T_{1},\Delta_{(\emptyset,\pan)})} \right) \\
					&=\Char\left(  \frac{H^1(G_{F_{\p_1}},\scrF^+_{\p_1} \T_{1})}{\res_{\p_1}\bbdelta_1(\T_{1},\Delta_{(\emptyset,\pan)}) }\right) \Char\left( \frac{H^1(G_{F_{\p_1}},\scrF^-_{\p_1} \T_{1})}{ \res_{/\b, \p_1}(\bbdelta_1(\T_2, \operatorname{tr}^*\Delta_{(\g,\b)}))} \right)
				\end{aligned}
			\end{align}	
			as required.
		\end{proof}
		\begin{thm}
			\label{main_theorem_1}
			Suppose $\bbdelta_0(\T_{3}, \Delta_{(\g, \b)}) \ne 0$. Then:
			\begin{align*}
				\bbdelta_0(\T_{3}, \Delta_{(\g, \b)})  =  \Exp^*_{\scrF^-_{\p_1} T_{\f}(\lambda)}\left( \bbdelta_1(\T_2, \operatorname{tr}^*\Delta_{(\g,\b)})\right)\cdot i^*_\f\left(\Log_{\f \otimes \lambda}^1(\mathbf{BK}^{\rm LZ}_{\f\otimes \lambda})\right).
			\end{align*}	
		\end{thm}
		\begin{proof}
			From the equation \ref{comtriangle1}, we get an exact sequence
			\begin{align}
				\begin{aligned}
					0 \rightarrow \widetilde{H}^1_{\mathrm{f}}(G_{F,\Sigma}, \T_2, \operatorname{tr}^*\Delta_{(\g,\b)}) \xrightarrow{\delta^1} \widetilde{H}^2_{\mathrm{f}}(G_{F,\Sigma}, \T_1, \Delta_{0,\pan}) \rightarrow \widetilde{H}^2_{\mathrm{f}}(G_{F,\Sigma}, \T_3, \Delta_{(\g,\b)}) \\
					\noindent \rightarrow  \widetilde{H}^2_{\mathrm{f}}(G_{F,\Sigma}, \T_2, \operatorname{tr}^*\Delta_{(\g,\b)}) \rightarrow 0
				\end{aligned}
			\end{align}
			Using Theorem \ref{old_log_relation}  and Proposition \ref{prop : final_rank1} shows that 
			\begin{align}
				\begin{aligned}
					&\Char \left( \widetilde{H}^2_{\mathrm{f}}(G_{F,\Sigma}, \T_3, \Delta_{(\g,\b)})  \right) \\
					&=  \Char\left(  \frac{H^1(G_{F_{\p_1}},\scrF^+_{\p_1} \T_{1})}{\res_{\p_1}\bbdelta_1(\T_{1},\Delta_{(\emptyset,\pan)}) }\right) \Char\left( \frac{H^1(G_{F_{\p_1}},\scrF^-_{\p_1} \T_{1})}{ \res_{/\b, \p_1}(\bbdelta_1(\T_2, \operatorname{tr}^*\Delta_{(\g,\b)}))} \right)
				\end{aligned}
			\end{align}
			Hence we have via \eqref{6.10}
			\begin{align}
				\bbdelta_0(\T_{3}, \Delta_{(\g, \b)})  =  \Exp^*_{\scrF^-_{\p_1} T_{\f}(\lambda)}\left( \bbdelta_1(\T_2, \operatorname{tr}^*\Delta_{(\g,\b)})\right)\cdot i^*_\f\left(\Log_{\f \otimes \lambda}^1(\mathbf{BK}^{\rm LZ}_{\f\otimes \lambda})\right).
			\end{align}	
			We also get $ \Exp^*_{\scrF^-_{\p_1} T_{\f}(\lambda)}\left( \bbdelta_1(\T_2, \operatorname{tr}^*\Delta_{(\g,\b)})\right) \ne0 \ne i^*_\f\left(\Log_{\f \otimes \lambda}^1(\mathbf{BK}^{\rm LZ}_{\f\otimes \lambda})\right) $.
		\end{proof}
		\begin{rmk}
			As previously noted, we expect that (in view of Iwasawa–Greenberg main conjectures) $ \bbdelta_0(\T_{3}, \Delta_{(\g, \b)}) $ is generated by $ 	\scr{L}_p^{(\g,\b)}(\f \otimes \As(\g) )(P,Q)  $ and $ \Exp^*_{\scrF^-_{\p_1} T_{\f}(\lambda)}\left( \bbdelta_1(\T_2, \operatorname{tr}^*\Delta_{(\g,\b)})\right)  $ is generated by $ \scr{L}_p^{(\g,\b)}(\f \otimes \breve{\g} )(P,Q) $. Whereas $ \Log_{\f \otimes \lambda}^1(\mathbf{BK}^{\rm LZ}_{\f\otimes \lambda}) $ is expected to interpolate the derivatives of Perrin-Riou's $ \mbf D_{\rm cris}(V_{f \otimes \lambda}) \otimes \cal{H}(\Gamma_{\mrm{cycl}}) $-valued $ p $-adic $ L $-function of $ f \otimes \lambda $, a crystalline specialization of $ \f \otimes\lambda $ (cf. \cite{PR93,perrin-riou_padic}). The image $ \Log_{\f \otimes \lambda}^1(\mathbf{BK}^{\rm LZ}_{\f\otimes \lambda}) $ coincide with the $ p $-adic $ L $-function in \cite[Theorem 0.1]{Dimitrov13} by the main work of \cite{loeffler2020iwasawatheoryquadratichilbert}. 
			
			Moreover, we also remark that \enquote*{super-factorization} (in the sense of \cite{BC25}) of our main results in \cref{section : factorization_1} is a project in progress with considering CM forms. This would give analogous results of Theorem 3.4 of \cite{BC25} in our setup. But the main challenges arise about construction of triple product $ p $-adic $ L $-functions for quadratic Hilbert modular forms in full generalities and getting the \enquote*{BDP}-principles. The authors are currently pursuing in this direction.
			
			Please note that this situation is a real quadratic version of \emph{rank (1,1)} case stated in \cite[\textsection4.5.2]{Darmon16GenralisedKatoClasses}.
			
		\end{rmk}
		\section{Factorization of algebraic $p$-adic $L$-functions (The rank (0,2) case)}
		\label{section : factorization_2}
		\subsection{Modules of leading terms in rank (0,2) case}
		As in \cref{section : factorization_1}, we will list down all the hypotheses at the beginning.
		\subsubsection{Hypotheses}
		Throughout \cref{section : factorization_2} we will work under the following assumption. Both $ \f $ and $ \g $ satisfy \ref{DIST}. While we assume \ref{Irr} for $ \f $, the family $ \g $ will satisfy the stronger version \ref{Irr+}. Along with those, the condition \ref{TAM} holds for $ \f $ and $ \g $ to ensure \ref{H=0} and \ref{Tam} for $ \T_{1} $, $ \T_{2} $ and $ \T_{3} $. In this section, we work in the scenario of \cref{rank2_scenario} i.e. 
		\begin{itemize}
			\item[\mylabel{sign2}{\textbf{(S2)}}] $ \varepsilon(\f_P \otimes \lambda) = -1 $ and  $ \varepsilon(\f_P \otimes \ad^0(\g_Q)(\lambda) )= 
			1$ whenever $(P,Q) \in \W_{\mrm{cl}}^{\f}	$ for the global root number.
		\end{itemize}
		For $ \f \otimes \lambda $ we assume the following plectic conjecture to get a rank two Euler system:
		\begin{itemize}
			\item[\mylabel{plectic}{\textbf{(Plec)}}] There exist two distinct nontrivial plectic Siegel classes $ \frak{g}_{f\otimes \lambda,1}^{\rm plec } $ and $ \frak{g}_{f\otimes \lambda,2}^{\rm plec } $ in the the first plectic-etale cohomology of Hilbert-Blumenthal abelian surface associated with $ f \otimes \lambda $ for any crystalline specialization $ f $ of $ \f $. (cf. \cite[\textsection 15]{nekovar16plectic})
		\end{itemize}
		We also assume the hypotheses \ref{ord}, \ref{MC}, \ref{NA}, and \ref{BI} for the reason stated in Remark \ref{remark_hypo}.
		\subsubsection{}
		\label{7.1.2}
		We have \begin{align}
			\label{7.1}
			\chi(\widetilde{R\Gamma}_{\mathrm{f}}(G_{F, \Sigma}, T_\f(\lambda), \Delta_\emptyset)) = -2.
		\end{align} 
		Hence there will be a cyclic submodule 
		$$\bbdelta_2(T_\f, \Delta_\emptyset) \subset \bigcap^2_{\R_f}\widetilde{H}^1_{\mathrm{f}}(G_{F,\Sigma},T_\f(\lambda),\Delta_\emptyset)$$
		\begin{rmk}
			\label{gen_kato_class}
			$\bbdelta_2(T_\f, \Delta_\emptyset)$ expected to be related with rank two Kolyvagin system $\mathbf{BK}^{\rm plec}_{\f \otimes \lambda, 1} \wedge \mathbf{BK}^{\rm plec}_{\f \otimes \lambda, 2} $ (cf. \cite[Section 5]{BSS18}) related to nontrivial rank two Euler system. The existence of this rank two Euler system is conjectural (cf. \cite[\textsection 2]{LZ20local}), but the hypothesis \ref{plectic} assures the existence and non-vanishing. Hence $\mathbf{BK}^{\rm plec}_{\f \otimes \lambda, 1} \wedge \mathbf{BK}^{\rm plec}_{\f \otimes \lambda, 2} $ is non trivial.
			
			We also like to remark that, though the existence of $ \mathbf{BK}^{\rm Plec}_{\f\otimes \lambda,1} \wedge \mathbf{BK}^{\rm Plec}_{\f\otimes \lambda,2}$ is conjectural, but there are some recent works towards the evidence of the higher rank Euler system. In \cite{fornea23plectic}, the authors introduced \emph{plectic} Stark-Heegner points for an elliptic curve $ A/F $ given that $ \lim_{s\to 1 }L(A/F,s) = 2 $ which is related to the local image of $ \mathbf{BK}^{\rm Plec}_{\f\otimes \lambda,1} \wedge \mathbf{BK}^{\rm Plec}_{\f\otimes \lambda,2}$ in a parallel weight 2 specialization. Also the \enquote{mock plectic points} in \cite{Darmon_Fornea_2025}, may be seen an oblique evidence for the plectic philosophy of Nekovář and Scholl, using the mock analogue of Hilbert modular forms. In the Section 3.8 of \cite{Darmon_Fornea_2025} authors provide some conjectural analogy between mock plectic points and the plectic p-adic
			invariants. In \cite{hernandez2022plectic}, the authors defined two variable anti-cyclotomic $ p $-adic $ L $-functions via Stark-Heegner points.
		\end{rmk}
		\begin{prop} 
			\label{7.2}
			We have
			$$\widetilde{R\Gamma}_{\mathrm{f}}(G_{F,\Sigma},T_\f(\lambda), \Delta_0), \; \widetilde{R\Gamma}_{\mathrm{f}}(G_{F,\Sigma},T_\f(\lambda), \Delta_\emptyset) \in D_{\mathrm{parf}}^{[1,2]}(\prescript{}{\R_\f}{\mathrm{Mod}})$$
			If $\bbdelta_2(T_\f(\lambda), \Delta_\emptyset) \neq 0$, then:
			\begin{enumerate}
				\item If the image of $\bbdelta_1(T_\f(\lambda), \Delta_\emptyset)$ under the composite map 
				$$ \res^-_p\coloneqq \bigoplus_{\p \mid p} \res_{\p}^-  \colon H^1(G_{F, \Sigma}, T_\f(\lambda)) \rightarrow \bigoplus_{\p \mid p} H^1(F_{\p}, T_\f(\lambda)) \rightarrow \bigoplus_{\p \mid p} H^1(F_{\p}, \scrF^-_{\p} T_\f(\lambda)) $$ is zero then $$ \widetilde{H}^1_{\mathrm{f}}(G_{F,\Sigma},T_\f(\lambda),\Delta_{\pan}) \cong \widetilde{H}^1_{\mathrm{f}}(G_{F,\Sigma},T_\f(\lambda),\Delta_\emptyset) $$ are both free $ \R_{\f} $-modules of rank two. 
				\item $\widetilde{H}^2_{\mathrm{f}}(G_{F,\Sigma},T_\f(\lambda),\Delta_\emptyset)$ is torsion, Moreover,
				\begin{equation*}
					\det(\widetilde{H}^1_{\mathrm{f}}(G_{F,\Sigma},T_\f(\lambda),\Delta_\emptyset)/\R_{\f}\bbdelta_1(T_\f(\lambda), \Delta_\emptyset)) = \left( \det(\widetilde{H}^2_{\mathrm{f}}(G_{F,\Sigma},T_\f(\lambda),\Delta_\emptyset))\right)^2 .
				\end{equation*}
				\item If in addition for each $ \p = \p_1,\p_2 $, $ \res_{\p}(\bbdelta_1(T_\f(\lambda), \Delta_{\emptyset})) \ne 0 $, then $ \widetilde{H}^1_{\mathrm{f}}(G_{F,\Sigma},T_\f(\lambda),\Delta_{0}) = 0 $ and the $ \R_{\f} $-module $  \widetilde{H}^2_{\mathrm{f}}(G_{F,\Sigma},T_\f(\lambda),\Delta_{0}) $ has rank two.
			\end{enumerate}
		\end{prop}
		\begin{proof}
			The fact that the Selmer complexes are perfect and that they are concentrated in the indicated degrees follows from Proposition \ref{dercat}. \\
			Since $\bbdelta_2(T_\f(\lambda), \Delta_\emptyset) \neq 0$, the $ \R_{\f} $-module $ \widetilde{H}^2_{\mathrm{f}}(G_{F,\Sigma},T_\f(\lambda),\Delta_\emptyset) $ is torsion by Theorem \ref{H2tors}. Therefor $ \widetilde{R\Gamma}_{\mathrm{f}}(G_{F,\Sigma},T_\f(\lambda), \Delta_\emptyset) \in D_{\mathrm{parf}}^{[1,2]}(\prescript{}{\R}{\mathrm{Mod}}) $ and $ \chi(\widetilde{R\Gamma}_{\mathrm{f}}(G_{F,\Sigma},T_\f(\lambda), \Delta_\emptyset))=-2 $ implie that rank of $ \widetilde{H}^1_{\mathrm{f}}(G_{F,\Sigma},T_\f(\lambda),\Delta_\emptyset) $ equals two. Moreover, since $ \R_{\f} $ is a finite dimensional regular local ring, for any finitely generated $ \R_{\f} $-module $ M $ we have $$ \Ext^1_{\R_{\f}}(\widetilde{H}^1_{\mathrm{f}}(G_{F,\Sigma},T_\f(\lambda),\Delta_\emptyset), M) = 0 = \Ext^3_{\R_{\f}}(\widetilde{H}^2_{\mathrm{f}}(G_{F,\Sigma},T_\f(\lambda),\Delta_\emptyset), M). $$ 
			Thence, $ \widetilde{H}^1_{\mathrm{f}}(G_{F,\Sigma},T_\f(\lambda),\Delta_\emptyset) $ is free. \\
			Since $ \varepsilon(\f \otimes \lambda) =-1 $, 
			\begin{equation}\label{res^-_p}
				\res^-_p(\bbdelta_1(T_\f(\lambda), \Delta_\emptyset)) = 0
			\end{equation} follows from the comparison of the class $ \bbdelta_2(T_\f(\lambda), \Delta) $ with the Kolyvagin system stated in Remark \ref{gen_kato_class}. Hence by (rank 2) Kato's reciprocity law for Hilbert modular form
			we have the equality $$ \widetilde{H}^1_{\mathrm{f}}(G_{F,\Sigma},T_\f(\lambda),\Delta_{\pan}) \cong \widetilde{H}^1_{\mathrm{f}}(G_{F,\Sigma},T_\f(\lambda),\Delta_\emptyset). $$
			which is ensured by the non triviality of $ \mathbf{BK}^{\rm plec}_{\f \otimes \lambda, 1} \wedge \mathbf{BK}^{\rm plec}_{\f \otimes \lambda, 2} $ and the fact $ \widetilde{H}^1_{\mathrm{f}}(G_{F,\Sigma},T_\f(\lambda),\Delta_\emptyset) $ is free. This proves the results in (1).		\\
			Statement in (2) is direct consequences of Corollary \ref{r2}. \\
			In addition, if $ \res_{p}(\bbdelta_1(T_\f(\lambda), \Delta_{\emptyset})) \ne 0 $, then $ \widetilde{H}^1_{\mathrm{f}}(G_{F,\Sigma},T_\f(\lambda),\Delta_{0}) = 0 $ is a consequence of conclusion (1). As $ \chi(\widetilde{R\Gamma}_{\mathrm{f}}(G_{F,\Sigma},T_\f(\lambda), \Delta_{0})) = -2  $, we get $ \R_{\f} $-rank of $ \widetilde{H}^2_{\mathrm{f}}(G_{F,\Sigma},T_\f(\lambda),\Delta_{0}) $ is two.
		\end{proof}
		We also choose trivializations for $\p = \p_1, \p_2$ (cf. discussion in \textsection \ref{6.1.2})
		\begin{equation}
			\label{7.3}
			\LOG_{\scrF^+_{\p} T_\f(\lambda)} \colon H^1(G_{F_{\p}}, \scrF^+_{\p} T_\f(\lambda)) \rightarrow \R_\f.
		\end{equation}
		and let us put $ \Log_{\scrF^+_{\p} T_\f(\lambda)} \coloneqq \LOG_{\scrF^+_{\p} T_\f(\lambda)} \otimes_{i_\f^*} {\rm id} 
		\colon H^1(G_{F_{\p}}, \scrF^+_{\p} \T_1) \rightarrow \R $.\\
		By results in Proposition \ref{7.2}.1, we can choose an isomorphism (compatible with \eqref{7.3})
		\begin{align}
			\Log^2_{\f \otimes \lambda} \colon \bigcap^2_{\R_f}\widetilde{H}^1_{\mathrm{f}}(G_{F,\Sigma},T_\f(\lambda),\Delta_\emptyset) \rightarrow \R_{\f},
		\end{align} 
		such that $ \Log^2_{\f \otimes \lambda}(\mathbf{BK}^{\rm plec}_{\f \otimes \lambda, 1} \wedge \mathbf{BK}^{\rm plec}_{\f \otimes \lambda, 2} ) $ is a generator of $ \bbdelta_2(T_\f, \Delta_\emptyset) $.
		
		\subsubsection{}
		Let us focus on $\T_{2} $, we have $$ \chi(\widetilde{R\Gamma}_{\mathrm{f}}(G_{F,\Sigma}, \T_2, \operatorname{tr}^*\Delta_{\g})) = -2 $$
		Hence we have a cyclic submodule
		$$ \bbdelta_2(\T_2, \operatorname{tr}^*\Delta_{\g}) \subset \bigcap^2_R \widetilde{H}^1_{\mathrm{f}}(G_{F,\Sigma}, \T_2, \operatorname{tr}^*\Delta_{\g}) $$
		On the other hand, we have
		$$ \chi(\widetilde{R\Gamma}_{\mathrm{f}}(G_{F,\Sigma},\T_2, \operatorname{tr}^*\Delta_{\mathrm{bal}})) = 0 $$
		Hence we have a submodule 
		$$ \bbdelta_0(\T_2, \operatorname{tr}^*\Delta_{\mathrm{bal}}) \subset \R$$
		The exact sequence \eqref{ses_sep18} in Lemma \ref{localses} gives a map
		\begin{equation}
			\res_{/\mathrm{bal}} \coloneqq \bigoplus_{\p \mid p}\res_{/\mathrm{bal}, \p} \colon \widetilde{H}^1_{\mathrm{f}}(G_{F,\Sigma}, \T_2, \operatorname{tr}^*\Delta_{\g})\rightarrow \bigoplus_{\p \mid p}H^1(G_{F_{\p}},\scrF^-_{\p} T_{\f}(\lambda)) \hotimes_O \R.
		\end{equation}
		Let us fix trivializations for $ \p = \p_1, \p_2 $ compatible with \eqref{7.3} (cf. \textsection \ref{6.1.3})
		$$\EXP^*_{\scrF^-_{\p} T_{\f} (\lambda)} \colon H^1(G_{F_{\p}},\scrF^-_{\p} T_{\f} (\lambda)) \xrightarrow{\sim} \R_\f $$
		and define $ \Exp^*_{\f \otimes \lambda,\p}$ as the following composition \begin{align*}
			\widetilde{H}^1_{\mathrm{f}}(G_{F,\Sigma}, \T_2, \operatorname{tr}^*\Delta_{\g}) \xrightarrow{\res_{/\mathrm{bal}, \p}} H^1(G_{F_{\p}},\scrF^-_{\p} T_{\f}(\lambda)) \hotimes_O \R 
			\xrightarrow{\EXP^*_{\scrF^-_{\p} T_{\f} (\lambda)} \otimes \id } \R.
		\end{align*}
		\begin{rmk} \label{rmk : ExpT2}
			If $ \Exp^*_{\f \otimes \lambda,\p}(\bbdelta_1(\T_{2}, \tr \Delta_{\g})) $ for all $ \p = \p_1, \p_2 $ is not trivial, then Theorem \ref{logrln}. tells 
			\begin{equation}
				\bigotimes_{\p \mid p} \Exp^*_{\f \otimes \lambda,\p} (\bbdelta_1(\T_{2}, \tr \Delta_{\g})) = \Char (\widetilde{H}^2_{\mathrm{f}}(G_{F,\Sigma}, \T_2, \operatorname{tr}^*\Delta_{\g})) \Char(\widetilde{H}^2_{\mathrm{f}}(G_{F,\Sigma}, \T_2, \operatorname{tr}^*\Delta_{\mathrm{bal}})).
			\end{equation}
			
		\end{rmk}
		We note that $\widetilde{R\Gamma}_{\mathrm{f}}(G_{F,\Sigma},\T_2, \operatorname{tr}^*\Delta_?) \in D_{\mathrm{parf}}^{[1,2]}(\prescript{}{\R}{\mathrm{Mod}}) $ for $?=\g, \b $ and hence 
		\begin{equation}
			\widetilde{H}^1_{\mathrm{f}}(G_{F,\Sigma}, \T_2, \operatorname{tr}^*\Delta_{\mathrm{bal}}) = \{0\}, \quad \widetilde{H}^2_{\mathrm{f}}(G_{F,\Sigma}, \T_2, \operatorname{tr}^*\Delta_{\mathrm{bal}}) \text{ is torsion,}
		\end{equation}
		\begin{equation}
			\begin{split}
				& \widetilde{H}^1_{\mathrm{f}}(G_{F,\Sigma}, \T_2, \operatorname{tr}^*\Delta_{\g}) \text{ is torsion free of rank two,} \quad \widetilde{H}^2_{\mathrm{f}}(G_{F,\Sigma}, \T_2, \operatorname{tr}^*\Delta_{\g}) \text{ is torsion,} \\
				&  \widetilde{H}^1_{\mathrm{f}}(G_{F,\Sigma}, \T_2, \operatorname{tr}^*\Delta_{\g})\rightarrow \bigoplus_{\p \mid p}H^1(G_{F_{\p}},\scrF^-_{\p} T_{\f}( \lambda)) \hotimes_O \R \xrightarrow{\sim} \R^2 \text{ is injective}
			\end{split}
		\end{equation}
		whenever $  \bbdelta_2(\T_{2}, \tr \Delta_{\g})) \ne 0 $.

		\subsubsection{}
		Note that we have $$\chi(\widetilde{R\Gamma}_{\mathrm{f}}(G_{F,\Sigma},\T_{3},\Delta_{\g})) =0$$. Therefore we get a submodule of leading terms
		\begin{equation}
			\bbdelta_0(\T_{3},\Delta_{\g})  \subset \R.
		\end{equation}
		
		Recall also the local conditions $ \Delta_+ $, we have $ \chi(\widetilde{R\Gamma}_{\mathrm{f}}(G_{F,\Sigma}, \T_{3}, \Delta_+)) =-2 $. Hence we get submodule $$ \bbdelta_2(\T_{3},\Delta_{+}) \subset \bigcap^2_R \widetilde{H}^1(G_{F,\Sigma}, \T_{3}, \Delta_+) $$
		Let us consider the natural map $ \res^{/\g}_\p $ as the composition of (induced by the inclusion $ \scrF^{+\g}_{\p} \T_{3} \rightarrow \scrF^{+\mathrm{bal}^+}_{\p} \T_{3} $) for $ \p = \p_1,\p_2  $
		\begin{equation*}
			\widetilde{H}^1(G_{F,\Sigma}, \T_{3}, \Delta_+) \rightarrow H^1(G_{F_{\p}}, \scrF^+_{\p}T_{\f} \otimes_O \scrF^-_{\p}T_{\g} \otimes_O \scrF^+_{\p}T_{\g}^*(\lambda)) \xrightarrow{\sim}  H^1(G_{F_{\p}}, \scrF^+_{\p}T_{\f}(\lambda)) \widehat{\otimes}_O \R,
		\end{equation*}
		Then define 
		\[\res^{/\g}_p \coloneqq \bigoplus_{\p \mid p} \res^{/\g}_\p.\]
		Lets assume that $ \bbdelta_0(\T_{3},\Delta_{\g}) \ne 0 $
		\begin{rmk}
			\label{remark_T3}
			Theorem \ref{logrln} gives us
			$$ \bigotimes_{\p \mid p}\Log_{\scrF^+_{\p} T_\f(\lambda)} (\bbdelta_1(\T_{3},\Delta_{+})) = \Char (\widetilde{H}^2_{\mathrm{f}}(G_{F,\Sigma}, \T_3, \Delta_{\g})) \Char(\widetilde{H}^2_{\mathrm{f}}(G_{F,\Sigma}, \T_3, \Delta_{+})) .$$
			We further remark that $ \widetilde{H}_{\rm f}^1(G_{F,\Sigma},\T_{3}, \Delta_{\g}) =0 $ by Theorem \ref{r0}.
		\end{rmk}
		\begin{prop}
			If $ \bbdelta_0(\T_{3}.\Delta_\g) \ne 0 $, then also 
			$$ \res_{\p}(\bbdelta_1(\T_{1}, \Delta_\emptyset)) \ne 0 \ne \bbdelta_1(\T_{2}, \tr\Delta_\g). $$
		\end{prop}
		\begin{proof}
			We have the natural map
			$$ \widetilde{H}^1_{\rm f}(G_{F,\Sigma}, \T_1, \Delta_0) \rightarrow \widetilde{H}^1_{\rm f}(G_{F,\Sigma}, \T_3, \Delta_\g) $$
			is injective via \ref{comtriangle}. Therefore by Remark \ref{remark_T3} we get $ \widetilde{H}_{\rm f}^1(G_{F,\Sigma}, \T_1, \Delta_0) = 0$. The self-duality of Selmer complexes gives $ \widetilde{H}_{\rm f}^2(G_{F,\Sigma}, \T_1, \Delta_\emptyset) $ is torsion. Hence Theorem \ref{nontorsion} tells us that $ \bbdelta_2(\T_{1},\Delta_\emptyset) \ne 0 $. Moreover, since 
			$$ \{0\} =  \widetilde{H}^1_{\rm f}(G_{F,\Sigma}, \T_1, \Delta_0) = \ker\left(  \widetilde{H}^1_{\rm f}(G_{F,\Sigma}, \T_1, \Delta_\emptyset) \xrightarrow{\res_p} \bigoplus_{\p \mid p} H^1(G_{F_\p}, \T_{1}) \right),  
			$$
			we conclude that $ \res_\p(\bbdelta_1(\T_{1},\Delta_\emptyset)) \ne 0 $ as required.\\
			We note that \ref{comtriangle} and the discussion above gives rise to an exact sequence on $ \R $-modules
			\begin{equation}\label{exact_sep28}
				0 \rightarrow 	\widetilde{H}^1_{\rm f}(G_{F,\Sigma}, \T_2, \tr \Delta_\g) \rightarrow\underbrace{ \widetilde{H}^2_{\rm f}(G_{F,\Sigma}, \T_1, \Delta_0)}_{\rm rank =2} \rightarrow \underbrace{\widetilde{H}^2_{\rm f}(G_{F,\Sigma}, \T_3, \Delta_\g)}_{\rm torsion}\rightarrow  \widetilde{H}^2_{\rm f}(G_{F,\Sigma}, \T_2, \tr \Delta_\g) \rightarrow 0 . 	
			\end{equation} 
			Hence \ref{exact_sep28} shows that $ \widetilde{H}^1_{\rm f}(G_{F,\Sigma}, \T_2, \tr \Delta_\g) $ is of rank two and $ \widetilde{H}^2_{\rm f}(G_{F,\Sigma}, \T_2, \tr \Delta_\g) $ is torsion. therefore we conclude that $ \bbdelta_2(\T_{2}, \tr\Delta_\g) \ne 0 $, which gives the second asserted non vanishing. 
		\end{proof}
		We may also consider the natural map $ \res^{/\b}_p \coloneqq \res^{/\b}_\p $ as the composition of (induced by the inclusion $ \scrF^{+\b}_{\p} \T_{3} \rightarrow \scrF^{+\mathrm{bal}^+}_{\p} \T_{3}, \: \p \mid p $)
		$$  \widetilde{H}^1_{\mathrm{f}}(G_{F,\Sigma}, \T_{3}, \Delta_+) \rightarrow \bigoplus_{\p \mid p} H^1(G_{F_{\p}}, \scrF^-_{\p}T_{\f} \otimes_O \scrF^+_{\p}T_{\g} \otimes_O \scrF^-_{\p}T_{\g}^*(\lambda)) \xrightarrow{\sim} \bigoplus_{\p \mid p} H^1(G_{F_{\p}}, \scrF^-_{\p}T_{\f} (\lambda)) \widehat{\otimes}_O \R $$
		\begin{rmk}
			$ \Exp^*_{\f \otimes \lambda, \p}(\bbdelta_1(\T_{3},\Delta_{+})) = 0 $ and we also get that $\bbdelta_0(\T_{3},\Delta_{\mathrm{bal}}) = 0$.
		\end{rmk}
		\subsection{Factorization}
		We will explicitly describe the map 
		$$ \delta^1 \colon \widetilde{H}^1_{\mathrm{f}}(G_{F,\Sigma}, \T_2, \operatorname{tr}^*\Delta_{\g}) \longrightarrow \widetilde{H}^2_{\mathrm{f}}(G_{F,\Sigma}, \T_1, \Delta_{0}) $$
		arising from \ref{comtriangle}. \\
		Let \begin{equation}
			\label{cocyclex_f}
			x_{\mathrm{f}} \coloneqq (x,(x_v), (\xi_v)) \in \widetilde{Z}^1_{\mathrm{f}}\left( G_{F, \Sigma}, \T_{2}, \tr \Delta_{\g}  \right)  
		\end{equation} 
		be any cocycle: 
		$$ x \in Z^1(G_{F,\Sigma}, \T_{2}), \: x_v \in U_v(\tr \Delta_{\g}, \T_{2}), \: \xi_v \in \T_{2},   $$
		$$ dx_v = 0 , \quad \res_v(x)- i^+_v(x_v) = d\xi_v \quad \forall v \in \Sigma. $$
		Since $ U^{\bullet}_{\p}(\tr \Delta_{\g}, \T_{2}) \coloneqq C^{\bullet}(G_{F_{\p}}, F_{\p}^{+\g}\T_{2}) $ and $ F_{\p}^{+\g}\T_{2} $ is isomorphic image of $ F_{\p}^{+\g}\T_{3} $, we have  
		$$ U^{\bullet}_{\p}(\tr \Delta_{\g}, \T_{2}) \cong U^{\bullet}_{\p}(\Delta_{\g}, \T_{3}) $$
		for all $  \p \mid p $. Hence we can view $ x_{\p} =: y_{\p} $ as an element of $ U^1_{\p}(\Delta_{\g}, \T_{3}) $ for all $  \p \mid p $.
		For each $ v \in \Sigma \ba \bP $, we have a canonical identification 
		$$ U^{1}_v(\tr \Delta_{\g}, \T_{2}) \oplus U^{1}_v( \Delta_{0}, \T_{1}) \cong U^{1}_v(\Delta_{\g}, \T_{3}) $$
		which arises from the split short exact sequence \ref{mainses}. We also have the following identification 
		$$ C^{\bullet}(G_{F, \Sigma}, \T_{2}) \oplus C^{\bullet}(G_{F, \Sigma}, \T_{1}) \cong C^{\bullet}(G_{F, \Sigma}, \T_{3}) $$
		Put $ y = \iota(x) \; + \tr(0)  $, which is a cocycle in $ Z^1(G_{F,\Sigma}, \T_{3}) $ with $ dy =0 $. Similarly we can define $$ \mu_v = \iota(\xi_v) + \tr(0) \in \T_{3} \quad \forall v \in \Sigma .$$ Then the cochain $$ y_\mathrm{f} \coloneqq (y,(y_v),(\mu_v)) \in \widetilde{C}^1_{\mathrm{f}}\left( G_{F, \Sigma}, \T_{3}, \Delta_{\g}  \right) $$
		maps to $ x_\f $ under the morphism $ \pi_{\tr} $ since we have $ \pi_{\tr} \circ \iota = \operatorname{id}_{\T_{2}} $. Moreover, for all $ v \in \Sigma^{p} $, we have 
		\begin{equation*}
			d\mu_v- \res_{v}(y)+ i_v^{+}(y_v) = \iota (d\xi_v - \res_{v}(x) + i_v^{+}(x_v))=0.
		\end{equation*}
		
		Now let's focus on $ \p \in \mathbbl{P} $. Observe that the image of $ 	d\mu_{\p}- \res_{\p}(y)+ i_{\p}^{+}(y_{\p}) \in C^1(G_{F_{\p}}, \T_{3}) $ under the morphism 
		\begin{equation*}
			C^1(G_{F_{\p}}, \T_{3}) \xrightarrow{\pi_{\tr}}C^1(G_{F_{\p}}, \T_{2})
		\end{equation*}
		equals $ d\xi_{\p} - \res_{\p}(x) + i_{\p}^{+}(x_{\p}) =0  $, hence 
		\begin{equation*}
			\begin{split}
				d\mu_{\p}- \res_{\p}(y)+ i_{\p}^{+}(y_{\p}) & \in \ker\left( C^1(G_{F_{\p}}, \T_{3}) \xrightarrow{\pi_{\tr}} C^1(G_{F_{\p}}, \T_{2}) \right) \\
				& = \operatorname{im}\left( C^1(G_{F_{\p}}, \T_{1}) \xhookrightarrow{\tr} C^1(G_{F_{\p}}, \T_{3}) \right) .
			\end{split}
		\end{equation*}
		If we denote $ z_{\p} \in C^1(G_{F_{\p}}, \T_{1}) $ such that 
		$$ \tr(z_{\p}) = d\mu_{\p}- \res_{\p}(y)+ i_{\p}^{+}(y_{\p}) ,$$
		we get $$ dy_{\mathrm{f}} = \left( dy = 0, (dx_{v} = 0 )_{v \in \Sigma}, (\{0\}_{v \in \Sigma \ba \bP},\{\tr z_{\p}\}_{\p \in \mathbbl{P}} ) \right) \in \widetilde{C}^2_{\mathrm{f}}\left( G_{F, \Sigma}, \T_{3}, \Delta_{\g}  \right) $$
		As $ d\mu_{\p}- \res_{\p}(y)+ i_{\p}^{+}(y_{\p}) \in C^1(G_{F_{\p}}, \T_{3}) $ is a cocycle, we get $ d\tr (z_{\p}) =0  $ and so $ z_{\p} \in C^1(G_{F_{\p}}, \T_{1}) $ is also a cocycle. We conclude that 
		$$ \tr(0,0,z_p)_{\mathrm{f}}= (0, (0)_{v \in \Sigma}, (\{0\}_{v \in \Sigma \ba \bP},\{\tr z_{\p}\}_{\p \in \mathbbl{P}})) \in  \widetilde{Z}^2_{\mathrm{f}}\left( G_{F, \Sigma}, \T_{1}, \tr \Delta_{0}  \right) $$
		and $\tr(0,0,z_p)_{\mathrm{f}} = dy_{\mathrm{f}}$, where $ (0,0,z_p)_{\mathrm{f}} \coloneqq (0, (0)_{v \in \Sigma}, (\{0\}_{v \in \Sigma \ba \bP},\{z_{\p}\}_{\p \in \mathbbl{P}})) $. Hence we get the following proposition:
		\begin{prop}
			For a cocycle given as in (\ref{cocyclex_f}), we have 
			$$ \delta^1([x_{\mathrm{f}}]) = [(0,0,z_p)_{\mathrm{f}}] \in  \widetilde{H}^2_{\mathrm{f}}(G_{F,\Sigma}, \T_1, \Delta_{0}) $$
		\end{prop}
		\begin{rmk}
			We note that for $ \p = \p_1, \p_2 $ the diagram
			\[\xymatrix{
				&\scrF_{\p}^{+\g}\T_3  \ar@{^{(}->}[r]^{i_{\p}^+} \ar[rd]_{\sim}^{\pi_{\tr}} &\T_3 \\
				& &\scrF_{\p}^{+\g}\T_2 \ar@{^{(}->}[u]_{\iota}
			}\]
			does not commute. In consequence, $ z_{\p} $ comes about to account for this failure.
		\end{rmk}
		\subsubsection{}
		Now we will describe the cocycle $ z_{\p} $ in terms of $ y_{\p} $.
		Let $\{ u^+_{\p} \} \subset \scrF^+_{\p} T_\g$ denote an $\R_\g$-basis of and let $\{ u^+_{\p} , u^-_{\p}  \} \subset  T_\g $ be an extended basis. Let us denote by $ \{ \breve{u}^+_{\p} , \breve{u}^-_{\p}  \} \subset  T_\g^*  $ the dual basis. Recall that 
		$$ \mathrm{Fil}^2_{\p}\ad(T_\g)(\lambda) \coloneqq \Hom(T_\g,\scrF^{+}_{\p} T_\g)(\lambda) = \scrF^{+}_{\p} T_\g \widehat{\otimes} T_\g^* (\lambda) =\left\langle u^+_{\p} \otimes \breve{u}^-_{\p}, u^+_{\p} \otimes \uu^+_{\p} \right\rangle $$
		\begin{equation*}
			\begin{split}
				\mathrm{Fil}^{+\g}_{\p} \ad^0(T_\g)(\lambda) \coloneqq & \ker \{\ad^0(T_\g) \hookrightarrow \ad(T_\g) \rightarrow \Hom(F_{\p}^+ T_\g,F_{\p}^- T_\g) \}(\lambda) \\
				= & \left\langle u^+_{\p} \otimes \breve{u}^-_{\p}, u^+_{\p} \otimes \uu^+_{\p}  - u^-_{\p} \otimes \uu^-_{\p} \right\rangle
			\end{split}
		\end{equation*}
		$$ \scrF^{+\g}_{\p} \T_{2} = T_\f \widehat{\otimes} \Fil^{+\g}_{\p} \ad^0(T_\g)(\lambda)  $$
		$$ \Fil^{\g}_{v}\ad^0(T_\g)(\lambda) \coloneqq \Hom(F_{\p}^- T_\g,F_{\p}^+  T_\g)(\lambda) = \left\langle u^+_{\p} \otimes \breve{u}^-_{\p} \right\rangle \subset 	\mathrm{Fil}^{+\g}_{\p} \ad^0(T_\g)(\lambda) $$
		all are $G_{F_{\p}}$-stable submodule such that the $G_{F_{\p}}$ action on the quotient  
		$$ \mathrm{gr}^{+\g}_{\p} \ad^0(T_\g) \coloneqq \mathrm{Fil}^{+\g}_{\p} \ad^0(T_\g)(\lambda) / \Fil^{\g}_{v}\ad^0(T_\g)(\lambda) $$
		is via $\lambda$ (c.f. the proof of Lemma \ref{localses}). \\
		Let us fix a basis $\{ w_1,w_2 \}$ of $T_\f$. Given $g \in G_{F_{\p}}$, let us define $a_i(g), b_i(g)  \in \R$ ($i = 1,2$) such that 
		$$ i_{\p}^+(x_{\p})(g) = \sum_{i=1}^{2} a_i(g)(u^+_{\p} \otimes \breve{u}^-_{\p})\otimes w_i + \sum_{i=1}^{2} b_i(g)( u^+_{\p} \otimes \uu^+_{\p}  - u^-_{\p} \otimes \uu^-_{\p})\otimes w_i . $$
		Hence the cocycle $i_{\p}^+(y_{\p}) \in C^1(G_{F_{\p},\T_{3}}) $ is then given explicitly by  
		$$ i_{\p}^+(y_{\p})(g) = \sum_{i=1}^{2} a_i(g)(u^+_{\p} \otimes \breve{u}^-_{\p})\otimes w_i + \sum_{i=1}^{2} b_i(g)( u^+_{\p} \otimes \uu^+_{\p} )\otimes w_i . $$
		By definition, the cocycle \( \operatorname{tr} \circ i_{\p}^+(y_{\p}) \in C^1(G_{F_{\p}}, \T_{1}) \) is the cocycle 
		$$ \operatorname{tr} \circ i_{\p}^+(y_{\p})(g) =  \sum_{i=1}^{2} b_i(g)w_i . $$
		Hence we get 
		\begin{prop}
			The cocycle $z_{\p} = \operatorname{tr} \circ i_{\p}^+(y_{\p}) $ agrees with the image of $i_{\p}^+(x_{\p}) $	under the natural projection map
			\begin{equation}
				C^1(G_{F_{\p}},\scrF^{+\g}_{\p} \T_{2}) \xrightarrow{\operatorname{pr}_{/\g,\p}} C^1(G_{F_{\p}}, T_\f \widehat{\otimes}\mathrm{gr}^{+\g}_{\p} \ad^0(T_\g)(\lambda)) \xrightarrow[\rm Lemma \ \ref{localses}]{\sim} C^1(G_{F_{\p}}, T_\f(\lambda)) \otimes \R,
			\end{equation}
			In particular the morphism $\delta^1$ factors as
			\begin{equation}
				\begin{tikzcd}[row sep= 1ex, column sep=small]
					[x_\mathrm{f}]=[(x,(x_{v}), (\xi_{v}))]\arrow[rr,mapsto] \arrow[d, symbol = \in ] \arrow[ddddddd,mapsto, bend right =75]  &          &\delta^1([x_\mathrm{f}])=[(0,0,\operatorname{pr}_{/\g}\circ i_{\p}^+(x_{\p}))] \arrow[d, symbol = \in ]\\
					\widetilde{H}^1_{\mathrm{f}}(G_{F,\Sigma}, \T_2, \operatorname{tr}^*\Delta_{\g})\arrow[rr, "\delta^1"]\arrow[ddddd]                              &            &\widetilde{H}^2_{\mathrm{f}}(G_{F,\Sigma}, \T_1, \Delta_{0})\\
					& & \\
					& & \\
					& & \\
					& & \\
					\smashoperator[r]{\bigoplus_{\p \mid p}}	H^1(G_{F_{\p}},\scrF^{+\g}_{\p} \T_{2})\arrow[r,"\operatorname{pr}_{/\g}" ']                                      &\smashoperator[r]{\bigoplus_{\p \mid p}}	H^1(G_{F_{\p}}, T_\f \widehat{\otimes}\mathrm{gr}^{+\g}_{\p} \ad^0(T_\g)(\lambda))\arrow[r, "\sim" ']  &\smashoperator[r]{\bigoplus_{\p \mid p}}	H^1(G_{F_{\p}}, \T_{1})\arrow[uuuuu, "\partial^1_{\f} " ']\\
					\left[ (i_{\p}^+(x_{\p}))_{\p} \right] =\left[ (\res_{\p}(x))_{\p} \right]  \arrow[rr,mapsto]  \arrow[u, symbol = \in ]                   &            &\left[ (\operatorname{pr}_{/\g,\p}(i_{\p}^+(x_{\p})))_{\p} \right] \arrow[u, symbol = \in ]\arrow[uuuuuuu,mapsto, bend right= 67]  
				\end{tikzcd}
			\end{equation}
			where the morphism $\partial_{\f \otimes \lambda}^1$ is induced from the exact sequence
			$$ \widetilde{R\Gamma}_{\mathrm{f}}(G_{F,\Sigma},\T_{1}, \Delta_{0}) \rightarrow \widetilde{R\Gamma}_{\mathrm{f}}(G_{F,\Sigma},\T_{1}, \Delta_{\emptyset}) \xrightarrow{\res_p}  \bigoplus_{\p \mid p} R\Gamma(G_{F_{\p}}, \T_{1}) \xrightarrow{\partial_{\f \otimes \lambda}}[+1] $$
			In particular 
			$$\bigoplus_{\p \mid p} H^1(G_{F_{\p}}, \T_{1}) \ni ([a_{\p}])_{\p} \xmapsto{\partial_{\f \otimes \lambda}^1} [(0,0,(\{0\}_{v \in \Sigma \ba \bP},\{a_{\p}\}_{\p \in \mathbbl{P}}))] \in \widetilde{H}_{\mathrm{f}}^2(G_{F,\Sigma},\T_{1}, \Delta_{0}). $$
		\end{prop}
		\begin{cor}
			\label{cor : delta_and_delta}
			Suppose $\bbdelta_0(\T_{3},\Delta_{\g}) \ne 0 $. then the map $ \delta^1 $ factors as 
			\begin{equation}
				\begin{tikzcd}[column sep = small]
					\widetilde{H}^1_{\mathrm{f}}(G_{F,\Sigma}, \T_2, \operatorname{tr}^*\Delta_{\g}) \arrow[rr, hookrightarrow, "\delta^1"] \arrow[dr, hookrightarrow,"\res_{/\mathrm{pan}}" ']&&\widetilde{H}^2_{\mathrm{f}}(G_{F,\Sigma}, \T_1, \Delta_{0})\\
					&\smashoperator[r]{\bigoplus_{\p \mid p}}\frac{H^1(G_{F_{\p}}, \T_{1})}{\res_{\p}(\widetilde{H}^1_{\mathrm{f}}(G_{F,\Sigma},\T_{1}, \Delta_{\pan}))} 
					\arrow[ur, hookrightarrow, "\partial_{\f \otimes \lambda}^1" '] &
				\end{tikzcd}
			\end{equation}
			and all the $ \R $-modules that appear are of rank two.
		\end{cor}
		
		\begin{proof}
			$ \delta^1 $ is Injective because $ \widetilde{H}^1_{\mathrm{f}}(G_{F,\Sigma}, \T_3, \Delta_{\g}) =0 $ since $\bbdelta_0(\T_{3},\Delta_{\g}) \ne 0 $ via Theorem \ref{r0}. \\
			Previous Proposition tells that $ \delta^1 $ factors as 
			\begin{equation}
				\begin{tikzcd}
					\widetilde{H}^1_{\mathrm{f}}(G_{F,\Sigma}, \T_2, \operatorname{tr}^*\Delta_{\g}) \arrow[r, hook, "\delta^1"]
					\arrow[d, hook, "\operatorname{pr}_{/\g} \circ \res_{p} " '] & \widetilde{H}^2_{\mathrm{f}}(G_{F,\Sigma}, \T_1, \Delta_{0}) \\
					\smashoperator[r]{\bigoplus_{\p \mid p}}H^1(G_{F_{\p}}, \T_{1}) \arrow[ur, "\partial_{\f \otimes \lambda}^1 " '] & 
				\end{tikzcd}
			\end{equation}
			Also $ \partial_{\f \otimes \lambda}^1$ factors as 
			\begin{equation}
				\begin{tikzcd}
					\smashoperator[r]{\bigoplus_{\p \mid p}}H^1(G_{F_{\p}}, \T_{1}) \arrow[r, "\partial_{\f \otimes \lambda}^1 "] \arrow[dr, twoheadrightarrow] & \widetilde{H}^2_{\mathrm{f}}(G_{F,\Sigma}, \T_1, \Delta_{0}) \\
					& \smashoperator[r]{\bigoplus_{\p \mid p}} \frac{H^1(G_{F_{\p}}, \T_{1})}{\res_{\p}(\widetilde{H}^1_{\mathrm{f}}(G_{F,\Sigma},\T_{1}, \Delta_{\emptyset}))} \arrow[u, hook, " \partial_{\f \otimes \lambda}^1 " ']
				\end{tikzcd}
			\end{equation}
			Hence we have a injective homomorphism $ \res_{/\pan}\coloneqq \bigoplus_{\p \mid p}\res_{/\pan, \p} $ which is the composition of 
			$$ \widetilde{H}^1_{\mathrm{f}}(G_{F,\Sigma}, \T_2, \operatorname{tr}^*\Delta_{\g}) \xrightarrow{\operatorname{pr}_{/\g} \circ \res_p} \bigoplus_{\p \mid p}H^1(G_{F_{\p}}, \T_{1})\twoheadrightarrow \bigoplus_{\p \mid p} \frac{H^1(G_{F_{\p}}, \T_{1})}{\res_{\p}(\widetilde{H}^1_{\mathrm{f}}(G_{F,\Sigma},\T_{1}, \Delta_{\emptyset}))} $$
			Since $ \widetilde{H}^1_{\mathrm{f}}(G_{F,\Sigma},\T_{1}, \Delta_{\emptyset}) = \widetilde{H}^1_{\mathrm{f}}(G_{F,\Sigma},\T_{1}, \Delta_{\pan}) $ we have the  required result.
		\end{proof}
		Recall our map $ \res_{/\b}\coloneqq\bigoplus_{\p \mid p}\res_{/\b, \p} $ (see \ref{resbal_2}) given as composite of 
		$$ \widetilde{H}^1_{\mathrm{f}}(G_{F,\Sigma}, \T_2, \operatorname{tr}^*\Delta_{\g}) \xrightarrow{\operatorname{pr}_{/\g} \circ \res_p} \bigoplus_{\p \mid p}H^1(G_{F_{\p}}, \T_{1}) \rightarrow \bigoplus_{\p \mid p} H^1(G_{F_\p},\scrF^-_\p T_\f \otimes \lambda) \otimes \R $$
		Hence $ \res_{/\b} $ factors as 
		$$ \res_{/\b} \colon \widetilde{H}^1_{\mathrm{f}}(G_{F,\Sigma}, \T_2, \operatorname{tr}^*\Delta_{\g}) \xrightarrow{\res_{/\pan}} \bigoplus_{\p \mid p} \frac{H^1(G_{F_{\p}}, \T_{1})}{\res_{\p}(\widetilde{H}^1_{\mathrm{f}}(G_{F,\Sigma},\T_{1}, \Delta_{\emptyset}))} \rightarrow \bigoplus_{\p \mid p} H^1(G_{F_\p},\scrF^-_\p T_\f \otimes \lambda) \otimes \R $$
		Since $ \widetilde{H}^1_{\mathrm{f}}(G_{F,\Sigma},\T_{1}, \Delta_{\emptyset}) = \widetilde{H}^1_{\mathrm{f}}(G_{F,\Sigma},\T_{1}, \Delta_{\pan}) $ in our present set-up, we have the following exact sequence for all $ \p \mid p $ 
		$$ 0 \rightarrow \frac{H^1(G_{F_\p},\scrF^+_\p \T_{1})}{\res_{\p}(\widetilde{H}^1_{\mathrm{f}}(G_{F,\Sigma},\T_{1}, \Delta_{\pan})) } \rightarrow \frac{H^1(G_{F_\p},\T_{1})}{\res_{\p}(\widetilde{H}^1_{\mathrm{f}}(G_{F,\Sigma},\T_{1}, \Delta_{\pan})) } \rightarrow H^1(G_{F_\p},\scrF^-_\p \T_{1}) \rightarrow 0 $$
		As the map $ \res_{/\pan} $ is injective and $ \widetilde{H}^1_{\mathrm{f}}(G_{F,\Sigma}, \T_2, \operatorname{tr}^*\Delta_{\g}) $ is torsion-free, 
		$$ \res_{/\pan, \p}(\bbdelta_1(\T_2, \operatorname{tr}^*\Delta_{\g})) \subset \frac{H^1(G_{F_\p},\scrF^+_\p \T_{1})}{\res_{\p}(\widetilde{H}^1_{\mathrm{f}}(G_{F,\Sigma},\T_{1}, \Delta_{\pan})) } $$ 
		is also torsion-free.
		Since $$ \frac{H^1(G_{F_\p},\scrF^+_\p \T_{1})}{\res_{\p}(\widetilde{H}^1_{\mathrm{f}}(G_{F,\Sigma},\T_{1}, \Delta_{\pan})) } $$ is torsion
		$\res_{/\pan, \p}(\bbdelta_1(\T_2, \operatorname{tr}^*\Delta_{\g}))$ maps isomorphically onto its image 
		$$ \res_{/\b, \p}(\bbdelta_1(\T_2, \operatorname{tr}^*\Delta_{\g})) \subset H^1(G_{F_\p},\scrF^-_\p \T_{1}) $$
		Hence we have an exact sequence
		\begin{equation}
			0 \rightarrow \frac{H^1(G_{F_\p},\scrF^+_\p \T_{1})}{\res_{\p}(\widetilde{H}^1_{\mathrm{f}}(G_{F,\Sigma},\T_{1}, \Delta_{\pan})) } \rightarrow \dfrac{\frac{H^1(G_{F_\p},\T_{1})}{\res_{\p}(\widetilde{H}^1_{\mathrm{f}}(G_{F,\Sigma},\T_{1}, \Delta_{\pan})) }}{\res_{/\pan, \p}(\bbdelta_1(\T_2, \operatorname{tr}^*\Delta_{\g}))} \rightarrow \frac{H^1(G_{F_\p},\scrF^-_\p \T_{1})}{ \res_{/\b, \p}(\bbdelta_1(\T_2, \operatorname{tr}^*\Delta_{\g}))} \rightarrow 0
		\end{equation} 
		\begin{prop}
			In the setting of Corollary \ref{cor : delta_and_delta}, we have 
			\begin{multline}
				\Char\left( \dfrac{\widetilde{H}^2_{\mathrm{f}}(G_{F,\Sigma}, \T_1, \Delta_{0})}{\delta^1(\bbdelta_1(\T_{2}, \tr\Delta_{\g}))} \right) \\
				= \Char\left( \bigoplus_{\p \mid p}\frac{H^1(G_{F_\p},\scrF^+_\p \T_{1})}{\res_{\p}(\widetilde{H}^1_{\mathrm{f}}(G_{F,\Sigma},\T_{1}, \Delta_{\pan})) }\right) \Char\left( \bigoplus_{\p \mid p} \frac{H^1(G_{F_\p},\scrF^-_\p \T_{1})}{ \res_{/\b, \p}(\bbdelta_1(\T_2, \operatorname{tr}^*\Delta_{\g}))} \right) \\ \times \Char(\widetilde{H}^2_{\mathrm{f}}(G_{F,\Sigma}, \T_1, \Delta_{\emptyset}))
			\end{multline} 
			
		\end{prop}
		\begin{proof}
			We have exact sequence
			\begin{equation}
				0 \rightarrow \widetilde{H}^1_{\mathrm{f}}(G_{F,\Sigma}, \T_1, \Delta_{\emptyset}) \xrightarrow{\res_p} \bigoplus_{\p \mid p} H^1(G_{F_\p},\T_{1}) \xrightarrow{\partial_{\f \otimes \lambda}^1} \widetilde{H}^2_{\mathrm{f}}(G_{F,\Sigma}, \T_1, \Delta_{0}) \rightarrow \widetilde{H}^2_{\mathrm{f}}(G_{F,\Sigma}, \T_1, \Delta_{\emptyset}) \rightarrow 0.
			\end{equation}
			Hence we get another exact sequence
			\begin{equation}
				0 \rightarrow \bigoplus_{\p \mid p} \dfrac{\frac{H^1(G_{F_\p},\T_{1})}{\res_{\p}(\widetilde{H}^1_{\mathrm{f}}(G_{F,\Sigma},\T_{1}, \Delta_{\pan})) }}{\res_{/\pan, \p}(\bbdelta_1(\T_2, \operatorname{tr}^*\Delta_{\g}))} \xrightarrow{\partial_{\f \otimes \lambda}^1} \dfrac{\widetilde{H}^2_{\mathrm{f}}(G_{F,\Sigma}, \T_1, \Delta_{0})}{\delta^1(\bbdelta_1(\T_{2}, \tr\Delta_{\g}))}\rightarrow \widetilde{H}^2_{\mathrm{f}}(G_{F,\Sigma}, \T_1, \Delta_{\emptyset}) \rightarrow 0.
			\end{equation}
			Therefore we have 
			\begin{multline}
				\Char\left( \dfrac{\widetilde{H}^2_{\mathrm{f}}(G_{F,\Sigma}, \T_1, \Delta_{0})}{\delta^1(\bbdelta_1(\T_{2}, \tr\Delta_{\g}))} \right) \\
				= \Char \left( \bigoplus_{\p \mid p} \dfrac{\frac{H^1(G_{F_\p},\T_{1})}{\res_{\p}(\widetilde{H}^1_{\mathrm{f}}(G_{F,\Sigma},\T_{1}, \Delta_{\pan})) }}{\res_{/\pan, \p}(\bbdelta_1(\T_2, \operatorname{tr}^*\Delta_{\g}))} \right) \Char(\widetilde{H}^2_{\mathrm{f}}(G_{F,\Sigma}, \T_1, \Delta_{\emptyset})) \\
				= \Char\left( \bigoplus_{\p \mid p}\frac{H^1(G_{F_\p},\scrF^+_\p \T_{1})}{\res_{\p}(\widetilde{H}^1_{\mathrm{f}}(G_{F,\Sigma},\T_{1}, \Delta_{\pan})) }\right) \Char\left( \bigoplus_{\p \mid p} \frac{H^1(G_{F_\p},\scrF^-_\p \T_{1})}{ \res_{/\b, \p}(\bbdelta_1(\T_2, \operatorname{tr}^*\Delta_{\g}))} \right) \\
				\times \Char(\widetilde{H}^2_{\mathrm{f}}(G_{F,\Sigma}, \T_1, \Delta_{\emptyset}))
			\end{multline}	
		\end{proof}
		\begin{thm}
			\label{main_theorem_2}
			Suppose $  \bbdelta_0(\T_{3}.\Delta_\g) \ne 0 $ and the restriction $ \res_{/\b, \p}(\bbdelta_1(\T_2, \operatorname{tr}^*\Delta_{\g})) \ne 0 $ for all $ \p = \p_1 $, and $ \p_2 $. Then we get
			\begin{multline*}
				\Char \left( \widetilde{H}^2_{\mathrm{f}}(G_{F,\Sigma}, \T_3, \Delta_{\g})  \right) \\
				= \Char(\widetilde{H}^2_{\mathrm{f}}(G_{F,\Sigma}, \T_2, \operatorname{tr}^*\Delta_{\mathrm{bal}}))
				\left\lbrace \Char\left( \bigoplus_{\p \mid p}\frac{H^1(G_{F_\p},\scrF^+_\p \T_{1})}{\res_{\p}(\widetilde{H}^1_{\mathrm{f}}(G_{F,\Sigma},\T_{1}, \Delta_{\pan})) }\right) \Char(\widetilde{H}^2_{\mathrm{f}}(G_{F,\Sigma}, \T_1, \Delta_{\emptyset})) \right\rbrace .
			\end{multline*}
		\end{thm}
		\begin{proof}
			From \eqref{comtriangle}, we get an exact sequence
			\begin{equation}
				0 \rightarrow \widetilde{H}^1_{\mathrm{f}}(G_{F,\Sigma}, \T_2, \operatorname{tr}^*\Delta_{\g}) \xrightarrow{\delta^1} \widetilde{H}^2_{\mathrm{f}}(G_{F,\Sigma}, \T_1, \Delta_{0}) \rightarrow \widetilde{H}^2_{\mathrm{f}}(G_{F,\Sigma}, \T_3, \Delta_{\g}) \rightarrow \widetilde{H}^2_{\mathrm{f}}(G_{F,\Sigma}, \T_2, \operatorname{tr}^*\Delta_{\g}) \rightarrow 0
			\end{equation}
			This shows that 
			\begin{multline}
				\Char \left( \dfrac{\widetilde{H}^1_{\mathrm{f}}(G_{F,\Sigma}, \T_2, \operatorname{tr}^*\Delta_{\g})}{\bbdelta_1(\T_{2}, \tr\Delta_{\g})} \right) \Char \left( \widetilde{H}^2_{\mathrm{f}}(G_{F,\Sigma}, \T_3, \Delta_{\g})  \right) \\
				= \Char \left( \dfrac{\widetilde{H}^2_{\mathrm{f}}(G_{F,\Sigma}, \T_1, \Delta_{0})}{\delta^1(\bbdelta_1(\T_{2}, \tr\Delta_{\g}))} \right) \Char \left( \widetilde{H}^2_{\mathrm{f}}(G_{F,\Sigma}, \T_2, \operatorname{tr}^*\Delta_{\g}) \right)  
			\end{multline}
			Hence we get
			\begin{multline}
				\Char \left( \widetilde{H}^2_{\mathrm{f}}(G_{F,\Sigma}, \T_2, \operatorname{tr}^*\Delta_{\g}) \right)	\Char \left( \widetilde{H}^2_{\mathrm{f}}(G_{F,\Sigma}, \T_3, \Delta_{\g})  \right) =	\Char \left( \dfrac{\widetilde{H}^2_{\mathrm{f}}(G_{F,\Sigma}, \T_1, \Delta_{0})}{\delta^1(\bbdelta_1(\T_{2}, \tr\Delta_{\g}))} \right) \\	
				=  \Char\left( \bigoplus_{\p \mid p} \frac{H^1(G_{F_\p},\scrF^-_\p \T_{1})}{ \res_{/\b, \p}(\bbdelta_1(\T_2, \operatorname{tr}^*\Delta_{\g}))} \right) \Char\left( \bigoplus_{\p \mid p}\frac{H^1(G_{F_\p},\scrF^+_\p \T_{1})}{\res_{\p}(\widetilde{H}^1_{\mathrm{f}}(G_{F,\Sigma},\T_{1}, \Delta_{\pan})) }\right) \\
				\hfill \times \quad \Char(\widetilde{H}^2_{\mathrm{f}}(G_{F,\Sigma}, \T_1, \Delta_{\emptyset})) \qquad \qquad \\
				=\left\lbrace \bigotimes_{\p \mid p} \Exp^*_{\f\otimes \lambda, \p} (\bbdelta_1(\T_{2}, \tr \Delta_{\g}))\right\rbrace \Char\left( \bigoplus_{\p \mid p}\frac{H^1(G_{F_\p},\scrF^+_\p \T_{1})}{\res_{\p}(\widetilde{H}^1_{\mathrm{f}}(G_{F,\Sigma},\T_{1}, \Delta_{\pan})) }\right) \\
				\hfill \times \quad \Char(\widetilde{H}^2_{\mathrm{f}}(G_{F,\Sigma}, \T_1, \Delta_{\emptyset})) 
			\end{multline}
			By Remark \ref{rmk : ExpT2} we get
			\begin{multline}
				\Char \left( \widetilde{H}^2_{\mathrm{f}}(G_{F,\Sigma}, \T_3, \Delta_{\g})  \right) \\
				= \Char(\widetilde{H}^2_{\mathrm{f}}(G_{F,\Sigma}, \T_2, \operatorname{tr}^*\Delta_{\mathrm{bal}}))
				\left\lbrace \Char\left( \bigoplus_{\p \mid p}\frac{H^1(G_{F_\p},\scrF^+_\p \T_{1})}{\res_{\p}(\widetilde{H}^1_{\mathrm{f}}(G_{F,\Sigma},\T_{1}, \Delta_{\pan})) }\right) \Char(\widetilde{H}^2_{\mathrm{f}}(G_{F,\Sigma}, \T_1, \Delta_{\emptyset})) \right\rbrace 
			\end{multline}
			
		\end{proof}
		\begin{cor}
			\label{cor_main_2} 
			If $ \widetilde{H}^2_{\mathrm{f}}(G_{F,\Sigma}, \T_1, \Delta_{\emptyset})_{\R\rm -torsion} $ is a pseudo-null then we have 
			\begin{align}
				\label{7.26}
				\bbdelta_0(T_3, \Delta_{\g})= \bbdelta_0(\T_2, \operatorname{tr}^*\Delta_{\mathrm{bal}}) 
				i_\f^*\left( \Log^2_{\f \otimes \lambda}(\mathbf{BK}^{\rm plec}_{\f \otimes \lambda, 1} \wedge \mathbf{BK}^{\rm plec}_{\f \otimes \lambda, 2} ) \right) .
			\end{align}
		\end{cor}
		\begin{proof}
			From the exact triangle 
			\begin{align}
				\dcom(\G, \T_{1}, \Delta_0) \rightarrow \dcom(\G, \T_{1}, \Delta_\pan) \rightarrow \bigoplus_{\p \mid p} R\Gamma (G_{F_\p}, \scrF_{\p}^+ \T_{1}) \xrightarrow{+1}
			\end{align}
			we get that 
			\[\bigoplus_{\p \mid p}\frac{H^1(G_{F_\p},\scrF^+_\p \T_{1})}{\res_{\p}(\widetilde{H}^1_{\mathrm{f}}(G_{F,\Sigma},\T_{1}, \Delta_{\pan})) } \cong \widetilde{H}^2_{\mathrm{f}}(G_{F,\Sigma}, \T_1, \Delta_{\emptyset})_{\R\rm -torsion}.  \]
			Hence from the given condition and \cref{main_theorem_2}, the required result follows immediately.
		\end{proof}
		\begin{rmk}
			As previously noted, we expect that (in view of Iwasawa–Greenberg main conjectures) $ \bbdelta_0(\T_{3}, \Delta_{\g}) $ is generated by $ 	\scr{L}_p^{\g}(\f \otimes \As(\g) )(P,Q)  $ and $ \bbdelta_0(\T_2, \operatorname{tr}^*\Delta_{\mathrm{bal}}) $ is generated by $ \scr{L}_p^{\b}(\f \otimes \breve{\g} )(P,Q) $. Whereas $  \Log^2_{\f \otimes \lambda}(\mathbf{BK}^{\rm plec}_{\f \otimes \lambda, 1} \wedge \mathbf{BK}^{\rm plec}_{\f \otimes \lambda, 2} )$ is expected to interpolate the second order derivatives of Perrin-Riou's $ \mbf D_{\rm cris}(V_{f \otimes \lambda}) \otimes \cal{H}(\Gamma_{\mrm{cycl}}) $-valued $ p $-adic $ L $-function of $ f \otimes \lambda $, a crystalline specialization of $ \f \otimes\lambda $ (cf. \cite{PR93,perrin-riou_padic}). 
			
			We also want to remark that this situation is a real quadratic version of the case stated in \cite[\textsection4.5.3]{Darmon16GenralisedKatoClasses}.
		\end{rmk}

		\bibliographystyle{amsalpha}
		
		\bibliography{ref.bib}

\providecommand{\bysame}{\leavevmode\hbox to3em{\hrulefill}\thinspace}
\providecommand{\MR}{\relax\ifhmode\unskip\space\fi MR }
\providecommand{\MRhref}[2]{%
  \href{http://www.ams.org/mathscinet-getitem?mr=#1}{#2}
}
\providecommand{\href}[2]{#2}
\begin{thebibliography}{BCPd24}

\bibitem[BC25]{BC25}
Kâzım B{\"u}yükboduk and Daniele Casazza, \emph{On the artin formalism for triple product $p$-adic $l$-functions: Super-factorization}, Preprint, {arXiv}:2301.08383, 2025.

\bibitem[BCPd24]{2024artinformalismtripleproduct}
Kâzım B{\"u}yükboduk, Daniele Casazza, Aprameyo Pal, and Carlos {de Vera-Piquero}, \emph{On the {A}rtin formalism for triple product $p$-adic {$L$}-functions: {C}how--{H}eegner points vs. {H}eegner points}, Preprint, {arXiv}:2409.08645, 2024.

\bibitem[BDP13]{BDP13}
Massimo Bertolini, Henri Darmon, and Kartik Prasanna, \emph{Generalized {H}eegner cycles and {$p$}-adic {R}ankin {$L$}-series}, Duke Math. J. \textbf{162} (2013), no.~6, 1033--1148, With an appendix by Brian Conrad. \MR{3053566}

\bibitem[BF20]{BCF}
Iv\'an {Blanco-Chac\'on} and Michele Fornea, \emph{Twisted triple product {$p$}-adic {$L$}-functions and {H}irzebruch-{Z}agier cycles}, J. Inst. Math. Jussieu \textbf{19} (2020), no.~6, 1947--1992. \MR{4166999}

\bibitem[BH24]{bergdall2021padiclfunctionshilbertmodular}
John Bergdall and David Hansen, \emph{On {$p$}-adic {$L$}-functions for {H}ilbert modular forms}, Mem. Amer. Math. Soc. \textbf{298} (2024), no.~1489, v+125. \MR{4772264}

\bibitem[BM19]{salazar2019tripleproductpadiclfunctions}
Daniel {Barrera-Salazar} and Santiago Molina, \emph{Triple product {$ p $}-adic {$L$}-functions for {S}himura curves over totally real number fields}, Preprint, {arXiv}:1908.00091, 2019.

\bibitem[BS22]{BS22}
K\^{a}zım B\"{u}y\"{u}kboduk and Ryotaro Sakamoto, \emph{On the non-critical exceptional zeros of {K}atz {$p$}-adic {$L$}-functions for {CM} fields}, Adv. Math. \textbf{406} (2022), Paper No. 108478, 54. \MR{4438060}

\bibitem[BS25]{BCS23}
Kâzım B{\"u}yükboduk and Ryotaro Sakamoto, \emph{On the artin formalism for triple product $p$-adic $l$-functions}, Preprint, {arXiv}:2501.06541, 2025.

\bibitem[BSS18]{BSS18}
David Burns, Ryotaro Sakamoto, and Takamichi Sano, \emph{On the theory of higher rank {E}uler, {K}olyvagin and {S}tark systems, {II}}, Preprint, {arXiv}:1805.08448, 2018.

\bibitem[BSV22]{BSV22b}
Massimo Bertolini, Marco~Adamo Seveso, and Rodolfo Venerucci, \emph{Reciprocity laws for balanced diagonal classes}, no. 434, 2022, Heegner points, Stark-Heegner points, and diagonal classes, pp.~77--174. \MR{4489473}

\bibitem[B{\"u}y09a]{kazim2}
Kâzım B{\"u}yükboduk, \emph{Kolyvagin systems of {S}tark units}, J. Reine Angew. Math. \textbf{631} (2009), 85--107. \MR{2542218}

\bibitem[B{\"u}y09b]{kazim3}
\bysame, \emph{Stark units and the main conjectures for totally real fields}, Compos. Math. \textbf{145} (2009), no.~5, 1163--1195. \MR{2551993}

\bibitem[B{\"u}y09c]{kazim1}
\bysame, \emph{Tamagawa defect of {E}uler systems}, J. Number Theory \textbf{129} (2009), no.~2, 402--417. \MR{2473887}

\bibitem[B{\"u}y10]{kazim4}
\bysame, \emph{On {E}uler systems of rank {$r$} and their {K}olyvagin systems}, Indiana Univ. Math. J. \textbf{59} (2010), no.~4, 1277--1332. \MR{2815034}

\bibitem[B{\"u}y11a]{kazim5}
\bysame, \emph{{$\Lambda$}-adic {K}olyvagin systems}, Int. Math. Res. Not. IMRN (2011), no.~14, 3141--3206. \MR{2817676}

\bibitem[B{\"u}y11b]{kazim6}
\bysame, \emph{Stickelberger elements and {K}olyvagin systems}, Nagoya Math. J. \textbf{203} (2011), 123--173. \MR{2834252}

\bibitem[B{\"u}y16]{kazim_kolysys}
\bysame, \emph{Deformations of {K}olyvagin systems}, Ann. Math. Qu\'e. \textbf{40} (2016), no.~2, 251--302. \MR{3529183}

\bibitem[Das16]{Dasgupta16}
Samit Dasgupta, \emph{Factorization of {$p$}-adic {R}ankin {$L$}-series}, Invent. Math. \textbf{205} (2016), no.~1, 221--268. \MR{3514962}

\bibitem[Das24]{myself2}
Bhargab Das, \emph{$p$-adic functional equation and {B}lock-{K}ato type conjecture for twisted triple product {G}alois representations over a real quadratic field}, Preprint, 2024.

\bibitem[DF25]{Darmon_Fornea_2025}
Henri Darmon and Michele Fornea, \emph{Mock plectic points}, Journal of the Institute of Mathematics of Jussieu (2025), 1–25.

\bibitem[Dim13]{Dimitrov13}
Mladen Dimitrov, \emph{Automorphic symbols, {$p$}-adic {$L$}-functions and ordinary cohomology of {H}ilbert modular varieties}, Amer. J. Math. \textbf{135} (2013), no.~4, 1117--1155. \MR{3086071}

\bibitem[DR16]{Darmon16GenralisedKatoClasses}
Henri Darmon and Victor Rotger, \emph{Elliptic curves of rank two and generalised {K}ato classes}, Res. Math. Sci. \textbf{3} (2016), Paper No. 27, 32. \MR{3540085}

\bibitem[DR22]{DR22}
\bysame, \emph{{$p$}-adic families of diagonal cycles}, no. 434, 2022, Heegner points, Stark-Heegner points, and diagonal classes, pp.~29--75. \MR{4489472}

\bibitem[FG23]{fornea23plectic}
Michele Fornea and Lennart Gehrmann, \emph{Plectic {S}tark-{H}eegner points}, Adv. Math. \textbf{414} (2023), Paper No. 108861, 42. \MR{4536125}

\bibitem[FJ24]{ForneaJin}
Michele Fornea and Zhaorong Jin, \emph{Hirzebruch-{Z}agier classes and rational elliptic curves over quintic fields}, Math. Z. \textbf{308} (2024), no.~1, Paper No. 18, 78. \MR{4784192}

\bibitem[For19]{For19}
Michele Fornea, \emph{Growth of the analytic rank of modular elliptic curves over quintic extensions}, Math. Res. Lett. \textbf{26} (2019), no.~6, 1571--1586. \MR{4078688}

\bibitem[FP94]{FPR94}
Jean-Marc Fontaine and Bernadette {Perrin-Riou}, \emph{Autour des conjectures de {B}loch et {K}ato: cohomologie galoisienne et valeurs de fonctions {$L$}}, Motives ({S}eattle, {WA}, 1991), Proc. Sympos. Pure Math., vol. 55, Part 1, Amer. Math. Soc., Providence, RI, 1994, pp.~599--706. \MR{1265546}

\bibitem[Gre82]{Greenberg1982}
Ralph Greenberg, \emph{Iwasawa's theory and p-adic {L}-functions for imaginary quadratic fields}, pp.~275--285, Birkh{\"a}user Boston, Boston, MA, 1982.

\bibitem[Gro80]{Gross80}
Benedict~H. Gross, \emph{On the factorization of {$p$}-adic {$L$}-series}, Invent. Math. \textbf{57} (1980), no.~1, 83--95. \MR{564185}

\bibitem[Hid89a]{hida89reps}
Haruzo Hida, \emph{Nearly ordinary {H}ecke algebras and {G}alois representations of several variables}, Algebraic analysis, geometry, and number theory ({B}altimore, {MD}, 1988), Johns Hopkins Univ. Press, Baltimore, MD, 1989, pp.~115--134. \MR{1463699}

\bibitem[Hid89b]{hida89nearly}
\bysame, \emph{On nearly ordinary {H}ecke algebras for {${\rm GL}(2)$} over totally real fields}, Algebraic number theory, Adv. Stud. Pure Math., vol.~17, Academic Press, Boston, MA, 1989, pp.~139--169. \MR{1097614}

\bibitem[Hid91]{Hida91}
\bysame, \emph{On {$p$}-adic {$L$}-functions of {${\rm GL}(2)\times {\rm GL}(2)$} over totally real fields}, Ann. Inst. Fourier (Grenoble) \textbf{41} (1991), no.~2, 311--391. \MR{1137290}

\bibitem[HM22]{hernandez2022plectic}
Víctor Hernández and Santiago Molina, \emph{Plectic points and {H}ida-{R}ankin $p$-adic {$L$}-functions}, Preprint, {arXiv}:2202.12573, 2022.

\bibitem[KLZ17]{kings_loeffler_zerbes17}
Guido Kings, David Loeffler, and Sarah~Livia Zerbes, \emph{Rankin-{E}isenstein classes and explicit reciprocity laws}, Camb. J. Math. \textbf{5} (2017), no.~1, 1--122. \MR{3637653}

\bibitem[Liu16]{Liu16}
Yifeng Liu, \emph{Hirzebruch-{Z}agier cycles and twisted triple product {S}elmer groups}, Invent. Math. \textbf{205} (2016), no.~3, 693--780. \MR{3539925}

\bibitem[LZ20a]{loeffler2020iwasawatheoryquadratichilbert}
David Loeffler and Sarah~Livia Zerbes, \emph{Iwasawa theory for quadratic {H}ilbert modular forms}, Preprint, {arXiv}:2006.14491, 2020.

\bibitem[LZ20b]{LZ20local}
\bysame, \emph{Euler systems with local conditions}, Development of {I}wasawa theory---the centennial of {K}. {I}wasawa's birth, Adv. Stud. Pure Math., vol.~86, Math. Soc. Japan, Tokyo, [2020] \copyright 2020, pp.~1--26. \MR{4385077}

\bibitem[LZ24]{loeffler2024blochkatoconjecturegsp4}
\bysame, \emph{On the {B}loch-{K}ato conjecture for {${\rm GSp}(4)$}}, Preprint, {arXiv}:2003.05960, 2024.

\bibitem[Mol21]{Molina21}
Santiago Molina, \emph{Finite slope triple product {$p$}-adic {$L$}-functions over totally real number fields}, J. Number Theory \textbf{223} (2021), 267--306. \MR{4231538}

\bibitem[MR04]{MR_kolysys}
Barry Mazur and Karl Rubin, \emph{Kolyvagin systems}, Mem. Amer. Math. Soc. \textbf{168} (2004), no.~799, viii+96. \MR{2031496}

\bibitem[Nek06]{Nek06}
Jan Nekov\'a\v{r}, \emph{Selmer complexes}, Ast\'erisque, no. 310, Soci\'et\'e math\'ematique de France, 2006 (en). \MR{2333680}

\bibitem[NS16]{nekovar16plectic}
J.~Nekov\'a\v{r} and A.~J. Scholl, \emph{Introduction to plectic cohomology}, Advances in the theory of automorphic forms and their {$L$}-functions, Contemp. Math., vol. 664, Amer. Math. Soc., Providence, RI, 2016, pp.~321--337. \MR{3502988}

\bibitem[NSW08]{NSW08}
J\"{u}rgen Neukirch, Alexander Schmidt, and Kay Wingberg, \emph{Cohomology of number fields}, second ed., Grundlehren der mathematischen Wissenschaften [Fundamental Principles of Mathematical Sciences], vol. 323, Springer-Verlag, Berlin, 2008. \MR{2392026}

\bibitem[Pal18]{Palva18}
Bharathwaj Palvannan, \emph{On {S}elmer groups and factoring {$p$}-adic {$L$}-functions}, Int. Math. Res. Not. IMRN (2018), no.~24, 7483--7554. \MR{3919711}

\bibitem[{Per}93]{PR93}
Bernadette {Perrin-Riou}, \emph{Fonctions {$L\, p$}-adiques d{\textquoteright}une courbe elliptique et points rationnels}, Annales de l'Institut Fourier \textbf{43} (1993), no.~4, 945--995 (fr). \MR{1252935}

\bibitem[{Per}00]{perrin-riou_padic}
\bysame, \emph{{$p$}-adic {$L$}-functions and {$p$}-adic representations}, SMF/AMS Texts and Monographs, vol.~3, American Mathematical Society, Providence, RI; Soci\'et\'e{} Math\'ematique de France, Paris, 2000, Translated from the 1995 French original by Leila Schneps and revised by the author. \MR{1743508}

\bibitem[Wan15]{Wan}
Xin Wan, \emph{The {I}wasawa main conjecture for {H}ilbert modular forms}, Forum Math. Sigma \textbf{3} (2015), Paper No. e18, 95. \MR{3482263}

\bibitem[Wil88]{Wil88}
A.~Wiles, \emph{On ordinary {$\lambda$}-adic representations associated to modular forms}, Invent. Math. \textbf{94} (1988), no.~3, 529--573. \MR{969243}

\end{thebibliography}
		

	\end{document}